\documentclass[a4paper]{amsart}

\usepackage[utf8]{inputenc}
\usepackage[T1]{fontenc}

\usepackage[table]{xcolor}
\definecolor{blue}{RGB}{50, 120, 200}
\definecolor{red}{RGB}{175, 49, 39}
\definecolor{green}{RGB}{100, 255, 150}

\usepackage{latexsym, amsthm, amsmath, amssymb, mathrsfs, graphicx, enumitem, dsfont, thmtools, mathtools, comment, tikz-cd}

\usepackage[english]{babel}%
\usepackage[colorlinks, linktoc=page, linkcolor={blue}, citecolor={red}, urlcolor={green}]{hyperref}
\addto\extrasenglish{%
}

\calclayout

\DeclareRobustCommand{\SkipTocEntry}[5]{}

\newtheorem{theorem}{Theorem}[section]
\declaretheorem[style=plain,name=Theorem,numbered=no]{theorem*}
\declaretheorem[style=plain,name=Problem]{problem}
\declaretheorem[style=plain,name=Lemma,numberlike=theorem]{lemma}
\declaretheorem[style=plain,name=Proposition,numberlike=theorem]{proposition}
\declaretheorem[style=plain,name=Corollary,numberlike=theorem]{corollary}

\declaretheorem[style=plain,name=Conjecture]{conjecture}
\declaretheorem[style=definition,name=Definition,numberlike=theorem]{definition}
\declaretheorem[style=remark,name=Example,numberlike=theorem]{example}
\declaretheorem[style=remark,name=Remark,numberlike=theorem]{remark}
\declaretheorem[style=remark,name=Notation,numberlike=theorem]{notation}

\newcommand{\QQ}{\ensuremath{\mathbb Q}}
\newcommand{\CC}{\ensuremath{\mathbb C}}
\newcommand{\NN}{\ensuremath{\mathbb N}}
\newcommand{\ZZ}{\ensuremath{\mathbb Z}}
\newcommand{\FF}{\ensuremath{\mathbb F}}
\newcommand{\FFpk}{\ensuremath{\overline{\FF}_{p_k}}}

\newcommand{\KK}{\ensuremath{\mathbb K}}
\newcommand{\LL}{\ensuremath{\mathbb L}}

\newcommand{\im}{\ensuremath{\mathrm{Im}\,}}
\renewcommand{\char}{\ensuremath{\mathrm{char}\,}}

\newcommand{\Ad}{\mathrm{Ad}\,}

\newcommand{\gal}{\ensuremath{\mathrm{Gal}\,}}

\newcommand{\rev}{\ensuremath{\mathrm{rev}\,}}

\newcommand{\id}{\mathds 1}
\newcommand{\Id}{\ensuremath{\mathrm{Id}}}
\newcommand{\Tr}{\ensuremath{\mathrm{Tr}}}
\newcommand{\Hom}{\ensuremath{\mathrm{Hom}}}
\newcommand{\End}{\ensuremath{\mathrm{End}}}

\newcommand{\Aut}{\ensuremath{\mathrm{Aut}}}
\newcommand{\tAut}{\ensuremath{\mathrm{Aut}^{\otimes}}}
\newcommand{\Sym}{\ensuremath{\mathrm{Sym}}}

\newcommand{\Los}{\L o\'s's Theorem}

\newcommand{\CUP}{\ensuremath{\mathrm{cup}}}
\newcommand{\CAP}{\ensuremath{\mathrm{cap}}}

\newcommand{\uqsl}{{U}_{q}(\mathfrak{sl}_{2})} 
\newcommand{\uprod}{\prod_\mathscr{U}}  
\newcommand{\ulimit}[1]{\operatorname*{ulim}\limits_{{#1} \to \infty}}  
\newcommand{\ulim}[1]{\operatorname*{ulim}_{{#1} \to \infty}}  
\newcommand{\uqg}{{U}_{q}(\mathfrak{g})}
\newcommand{\Dr}[1]{\ensuremath{\mathcal{Z}\left(#1\right)}}
\newcommand{\Rep}[1]{\ensuremath{{\mathrm{\bf{Rep}}}\left(#1\right)}}
\newcommand{\Tilt}[2]{\ensuremath{{\mathrm{\bf{Tilt}}}_{#2}\left(#1\right)}}
\newcommand{\Vect}{\ensuremath{{\mathrm{\bf{Vec}}}}}
\newcommand{\sVect}{\ensuremath{{\mathrm{\bf{sVec}}}}}
\newcommand{\cat}{\ensuremath{\mathcal C}}

\newcommand{\obj}[1]{\ensuremath{\underline{\bf{#1}}}}

\newcommand{\TLf}[2]{\ensuremath{{\mathrm{\bf{TL}}}\left(#1;\, #2\right)}}
\newcommand{\preTLf}[2]{\ensuremath{{\mathrm{\widetilde{\bf{TL}}}}\left(#1;\, #2\right)}}
\newcommand{\Mug}[1]{\ensuremath{\mathcal{Z}_{\text{M\"ug}}\left({#1}\right)}}
\newcommand{\Cent}[2]{\ensuremath{\mathrm{Cent}_{#1}\left({#2}\right)}}

\newcommand{\VerF}[1]{\ensuremath{\mathcal{A}_{\varkappa-1}\left({#1}\right)}}
\newcommand{\Ver}{\ensuremath{\mathcal{A}_{\varkappa-1}}}
\newcommand{\Verp}{\mathbf{Ver}_p}

\newcommand{\YD}[1]{\ensuremath{\mathcal{YD}_{\mathcal O(#1)}^{\mathcal O(#1)}}}

\newcommand{\parag}[1]{\vspace{2mm}\noindent{\bf #1}\hspace{1mm}}

\usepackage{tikz}
\usetikzlibrary{matrix}
\usetikzlibrary{fit}
\usetikzlibrary{calc}
\usetikzlibrary{backgrounds}
\usetikzlibrary{arrows}
\usetikzlibrary{shapes,shapes.geometric,shapes.misc}
\tikzstyle{tikzfig}=[baseline=-0.25em,scale=0.5]
\pgfkeys{/tikz/tikzit fill/.initial=0}
\pgfkeys{/tikz/tikzit draw/.initial=0}
\pgfkeys{/tikz/tikzit shape/.initial=0}
\pgfkeys{/tikz/tikzit category/.initial=0}
\pgfdeclarelayer{edgelayer}
\pgfdeclarelayer{nodelayer}
\pgfsetlayers{background,edgelayer,nodelayer,main}
\tikzstyle{none}=[inner sep=0mm]

\tikzstyle{every loop}=[]

\tikzstyle{blob}=[fill=white, draw=black, shape=circle, inner sep=1.5pt]
\tikzstyle{fullblob}=[fill=black, draw=black, shape=circle, tikzit fill=black, tikzit draw=black, inner sep=1.5pt]
\tikzstyle{jonesrectangle}=[fill=white, draw=black, shape=rectangle, minimum width=1cm, minimum height=0.6cm]
\tikzstyle{thickstrand}=[-, thick]
\tikzstyle{arrow}=[->, thick]	
\tikzstyle{bluearrow}=[draw={rgb,255: red,0; green,0; blue,255}, thick, <-]
\tikzstyle{redarrow}=[draw={rgb,255: red,255; green,0; blue,0}, thick, <-]
\tikzstyle{blueline}=[draw={rgb,255: red,0; green,0; blue,255}, thick, -]
\tikzstyle{redline}=[draw={rgb,255: red,255; green,0; blue,0}, thick, -]
\tikzstyle{braid-over}=[-, draw=white, , double=black, double distance=0.8pt, tikzit draw={rgb,255: red,128; green,0; blue,128}]
\tikzstyle{background}=[-, dashed, draw={rgb,255: red,191; green,191; blue,191}, thin]

\begin{document}
	\title{The Drinfeld Center of the Generic Temperley--Lieb Category}
	\author{Moaaz Alqady}
	\address{M.A., Department of Mathematics, University of Oregon, Fenton Hall, Eugene, OR 97403, USA}
	\email{malqady@uoregon.edu}

	\begin{abstract}
		We show that the Temperley--Lieb category $\TLf{q}{\CC}$ embeds in an ultraproduct of modular tensor categories when $q$ is not a root of unity.
		As a result, we show that its Drinfeld center is semisimple and describe its simple objects.
		The canonical functor $$\TLf{q}{\CC}\boxtimes \TLf{q}{\CC}^\rev \boxtimes \Rep{\ZZ/2\ZZ} \to \Dr{\TLf{q}{\CC}},$$ induced by the braiding and the $\ZZ/2\ZZ$--grading on the Temperley--Lieb category, is thus shown to be a monoidal equivalence, which becomes a braided equivalence upon twisting the braiding by a certain bicharacter.
		Along the way, we formalize some general results about ultraproducts of tensor categories and tensor functors, building on earlier works of Crumley, Harman, and Flake--Harman--Laugwitz.
		We also discuss the center at some exceptional values of $q$.
	\end{abstract}

	\maketitle
	
	\setcounter{tocdepth}{1}

	\tableofcontents


\section{Introduction}\label{sec:intro}

Given any tensor category $\mathcal C$, Joyal and Street \cite{JS}, based on the work of Drinfeld \cite{Dr}, gave a general construction of a braided tensor category $\Dr{\cat}$, called its \emph{Drinfeld center}.
This allowed for the discovery of many interesting braided tensor categories, playing a central role in modern mathematics ranging from statistical mechanics to invariants of knots and manifolds to topological quantum field theories and topological quantum computation (for example, see \cite{T}, \cite{BK}, \cite{Wang}).
Despite the simplicity of its definition (see \autoref{DrCenter}), the Drinfeld center of a given tensor category is typically difficult to describe explicitly.
In the special case when $\cat$ is a modular tensor category, it was shown in \cite{DGNO} that the center has a particularly nice description. 
Interestingly, the Drinfeld centers of the representation categories of quantum groups $U_q(\mathfrak g)$ for a Lie algebra $\mathfrak g$ have not been described explicitly until now. 
It is the main goal of this paper to describe this center completely in the case when $\mathfrak g = \mathfrak{sl}_2$ and $q$ is generic, as well as to suggest an approach for exploring the center for other Lie types.
The category of type I representations of $U_q(\mathfrak{sl}_2)$ over a field $\KK$ is known to be equivalent to the Temperley--Lieb category $\TLf{q}{\KK}$, so we work with the latter diagrammatic version. 

In \cite{A} and \cite{AP}, Andersen and Paradowski described certain fusion categories associated with any quantum group at a given roots of unity $q$.
Namely, they considered a semisimple quotient of the category of tilting modules of the given quantum group at $q$.
These categories were also considered by Reshetikhin and Turaev in \cite{RT} in the context of 3-manifold invariants.
The fusion rules in the $\mathfrak g= \mathfrak{sl_2}$ case had already appeared in the work of Verlinde on conformal field theories \cite{V}.
In the case of the quantum group $\uqsl$, one can equivalently describe these categories as a semisimple quotient of $\TLf{q}{\KK}$ for a root of unity $q\in \KK$. 
It turns out the categories constructed in this way are modular tensor categories whenever $q$ is a primitive root of unity of order $2\varkappa$ for an integer $\varkappa\geq 3$. 
We denote these categories by $\VerF{\KK}$.

Our approach for understanding the center of the Temperley--Lieb category is to ``approximate'' $\TLf{q}{\CC}$ as a ``limit'' of the modular tensor categories $\Ver$. 
More precisely, we express $\TLf{q}{\CC}$ (for generic $q\in\CC$) as a subcategory of the ultraproduct of a family of categories $\mathcal A_{\varkappa_i-1}(\KK_i)$ where $\varkappa_i\to\infty$ and $\char \KK_i \to 0$. 
The use of ultraproducts as a way of taking limits of tensor categories was first introduced by Deligne in his seminal paper \cite{D2}; therein, he described the category $\Rep{S_t}$ for transcendental $t\in\CC$ as a subcategory of the ultraproduct of $\Rep{S_n}$ for integers $n\to \infty$. 
Later, Harman \cite{Har} extended this result for algebraic $t$ by varying both $n$ as well as the characteristic of the underlying field. 
This ultraproduct description was used in \cite{FHL} to describe the indecomposable objects of the center of Deligne's interpolation category $\Rep{S_t}$.
It was a rare instance in the subject of representation theory where knowledge about the finite characteristic case (which is typically more difficult) informs our knowledge about the characteristic $0$ case.
Our results show that the usefulness of this approach is not unique to Deligne's interpolation categories, but is a more widely applicable phenomenon.

First, we observe that the braiding and the $\ZZ/2\ZZ$--grading on the Temperley--Lieb category induce a canonical functor
$$Z:\TLf{q}{\CC}\boxtimes \TLf{q}{\CC}^\rev\boxtimes \Rep{\ZZ/2\ZZ}\to \Dr{\TLf{q}{\CC}}.$$

Then, we show that for every $q\in\CC^\times$ that is not a root of unity, we can identify $q$ with a sequence $q_i$ of roots of unity of even orders in the ultraproduct of some fields $\KK_i$.
In the case of transcendental $q$, it is enough to take $\KK_i=\CC$ and $q_i$ any sequence of roots of unity of unbounded degree.
When $q$ is algebraic, we use a generalization of Zsigmondy's Theorem to algebraic integers, due to Postnikova and Schinzel \cite{PS}, to find a sequence of prime numbers $p_i\to\infty$ and a sequence of roots of unity $q_i$ of even orders $2\varkappa_i\to\infty$ in $\KK_i=\overline\FF_{p_i}$ whose ultraproduct limit can be identified with $q$.
As a result, we are able to express $\TLf{q}{\CC}$ as a subcategory in the ultraproduct of $\mathcal A_{\varkappa_i-1}(\KK_i)$.

By relating the center of $\TLf{q}{\CC}$ to the center of the ultraproduct factors, we are able to extract information about $\Dr{\TLf{q}{\CC}}$ through our knowledge of $\Dr{\mathcal A_{\varkappa_i-1}(\KK_i)}$; the latter ones are easy to describe since $\mathcal A_{\varkappa_i-1}(\KK_i)$ are modular tensor categories.
Using this ultraproduct approach, we show that, for generic $q\in\CC$, the center $\Dr{\TLf{q}{\CC}}$ is semisimple and we describe its simple objects and fusion rules completely. 
As a result, our main theorem (\autoref{mainThm}) is that the canonical functor $Z$ is an equivalence of monoidal categories; upon twisting the braiding by a certain $(\ZZ/2\ZZ)^{\oplus 3}$ bicharacter, this functor becomes a braided equivalence.

We remark that many of the techniques we use should extend in a somewhat straightforward way to other Lie types; 
one can approximate the categories $\Rep{U_q(\mathfrak g)}$ as an ``ultralimit'' of semisimplifications of the categories of tilting modules over $U_q(\mathfrak g)$ at certain roots of unity. 
However, we leave this investigation for a later work.
One can also apply this ultraproduct technology to express the classical case of $q=1$ by using the Verlinde categories $\mathbf{Ver}_p$ as an approximating sequence. 
But since the centers of $\mathbf{Ver}_p$ are more difficult, this approach is not the most fruitful in the classical limit.
A geometric approach is more useful in this case; the Drinfeld center of the representation category of an algebraic group $G$ can be described as the category of $G$-equivariant coherent sheaves supported at finitely many points.
Setting $G=SL_2$, this gives a sheaf-theoretic description of the Temperley--Lieb center at $q=1$. 

For the case of $q=-1$, we show that the Temperley--Lieb category is related to the $q=1$ case by a cocycle twist.
We prove in \autoref{TwistedCenter} that the Drinfeld center of a category that is twisted by a cocycle is equivalent to a cocycle twist of the Drinfeld center of the original category.
As a result, we get an explicit description of the center in terms of finitely-supported equivariant coherent sheaves on $SL_2$ in this case as well.

Finally, we also discuss the center in the crystal Temperley--Lieb category in \autoref{sec:exceptionalQ}.
 In a previous joint work with Stroi\'nski \cite{AS}, we studied the crystal limit of $\TLf{q}{\KK}$ as $q\to 0$. 
We showed that the resulting diagrammatic category $\mathbf{CrysTL}$ is equivalent to the linearized category of $\mathfrak{sl}_2$--crystals. 
In a private communication, Etingof asked whether the Drinfeld center of $\mathbf{CrysTL}$ contains any interesting objects (other than the ones induced by the $\ZZ/2\ZZ$--grading).
We conjecture that the answer is no and provide some partial results towards proving this conjecture.

The article is organized as follows: In \autoref{sec:prelim}, we collect the needed constructions and background from the general theory of tensor categories.
In \autoref{sec:TL}, we give all the necessary background on the Temperley--Lieb category and its semisimplifications when $q$ is root of unity. 
The goal of \autoref{sec:TLultra} is to prove that the Temperley--Lieb category can be identified with a subcategory of the ultraproduct of a family of modular tensor categories $\mathcal A_{\varkappa_i-1}(\KK_i)$.
\autoref{sec:TLcenter} uses this identification to describe the Drinfeld center of the Temperley--Lieb category for generic $q$. 
We make some remarks and formulate some conjectures in \autoref{sec:exceptionalQ} on the centers of $\Rep{U_q(\mathfrak g)}$ for other lie algebras $\mathfrak g$ or for exceptional values of $q$.
Finally, \autoref{sec:ultraprod} is a quick self-contained exposition of ultraproducts and their applications to the theories of tensor categories and tensor functors. 
Despite their apparent usefulness, ultraproducts of tensor categories have only been formally documented in the literature (in the case of Tannakian categories) in \cite{C}. 
Building on the theory developed in \textit{loc. cit.}, we hope that \autoref{sec:ultraprod} will help address this gap in the literature.


\parag{Acknowledgements.}
This work was done as part of a PhD thesis, supervised by Victor Ostrik to whom the author is very grateful for many valuable discussions.
The author learned the model theoretic background in the appendix to this paper from many people in the American University in Cairo's mathematics community; in particular, conversations with Ibrahim Abdelaziz, as well as the lecture notes of Daoud Siniora were very helpful. 
Thanks are also due to Johannes Flake, Nate Harman, Robert Laugwitz, Mateusz Stroi\'nski, and Philip Thomas for interesting conversations.


\section{Preliminaries on Tensor Categories}\label{sec:prelim}

This section aims to review some standard constructions for monoidal categories. We follow the terminology of \cite{EGNO} for all notions that are not explicitly defined here.	Throughout this paper, all our fields are assumed to be algebraically closed.

	\subsection*{Drinfeld Center}
	Throughout this section, we let $\cat$ be a monoidal category with tensor product $\otimes$, associator $\alpha:(-\otimes-)\otimes- \overset{\sim}{\longrightarrow} -\otimes(-\otimes-)$, monoidal identity $\id$, and unitor isomorphisms $l:\id\otimes - \overset{\sim}{\longrightarrow} -$ and $r:-\otimes \id  \overset{\sim}{\longrightarrow} -$.
	
	\begin{definition}
		A \emph{half-braiding} for an object $ X $ in $\cat$ is a natural isomorphism $ \varphi:X\otimes- \implies -\otimes X $ such that for any two objects $Y$ and $Z$ in $\cat$, the following diagram commutes.
\[\begin{tikzcd}
	& {(Y\otimes Z)\otimes X} & {Y\otimes (Z\otimes X)} \\
	{X\otimes (Y\otimes Z)} &&& {Y\otimes(X\otimes Z)} \\
	& {(X\otimes Y)\otimes Z} & {(Y\otimes X)\otimes Z}
	\arrow["\alpha_{Y, Z, X}", from=1-2, to=1-3]
	\arrow["{\varphi_{Y\otimes Z}}", from=2-1, to=1-2]
	\arrow["{1_Y\otimes \varphi_Z}"', from=2-4, to=1-3]
	\arrow["{\alpha_{X,Y,Z}}", from=3-2, to=2-1]
	\arrow["{\varphi_Y \otimes 1_Z}"', from=3-2, to=3-3]
	\arrow["{\alpha_{Y, X,Z}}"', from=3-3, to=2-4]
\end{tikzcd}\]
	\end{definition}

	\begin{definition}\label{DrCenter}
		The \emph{Drinfeld Center} of $\cat$ is the category $\Dr{\cat}$ whose objects are pairs $(X,\varphi)$, where $\varphi$ is a half-braiding for $X$ and whose morphisms $f:( X, \varphi) \to ( X', \varphi')$ are maps $f:X\to X'$ in $\cat$ such that the following diagram commutes for every object $Y$ in $\cat$.
\[\begin{tikzcd}
	{X\otimes Y} & {Y\otimes X} \\
	{X'\otimes Y} & {Y\otimes X'}
	\arrow["{\varphi_Y}", from=1-1, to=1-2]
	\arrow["{f\otimes1_Y}"', from=1-1, to=2-1]
	\arrow["{1_Y\otimes X}", from=1-2, to=2-2]
	\arrow["{\varphi'_Y}"', from=2-1, to=2-2]
\end{tikzcd}\]

		It is monoidal via $$(X,\varphi)\otimes (X',\varphi') := \left(X\otimes X', \alpha_{-,X,X'}\circ \left(\varphi\otimes 1_{X'}\right)\circ \alpha^{-1}_{X,-,X'}\circ \left(1_X \otimes \varphi'\right)\circ \alpha_{X,X',-}\right),$$ with the associator inherited from $\cat$, and monoidal identity $(\id, r^{-1}\circ l)$, \cite[Section 7.13]{EGNO}. 
		It is braided with braiding $(X,\varphi)\otimes (X',\varphi')\to (X',\varphi')\otimes (X,\varphi)$ given by $\varphi_{X'}$, \cite[Proposition 8.5.1]{EGNO}.
		
	\end{definition}
	
	There is an obvious forgetful functor $U:\Dr{\cat}\to\cat$ which just remembers the underlying objects. 
	A braiding on $\cat$ is the same data as a section to $U$.
	Indeed, for a braided category $\cat$ with braiding $\sigma_{X,Y}:X\otimes Y\to Y\otimes X$, we have two canonical monoidal sections of $U$ given by $X\mapsto (X, \sigma_{X,-})$ and $X\mapsto (X,\sigma^{-1}_{-,X})$. 
	Together, they induce a braided tensor functor $D:\cat\boxtimes\cat^{\rev} \to \Dr{\cat}$, whenever $\cat\boxtimes \cat^\rev$ exists (see the subsection on Deligne's tensor product below); namely, $D$ is determined by $A\boxtimes B \mapsto (A\otimes B, (\sigma_{A,-}\otimes \Id_B)\circ (\Id_A\otimes \sigma_{-,B}^{-1}))$, where the implicit associators in the half-braiding have been suppressed for brevity \cite[Proposition 8.6.1]{EGNO}.
	
	\begin{definition}\label{def:factorizable}
		A braided tensor category $\cat$ is called \emph{factorizable} if the canonical functor $D$ is an equivalence.
	\end{definition}

One easy way to find half-braidings on the monoidal unit is to find a grading on the category.

\begin{definition}
	A \emph{$G$--grading} on a tensor category $\cat$ for a group $G$ is a decomposition $\cat = \bigoplus_{g\in G} \cat_g$ of $\cat$ into a direct sum of full $\KK$-linear subcategories $\cat_g$ such that $\id_\cat\in\cat_e$, and, whenever $X\in\cat_g$ and $Y\in \cat_h$, we have $X\otimes Y \in\cat_{gh}$.
	A $G$--grading on $\cat$ is called \emph{faithful} if $\cat_g\neq 0$ for every $g\in G$.
	\end{definition}
	
	\begin{proposition}\cite[Proposition 3.9]{GN}\label{grading-duality}
		Let $\cat$ be a tensor category with a faithful $G$ grading. Then, we have an injective group homomorphism from $G^\vee:=\Hom(G,\KK^\times)$ to the group of half-braidings (under $\otimes$) on $\id_\cat$ in $\cat$. When $G$ is abelian, this induces a functor $\Rep{G}\to \Dr{\cat}$.
	\end{proposition}
	
	\begin{proof}
		The data of a half braiding on $\id_\cat$ is equivalent to a natural tensor automorphism $\Phi\in\Aut_\otimes(\Id_\cat)$ of the identity functor on $\cat$; indeed, a half braiding is a natural isomorphism $\Id_\cat \cong \id_\cat \otimes - \to -\otimes \id_\cat \cong \Id_\cat$ and the hexagon axiom precisely encodes compatibility with the (trivial) monoidal structure of $\Id_\cat$.
		So the group of half-braidings on $\id_\cat$ under $\otimes$ is isomorphic to $\Aut_\otimes(\Id_\cat)$. 
	
	Given a character $\chi\in G^\vee$, we define $\Phi_\chi\in\Aut_\otimes(\Id_\cat)$ by $\Phi_\chi(X)=\chi(g) \;\Id_X$ for every $X\in\cat_g$. Thus, we get an injective group homomorphism $G^\vee \to  \Aut_\otimes(\Id_\cat)$. 
	When $G$ is abelian, $\Rep{G}$ is just the category of characters of $G$, so the last claim follows.
	\end{proof}
	
	In fact, in the case of fusion categories, the following stronger result is proven in \emph{loc. cit.}
	
	\begin{proposition}\cite[Proposition 3.9]{GN}\label{FusionGradingVsHalfBraidings}
		Let $\cat$ be a fusion category and $G$ its universal grading group. Let $G_{ab}$ be the maximal abelian quotient of $G$.
		There is an isomorphism $G_{ab}^\vee \cong \Aut_\otimes(\Id_\cat)$. Hence, the group of half-braidings on $\id_\cat$ (under $\otimes$) is isomorphic to $G^\vee_{ab}$. 
	\end{proposition}

\begin{remark}
	More generally, $n$-dimensional representations of the universal grading group classify half-braidings on $\id^{\oplus n}$. For more details, \cite[Theorem 3.5]{GN} and the subsequent remarks.
\end{remark}


\subsection*{M\"uger Center}

	A different notion of the center of a braided monoidal category $\cat$ is the \emph{M\"uger center} or the \emph{symmetric center}, which was first considered in \cite{Mug}.
	
	\begin{definition}
		Given objects $X$ and $Y$ of a braided monoidal category $C$ with braiding $\sigma$, we say $X$ and $Y$ \emph{centralize one another} if $\sigma_{Y,X}\circ \sigma_{X,Y} = \Id_{X\otimes Y}$.
		For a given full subcategory $\mathcal D \subset \cat$, its \emph{centralizer} in $\cat$ is the full subcategory $\Cent{\cat}{\mathcal D} \subset \cat$ whose objects centralize those of $\mathcal D$, i.e. $\Cent{\cat}{\mathcal D}$ is the full subcategory of $\cat$ of all objects $X$ such that for each $Y\in \mathcal D$ we have $\sigma_{Y,X}\circ \sigma_{X,Y} = \Id_{X\otimes Y}$. 
		Two subcategories $\mathcal D$ and $\mathcal D'$ of $\cat$ are said to centralize one another if $\mathcal D\subset\Cent{\cat}{\mathcal D'}$ and $\mathcal D'\subset\Cent{\cat}{\mathcal D}$
	\end{definition}
	
	\begin{definition}
		\emph{The M\"uger Center} of $\cat$ is the centralizer $\Cent{\cat}{\cat}$ of $\cat$ itself, which we also denote by $\Mug{\cat}$.
	\end{definition}

	It is easy to verify that $\Mug{\cat}$ is a symmetric monoidal subcategory of $\cat$. 
	The importance of this construction for us is that it allows us to have some control over the Drinfeld Center in the case where $\cat$ is a fusion category. 
	
	Recall that for a braided fusion category equipped with a spherical structure, one may compute its modular invariant defined as follows: for simple objects $X$ and $Y$, set $s_{X,Y} = \Tr(\sigma_{Y,X} \circ \sigma_{X,Y})$) and $t_{X,Y} = \delta_{X,Y} \theta_X$, where $\delta_{X,Y}$ is the Kronecker delta and $\theta_X$ is the value of the twist on $X$; pictorially, 
	$$\begin{tikzpicture}[scale = 0.8]
	\begin{pgfonlayer}{nodelayer}
		\node [style=none] (0) at (-2, 0) {};
		\node [style=none] (1) at (-1, 0) {};
		\node [style=none] (2) at (0, 0) {};
		\node [style=none] (3) at (1, 0) {};
		\node [style=none] (4) at (-4, 0) {$s_{X,Y} := $};
		\node [style=none] (6) at (-2, 1) {$X$};
		\node [style=none] (7) at (1, 1) {$Y$};
		\node [style=none] (4) at (2, 0) {,};
		\node [style=none] (8) at (6, 0) {$t_{X}$};
		\node [style=none] (9) at (7, -1.4) {};
		\node [style=none] (10) at (7, 1.4) {};
		\node [style=none] (12) at (7.5, 1) {$X$};
		\node [style=none] (13) at (8.5, 0) {$:=$};
		\node [style=none] (14) at (10, -1.4) {};
		\node [style=none] (17) at (10, -0.25) {};
		\node [style=none] (19) at (10, 1.4) {};
		\node [style=none] (20) at (9.25, 0) {};
		\node [style=none] (22) at (10, 0.25) {};
		\node [style=none] (23) at (10.5, 1) {$X$};
	\end{pgfonlayer}
	\begin{pgfonlayer}{edgelayer}
		\draw [style=thickstrand, bend left=90, looseness=1.75] (0.center) to (2.center);
		\draw [style=thickstrand, bend right=90, looseness=1.75] (1.center) to (3.center);
		\draw [style=braid-over, bend right=90, looseness=1.75] (0.center) to (2.center);
		\draw [style=braid-over, bend left=90, looseness=1.75] (1.center) to (3.center);
		\draw [style=thickstrand] (10.center) to (9.center);
		\draw [style=thickstrand, in=270, out=90] (14.center) to (17.center);
		\draw [style=thickstrand] (22.center) to (19.center);
		\draw [style=thickstrand, in=90, out=90, looseness=1.75] (20.center) to (17.center);
		\draw [style=braid-over, in=-90, out=-90, looseness=1.50] (22.center) to (20.center);
	\end{pgfonlayer}
\end{tikzpicture}
$$
	Thus, we have $|\mathcal O(\cat)|\times |\mathcal O(\cat)|$ matrices $S= (s_{X,Y})_{X,Y\in\mathcal O(\cat)}$ and $T=(\delta_{X,Y} t_{X})_{X,Y\in\mathcal O(\cat)}$, where $\mathcal O(\cat)$ is the set of isomorphism classes of simple objects in $\cat$ (each identified with an arbitrary representative therein), and $\delta_{X,Y}$ is the Kronecker delta.
	These are called the $S$-matrix and the $T$-matrix of $\cat$, respectively.
	\begin{definition}
		A braided fusion category equipped with a spherical structure is called \emph{modular} if its $S$-matrix is non-degenerate.
	\end{definition}
	
	\begin{proposition}\cite[Proposition 3.7]{DGNO}\label{modularityCriteria}
		Let $\cat$ be a braided fusion category equipped with a spherical structure. Then, the following conditions are equivalent:
		\begin{itemize}
			\item $\cat$ is modular;
			\item $\Mug{\cat} \simeq \Vect$;
			\item $\cat$ is factorizable.
		\end{itemize}
	\end{proposition}
	There are stronger variants of this theorem for non-spherical fusion categories and even for non-semisimple braided finite monoidal category \cite{Shim}; however, for our purposes, this weaker version shall suffice.
	As an immediate corollary, we have the following result:
	
	\begin{corollary}\label{modularCenter}
		For a modular tensor category $\cat$, the canonical functor $D:\cat\boxtimes \cat^{\rev}\overset{\sim}{\longrightarrow}\Dr{\cat}$ is an equivalence of braided tensor categories.
	\end{corollary}
	
	\subsection*{Cauchy Completion}
	
		A general construction of completion exists for any pre-additive category $\mathcal A$ to an additive Karoubian category $A^c$ (i.e. an additive category where every idempotent splits), called its Cauchy completion, which we now describe.
	First, we form the additive closure of $\mathcal A$, whose objects are formal direct sums of objects in $\mathcal A$ and whose morphisms are formal matrices whose columns and rows are indexed by the summands of the domain and codomain and whose entries are given by a morphism in $\mathcal A$ between the corresponding summands; let us denote this additive envelope of $\mathcal A$ by $\mathcal A^{\oplus}$. 
	Note that if $\mathcal A$ is monoidal, then so is $\mathcal A^{\oplus}$ via the obvious formulae for distributivity of $\otimes$ over $\oplus$.
	Next, if $\mathcal B$ is any additive category, we may consider its Karoubian envelope $\mathcal B^{\mathrm{kar}}$ which formally adjoins idempotents; it is defined by 
	$$\mathrm{Ob}(\mathcal B^{\mathrm{kar}}) = \{X_e:=(X,e) \,|\, X\in \mathrm{Ob}(X), \, e\in \End_{\mathcal B}(X) \text{ with } e^2=e\},  \text{ and}$$
	$$\Hom_{\mathcal B^{\mathrm{kar}}}(X_e, Y_f) = f \,\Hom_{\mathcal B}(X,Y)\, e \subset \Hom_{\mathcal B}(X,Y).$$
	Again, if $\mathcal B$ is monoidal, then so is $\mathcal B^{\mathrm{kar}}$ via $X_e\otimes Y_f = (X\otimes Y)_{e\otimes f}$.
	Then, the \emph{Cauchy completion} of a pre-additive category $\mathcal A$ is defined as $A^c := \left(\mathcal A^{\oplus}\right)^{\mathrm{kar}}$, which is again monoidal if $\mathcal A$ is.
	It is also easy to check that if $\mathcal A$ admits braided or spherical structures, then they can be extended to $\mathcal A^c$ in the obvious way. 
	
		\begin{proposition}\cite[Corollary 4.4]{K}\label{Krull-Schmidt}
		A Karoubian $\KK$-linear category with finite dimensional morphism spaces is \emph{Krull-Schmidt}, that is every object decomposes into a unique (up to order and isomorphism) direct sum of finitely many indecomposable objects, each of which having a local endomorphism ring. 
	\end{proposition}

	\begin{definition}\label{tensorCat}
	Let us call $\cat$ a \emph{tensor category} if it is a rigid monoidal idempotent-complete $\KK$-linear category  such that morphism spaces are finite dimensional and $\End(\id)$ is $1$-dimensional as vector spaces over $\KK$.
\end{definition}

\begin{remark}
The notion of a tensor category we use is weaker than the more standard definition of \cite{EGNO} as we have weakened the condition of being abelian to being merely Karoubian, which suits our purposes better.	
\end{remark}
	
	The operations of Cauchy completion and Drinfeld center intertwine in the following sense:
	
	\begin{proposition}\cite[Proposition A.1]{FL}\label{Kar-Dr}
		The embedding of $\Dr{\cat}$ as a full subcategory in $\Dr{\cat^c}$ given by $(X,\varphi)\mapsto(X,\varphi')$, where $\varphi':X\otimes - \implies - \otimes X$ extends $\varphi$ via
		$$\varphi'_{Y_e} = (e\otimes \Id_X) \circ  \varphi_Y \circ (\Id_X \otimes e) = (e\otimes \Id_X) \circ  \varphi_Y  = \varphi_Y \circ (\Id_X \otimes e)$$
		induces a braided monoidal fully faithful functor $\Dr{\cat}^c\to \Dr{\cat^c}$.
	\end{proposition}

\subsection*{Semisimplicity and Semisimplification} 
Recall that a tensor category is \emph{semisimple} if it is abelian and every object decomposes as a direct sum of simple objects. 
For sufficiently nice non-semisimple tensor categories, a well-known semisimplification procedure exists (which was first considered in \cite{RT} and later studied in abstractly in \cite{BW}).
Recall that if $\cat$ is a spherical tensor category with spherical structure $\alpha_X:X\to X^{**}$, the quantum trace of $f:X\to X$ is defined as 
	$$\Tr_{\cat}(f): \id \xrightarrow{\textrm{coev}_X} X\otimes X^* \xrightarrow{f\otimes \Id_X} X\otimes X^*\xrightarrow{\alpha_{X}\otimes X} X^{**}\otimes X^* \xrightarrow{\mathrm{ev}_X} \id.$$
	The \emph{quantum dimension} of an object $X$ is $\dim_q(X):= \Tr(\Id_X)$.
	
	\begin{definition}
		A morphism $f:X\to Y$ in a spherical category $\cat$ is called \emph{negligible} if for every morphism $g:Y\to X$, we have $\Tr(g\circ f) = 0$. An object $X$ is said to be \emph{negligible} if $\Id_X$ is a negligible morphism.
	\end{definition}

\begin{definition}
	A spherical tensor category $\cat$ is \emph{tractable} if $\Tr(f) = 0$ for any nilpotent endomorphism in $\cat$. 
\end{definition}

The following result is well-known, with a full proof available in \cite[Proposition 2.4]{EO2}.

\begin{proposition}\label{abelianTractable}
	An abelian spherical tensor category is tractable.
\end{proposition}

Note that a spherical tensor category that contains a negligible morphism cannot possibly be semisimple, since a negligible morphism between indecomposables must lie in the Jacobson radical.
indeed, if $f:X\to Y$ is such a morphism, then, for any $g:Y\to X$, we get $g\circ f$ being not invertible, for otherwise $X$ would be a summand of $Y$. 
It follows then that $g\circ f$ must lie in the Jacobson radical of $\End(X)$ since it is outside its \emph{unique} maximal ideal ($X$ being indecomposable implies that $\End(X)$ is a local ring).

Recall that a tensor ideal $\mathcal I$ in a tensor category $\cat$ is a collection of subspaces $\mathcal I(X,Y)\subset \Hom(X,Y)$ for all objects $X$ and $Y$ in $\cat$ such that if $f$ lies in one of these subspaces, then so does $f\circ g$, $g\circ g$, $f\otimes g$, and $g\otimes f$ for all morphisms $g$ in $\cat$.

\begin{lemma}\label{functorTractable}
		Suppose $\cat$ is a spherical category admitting a functor $F:\cat\to\mathcal A$ to a tractable category $\mathcal A$; and suppose further that $F$ intertwines the spherical structures of $\cat$ and $\mathcal A$.
		Then, $\cat$ is also tractable.
\end{lemma}

\begin{proof}
	Given a nilpotent endomorphism $f$ in $\cat$, $\Tr_\cat(f) = \Tr_{\mathcal A}(Ff) = 0$ since $Ff$ is nilpotent. 
\end{proof}

\begin{example}
The collection of all negligible morphisms in a spherical tensor category forms a tensor ideal, $\mathcal N$, called the negligible ideal.	
\end{example}

We can form a quotient $\cat/\mathcal I$ of a tensor category $\cat$ by a tensor ideal $\mathcal I$ defined to be the category with the same objects as $\cat$ and 
$$\Hom_{\cat/\mathcal I}(X,Y) := \Hom_\cat(X,Y)/\mathcal I(X,Y)$$

It turns out, for tractable tensor categories, the existence of negligible morphisms is the only obstruction to semisimplicity.

\begin{theorem}\cite[Theorem 2.6]{EO2}\label{tractable-semisimplification}
Let $\cat$ be a tractable tensor category over an algebraically closed field, and let $\mathcal N$ be its negligible ideal. The category $\cat/\mathcal N$ is semisimple and its simple objects are the indecomposables of $\cat$ with non-zero dimension.
\end{theorem}

\begin{definition}
	The \emph{semisimplification} of a tractable tensor category $\cat$ to be $\overline\cat:=\cat/\mathcal N$, i.e. its quotient by its negligible ideal.
\end{definition}

Let $m_\cat(X,Y)$ denote the summand multiplicity of in indecomposable object $X$ in some object $Y$ of $\cat$.

	\begin{lemma}\label{multiplicitiesLemma}
	If $X, Y, $ and $Z$ are indecomposable objects of nonzero dimension in a tractable $\KK$-linear tensor category with $\KK$ an algebraically closed field, then $m_\cat(X^*, Y\otimes Z) = m_\cat(Z^*, X\otimes Y)$. 
	\end{lemma}
	
	\begin{proof}
			First, we observe that the claim holds when $\cat$ is semisimple, by Schur's Lemma. Indeed, passing to the semisimplification, we see that 
				\[
				m_{\overline\cat}(\overline{X}^*, \overline{Y}\otimes\overline{Z}) = \dim\Hom_{\overline\cat}(\overline{X}^*, \overline{Y}\otimes\overline{Z}) =\dim\Hom_{\overline\cat}(\overline{Z}^*, \overline{X}\otimes\overline{Y}) = m_{\overline\cat}(\overline{Z}^*, \overline{X}\otimes\overline{Y}),
				\]
It remains to argue that multiplicities did not change upon passing to the semisimplification. Indeed, a summand of $X^*$ in $Y\otimes Z$ corresponds to morphisms $a:X^*\to Y\otimes Z$ and $b:Y\otimes Z\to X^*$ in $\cat$ satisfying $ba = \Id_{X^*}$. 
Passing to the semisimplification, we still have $\overline{b}\overline{a} = \Id_{\overline{X}^*}$ since none of these morphisms are negligible (otherwise $\dim(X)=0$).  
Conversely, if $a:\overline{X}^*\to \overline{Y}\otimes \overline{Z}$ and $b:\overline{Y}\otimes \overline{Z}\to \overline{X}^*$ in $\overline{\cat}$ satisfy $ba = \Id_{\overline{X}^*}$, we can lift them to some $\tilde{a}$ and $\tilde{b}$ in $\cat$ satisfying $\tilde b\tilde a = \Id_{X^*} + g$, where $g$ is negligible. But if $\tilde b\tilde a$ non-invertible, then it lies in the Jacobson radical of the local ring $\End_\cat(X^*)$, and hence negligible. That would imply that $Id_{X^*}$ is negligible, so that $\dim X = 0$. So $\tilde b \tilde a$ is invertible and $X^*$ is a summand of $Y\otimes Z$.
	\end{proof}

	\subsection*{Deligne's Tensor Product of Tensor Categories}
	Given two finite abelian $\KK$-linear categories ${\mathcal A}$ and $\mathcal B$ over a perfect field $\KK$, Deligne \cite[Section 5]{D1} showed the existence of an abelian $\KK$-linear category ${\mathcal A}\boxtimes \mathcal B$ and a bifunctor $\boxtimes: \mathcal A\times \mathcal B \to {\mathcal A}\boxtimes \mathcal B$ which is right-exact in both ${\mathcal A}$ and $\mathcal B$ and is universal with this property; i.e. for any bifunctor $F:\mathcal A\times\mathcal B\to \mathcal C$ that is right-exact in both ${\mathcal A}$ and $\mathcal B$, there exists a unique right-exact functor $\tilde F: \mathcal A\boxtimes \mathcal B \to \mathcal C$ satisfying $F = \tilde F \circ \boxtimes$. 
	Later, López Franco \cite{ILF} extended Deligne's construction to all finitely cocomplete (for example, semisimple) abelian $\KK$-linear categories by showing that Deligne's construction coincides with Kelly's tensor product \cite{Kel}.
	
	When the categories ${\mathcal A}$ and $\mathcal B$ are both semisimple, their Deligne tensor product ${\mathcal A}\boxtimes \mathcal B$ is again semisimple. 
	The simple objects of ${\mathcal A}\boxtimes \mathcal B$ are exactly those of the form $X\boxtimes Y$ where $X$ and $Y$ are simple objects in ${\mathcal A}$ and $\mathcal B$, respectively. 
	When the categories ${\mathcal A}$ and $\mathcal B$ are monoidal with associator $\alpha^{\mathcal A}$ and $\alpha^{\mathcal B}$, their product ${\mathcal A}\boxtimes \mathcal B$ has a natural monoidal structure given by
	$$(X\boxtimes Y)\otimes_{{\mathcal A}\boxtimes \mathcal B}(X'\boxtimes Y') = (X\otimes_{\mathcal A} Y) \boxtimes (X'\otimes_{\mathcal B} Y')$$
	and similarly on morphisms, with associators $$\alpha^{{\mathcal A}\boxtimes \mathcal B}_{X\boxtimes Y, X'\boxtimes Y', X''\boxtimes Y''} = \alpha^{{\mathcal A}}_{X, X', X''}  \boxtimes \alpha^{\mathcal B}_{Y, Y', Y''}.$$ 
	Moreover, when $F: \mathcal A\times \mathcal B\to \cat$ has a monoidal structure, the canonical functor $\tilde F:\mathcal A\boxtimes \mathcal B\to \cat$ admits an obvious monoidal structure.
	Similarly, if $\sigma^{\mathcal A}$ and $\sigma^{\mathcal B}$ are braidings for ${\mathcal A}$ and $\mathcal B$, then ${\mathcal A}\boxtimes \mathcal B$ is braided via $\sigma^{{\mathcal A}\boxtimes \mathcal B} = \sigma^{\mathcal A}\boxtimes \sigma^{\mathcal B}$. 
	Duals and spherical structures exist for ${\mathcal A}\boxtimes \mathcal B$ whenever they do for ${\mathcal A}$ and $\mathcal B$ and are again defined component-wise.

	\begin{remark}
		In the non-semisimple case, the construction is more subtle, but we will only use the semisimple case in this paper. See \cite[Section 5]{D1} and \cite{ILF} for more general cases.
	\end{remark}

	\subsection*{Grading Twists} Suppose $\cat=\bigoplus_{g\in G} \cat_g$ is a $G$-graded monoidal category with a braiding $\sigma$, for some group $G$ (which is necessarily abelian since $\cat$ is braided). 
	Let $\gamma:G\times G \to \KK^\times$ be a bicharacter of $G$, i.e. a function satisfying 
	$$\gamma(g_1g_2, g_3) = \gamma(g_1,g_3)\gamma(g_2,g_3) \quad \text{and}\quad \gamma(g_1,g_2g_3) = \gamma(g_1,g_2)\gamma(g_1,g_3),$$
	for each $g_1,g_2,g_3\in G$.
	One can introduce new braided monoidal category $\cat^\gamma$, which is the same as $\cat$ as a monoidal category, but where the braiding $\sigma$ is twisted by $\gamma$ as follows: $\sigma^\gamma_{X_1,X_2} = \gamma(g_1, g_2) \sigma_{X_1,X_2}$ whenever $X_i\in \cat_{g_i}$ for $i=1,2$. 
	It is easy to check that this will still satisfy the hexagon axioms for a braiding.
	\begin{example}
		The category $\Vect_{\ZZ/2\ZZ}^\gamma$ of $\ZZ/2\ZZ$-graded vector spaces where the braiding is twisted by the nontrivial bicharacter of $\ZZ/2\ZZ$ given by $\gamma(g,h) = (-1)^{gh}$ is braided equivalent to the category of super vector spaces. 
	\end{example}
	
	The associator $\alpha$ of a $G$-graded monoidal category $\cat$ may be twisted by a $3$-cocycle $\omega\in Z^3(G,\KK^\times)$, i.e. a function $\omega:G\times G\times G\to\KK^\times$ satisfying 
	$$\omega(g_1,g_2,g_3)\omega(g_1,g_2g_3,g_4)\omega(g_2,g_3,g_4) = \omega(g_1g_2,g_3,g_4)\omega(g_1,g_2,g_3g_4),$$
	for all $g_1,g_2,g_3, g_4\in G$. This gives a new monoidal category, denoted $\cat_\omega$ whose associator is given by $\alpha^\omega_{X_1,X_2,X_3} = \omega(g_1,g_2,g_3)\alpha_{X_1,X_2,X_3}$ where $X_i\in \cat_{g_i}$ for $i=1,2,3$.
	
	More generally, one may twist both the associator and the braiding of $\cat$ simultaneously by an abelian 3-cocyle $(\omega, \gamma)\in Z_{ab}^3(G, \KK^{\times})$, i.e. functions
	$$\omega: G\times G\times G\to \KK^\times \quad \text{and}\quad \gamma:G\times G\to \KK^\times$$
	satisfying 
	$$\omega(g_1,g_2,g_3)\omega(g_1,g_2g_3,g_4)\omega(g_2,g_3,g_4) = \omega(g_1g_2,g_3,g_4)\omega(g_1,g_2,g_3g_4),$$
	$$\omega(g_2,g_3,g_1)\gamma(g_1,g_2g_3)\omega(g_1,g_2,g_3)=\gamma(g_1,g_3)\omega(g_2,g_1,g_3)\gamma(g_1,g_2), \text{and}$$
	$$\omega(g_3,g_1,g_2)^{-1}\gamma(g_1g_2,g_3)\omega(g_1,g_2,g_3)^{-1} = \gamma(g_1,g_3)\omega(g_1,g_3,g_2)^{-1}\gamma(g_2,g_3)$$
	for all $g_1,g_2,g_3,g_4\in G$.
	Given such a cocycle $(\omega,\gamma)$ and a braided monoidal category $\cat$ with associator $\alpha$ and braiding $\sigma$, one defines the twisted associator and braiding via $$\alpha^\omega_{X_1,X_2, X_3}:= \omega(g_1,g_2,g_3)\alpha_{X_1,X_2,X_3} \quad \text{and}\quad \sigma_{X_1, X_2}^\gamma := \gamma(g_1,g_2)\sigma_{X_1,X_2}$$
	when $X_i\in \cat_{g_i}$ for $i=1,2,3$. We denote the resulting category $\cat^{(\omega,\gamma)}$. Note that $\cat^{(1,\gamma)}=\cat^\gamma$. 
	
	This construction was given by Joyal and Street \cite{JS} in the case of pointed categories.
	A generalization of this twisting procedure is the so-called zesting construction, \cite{Zest}.
	
	Observe that if $\cat$ is $G$-graded, its Drinfeld center $\Dr{\cat}$ inherits a $G$ grading via its forgetful functor to $\cat$, i.e. the underlying object in $\cat$ determines the graded component of an object in $\Dr{\cat}$.
	Grading twists are compatible with taking Drinfeld centers in the following sense.
	
	\begin{proposition}\label{TwistedCenter}
		Suppose $\cat$ a monoidal category with a $G$-grading and $(\omega,\gamma)\in Z_{\mathrm{ab}}^3(G,\KK^\times)$ is a $3$-cocycle on $G$.
		There is a braided tensor equivalence $F:\Dr{\cat_\omega}\overset{\sim}{\longrightarrow}\Dr{\cat}^{(\omega,\gamma)}$ given on objects by $(X,\varphi)\mapsto (X,\varphi^\gamma)$, where $\varphi^\gamma_Y = \gamma(g,h)\cdot\varphi_X$ for $X\in\cat_g$ and $Y\in\cat_h$.
	\end{proposition}
	
	\begin{proof}
		Define $F$ on objects as above and on morphisms as the identity.
		Since there are no nonzero morphisms between objects in different graded components, a nonzero morphism $f:X\to Y$ in $\cat$ that is compatible with the half-braidings of $X$ and $Y$ remains so upon twisting the half-braidings (half-braidings on $X$ and on $Y$ are rescaled by the same scalars).
		So, $F$ is well-defined on morphisms.
		It is clearly an equivalence since the twisting operation is invertible.
		
		Observe that for $X\in\cat_g$ and $Y\in\cat_h$, we have
		
		\begin{align*}
			F((X,\varphi)\otimes (Y,\psi)) &= F\left(X\otimes Y, \alpha_{-,X,Y}\circ \left(\varphi\otimes 1_{Y}\right)\circ \alpha^{-1}_{X,-,Y}\circ \left(1_X \otimes \psi\right)\circ \alpha_{X,Y,-}\right) \\
			&= \left(X\otimes Y, \gamma(gh, -)\cdot \left(\alpha_{-,X,Y}\circ \left(\varphi\otimes 1_{Y}\right)\circ \alpha^{-1}_{X,-,Y}\circ \left(1_X \otimes \psi\right)\circ \alpha_{X,Y,-}\right)\right).
		\end{align*}
		On the other hand,
		\begin{align*}
			F(X,\varphi)\otimes F(Y,\psi) &= (X, \gamma(g,-)\cdot \varphi) \otimes (Y, \gamma(h,-)\cdot \psi) \\
			&= \left(X\otimes Y, \alpha^\omega_{-,X,Y}\circ \left(\varphi^\gamma\otimes 1_{Y}\right)\circ (\alpha^\omega)^{-1}_{X,-,Y}\circ \left(1_X \otimes \psi^\gamma\right)\circ \alpha^\omega_{X,Y,-}\right)\\
			&= \left(X\otimes Y, \lambda_{g,h,-} \cdot \alpha_{-,X,Y}\circ \left(\varphi\otimes 1_{Y}\right)\circ \alpha^{-1}_{X,-,Y}\circ \left(1_X \otimes \psi\right)\circ \alpha_{X,Y,-}\right),
		\end{align*}
		where $\lambda_{g,h,-} = \omega(-,g,h)\gamma(g,-)\omega^{-1}(g,-,y)\gamma(h,-)\omega(g,h,-)$ by definition of $\otimes$ in $\Dr{\cat}^{(\omega,\gamma)}$.
		But,  by the definition of an abelian $3$-cocycle, we have $\lambda_{g,h,-}  = \gamma(gh, -)$ so that the two half-braidings on $X\otimes Y$ above coincide.
		Therefore, $\Id_{X\otimes Y}:F(X,\varphi)\otimes F(Y,\psi) \to F((X,\varphi)\otimes (Y,\psi))$ is a well-defined morphism in $\Dr{\cat}^{(\omega,\gamma)}$.
		It is easy to see that it equips $F$ with the structure of a monoidal functor since the associators on both its domain and codomain are given by $\alpha^\omega$. 
		The fact that $F$ is braided is also clear: the braiding $(X,\varphi)\otimes (Y,\psi)\to (Y,\psi)\otimes (X,\varphi)$ in $\Dr{\cat}$ is given by $\varphi_Y$, so it becomes $\varphi_Y^\gamma$ in $\Dr{\cat}^{(\omega, \gamma)}$; this is the same as  
		the braiding $(X,\varphi^\gamma)\otimes (Y,\psi^\gamma)\to (Y,\psi^\gamma)\otimes (X,\varphi^\gamma)$ in $\Dr{\cat^{(\omega, \gamma)}}$.
	\end{proof}
	
	\begin{remark}
		If $G$ is of odd order and $(\omega,\gamma)$ is an abelian $3$-cocycles on $G$, then $[\omega]=0\in H^3(G,\KK^\times)$, so that we have a monoidal equivalence $\cat_\omega\simeq \cat$. This follows from the isomorphism $H_{\textrm{ab}}^3(G,\KK^\times)\cong \textrm{Quad}(G)$ between abelian cohomology of $G$ and quadratic forms on $G$ which was proved by Eilenberg and MacLane, \cite[Theorem 26.1]{EM} or see \cite[Section 8.4]{EGNO} for a modern exposition. Thus, \autoref{TwistedCenter} is only really useful when $G$ is a $2$-group. 
	\end{remark}


\section{The Temperley--Lieb Category}\label{sec:TL}

The results in this section are not new; a subset of them exists in the literature (though hard to find in some cases), a subset is only written in special cases (e.g. $\KK=\CC$) or with different terminology (e.g. subfactors, recoupling, etc.), and others are folklore, but do not exist in written form (at least to the author's knowledge). 
We include here a full and unified treatment of all the results we need with complete proofs for the reader's convenience, but the expert reader is free to skip/skim through this section. 
We try to include primary (or at least early) references whenever possible, and apologize for any omissions or oversights. 
One early exposition written in the language of tensor categories is \cite[Chapter XII]{T}, where the Temperley--Lieb category is called the skein category. 

		\subsection*{Definition and Universal Property}
	Recall that the $q$-integers are defined, for $n\in\ZZ$, as $$\displaystyle [n]_q = \frac{q^n-q^{-n}}{q-q^{-1}} = q^{n-1}+q^{n-3}+\dots + q^{1-n}.$$
	\begin{definition}\label{TL-def}
		Let $\preTLf{q}{\KK}$ be the (small) strict $\KK[q,q^{-1}]$-linear monoidal category defined as follows:
		\begin{itemize}
			\item $\mathrm{Ob}(\preTLf{q}{\KK}) = \NN = \{\obj{0}, \obj{1}, \obj{2}, \dots\}$, with $\id_{\preTLf{q}{\KK}} = \obj{0}$ and $\obj{m}\otimes\obj{n} = \obj{m+n}$;
			\item The morphisms of $\preTLf{q}{\KK}$ are generated by $\CUP\in \Hom_{\preTLf{q}{\KK}}(\obj{0}, \obj{2})$ and $\CAP \in \Hom_{\preTLf{q}{\KK}} (\obj{2}, \obj{0})$ depicted by the string diagrams
			\[
			\raisebox{4pt}{$\CUP=\; $}
			\begin{tikzpicture}
	\begin{pgfonlayer}{nodelayer}
		\node [style=none] (0) at (-0.5, 0.5) {};
		\node [style=none] (1) at (0, 0) {};
		\node [style=none] (2) at (0.5, 0.5) {};
	\end{pgfonlayer}
	\begin{pgfonlayer}{edgelayer}
		\draw [style=thickstrand, bend right=45] (0.center) to (1.center);
		\draw [style=thickstrand, bend left=45] (2.center) to (1.center);
	\end{pgfonlayer}
\end{tikzpicture}\quad
			\raisebox{4pt}{ and \quad 
				$\CAP=\;$ }
			\begin{tikzpicture}
	\begin{pgfonlayer}{nodelayer}
		\node [style=none] (0) at (-0.5, 0) {};
		\node [style=none] (1) at (0, 0.5) {};
		\node [style=none] (2) at (0.5, 0) {};
	\end{pgfonlayer}
	\begin{pgfonlayer}{edgelayer}
		\draw [style=thickstrand, bend left=45] (0.center) to (1.center);
		\draw [style=thickstrand, bend right=45] (2.center) to (1.center);
	\end{pgfonlayer}
\end{tikzpicture}\;;
			\]
			
			\item the generators $\CUP$ and $\CAP$ satisfy the following relations:
			\[
			\begin{tikzpicture}
	\begin{pgfonlayer}{nodelayer}
		\node [style=none] (0) at (-1.5, 0.75) {};
		\node [style=none] (1) at (-1.25, 1) {};
		\node [style=none] (2) at (-1, 0.75) {};
		\node [style=none] (3) at (-2, 0.5) {};
		\node [style=none] (4) at (-1.75, 0.25) {};
		\node [style=none] (5) at (-1.5, 0.5) {};
		\node [style=none] (6) at (-1, 0) {};
		\node [style=none] (7) at (-2, 1.25) {};
		\node [style=none] (8) at (-0.5, 0.75) {$=$};
		\node [style=none] (9) at (0, 0) {};
		\node [style=none] (10) at (0, 1.25) {};
		\node [style=none] (11) at (1.5, 0.75) {};
		\node [style=none] (12) at (1.25, 1) {};
		\node [style=none] (13) at (1, 0.75) {};
		\node [style=none] (14) at (2, 0.5) {};
		\node [style=none] (15) at (1.75, 0.25) {};
		\node [style=none] (16) at (1.5, 0.5) {};
		\node [style=none] (17) at (1, 0) {};
		\node [style=none] (18) at (2, 1.25) {};
		\node [style=none] (19) at (0.5, 0.75) {$=$};
	\end{pgfonlayer}
	\begin{pgfonlayer}{edgelayer}
		\draw [style=thickstrand, bend left=45] (0.center) to (1.center);
		\draw [style=thickstrand, bend right=45] (2.center) to (1.center);
		\draw [style=thickstrand, bend right=45] (3.center) to (4.center);
		\draw [style=thickstrand, bend left=45] (5.center) to (4.center);
		\draw [style=thickstrand] (5.center) to (0.center);
		\draw [style=thickstrand] (6.center) to (2.center);
		\draw [style=thickstrand] (3.center) to (7.center);
		\draw [style=thickstrand] (9.center) to (10.center);
		\draw [style=thickstrand, bend right=45] (11.center) to (12.center);
		\draw [style=thickstrand, bend left=45] (13.center) to (12.center);
		\draw [style=thickstrand, bend left=45] (14.center) to (15.center);
		\draw [style=thickstrand, bend right=45] (16.center) to (15.center);
		\draw [style=thickstrand] (16.center) to (11.center);
		\draw [style=thickstrand] (17.center) to (13.center);
		\draw [style=thickstrand] (14.center) to (18.center);
	\end{pgfonlayer}
\end{tikzpicture}
			\]
			and
			\[
\begin{tikzpicture}
	\begin{pgfonlayer}{nodelayer}
		\node [style=none] (1) at (0, 1) {};
		\node [style=none] (3) at (-0.5, 0.5) {};
		\node [style=none] (4) at (0, 0) {};
		\node [style=none] (5) at (0.5, 0.5) {};
		\node [style=none] (6) at (1.65, 0.5) {$=-[2]_q\;\Id_{{\obj{0}}}.$};
	\end{pgfonlayer}
	\begin{pgfonlayer}{edgelayer}
		\draw [style=thickstrand, bend right=45] (3.center) to (4.center);
		\draw [style=thickstrand, bend left=45] (5.center) to (4.center);
		\draw [style=thickstrand, bend right=45] (5.center) to (1.center);
		\draw [style=thickstrand, bend left=45] (3.center) to (1.center);
	\end{pgfonlayer}
\end{tikzpicture}
			\]
		\end{itemize}
	\end{definition}
	
	The category $\preTLf{q}{\KK}$ is rigid, with $\obj{1}$ self-dual, where the morphisms $\CUP$ and $\CAP$ are the coevaluation and evaluation morphisms. 
	The duality functors $X\mapsto X^*$ and $X\mapsto \prescript{*}{}X$ are given on morphisms by counterclockwise and clockwise rotations by $180^\circ$ respectively, and hence are canonically isomorphic. 
	Thus, $\preTLf{q}{\KK}$ is spherical; indeed, it has exactly two pivotal structures (both spherical), one given by the canonical isomorphism $X\overset{\sim}{\to}X^{**}$, and the other given by composing the first with with the isomorphism $X\overset{\sim}{\to}X$ defined by $(-1)^n \Id_X$ for $X=\obj{n}$. 
	The fact that no additional pivotal structures exist is a consequence of the set of pivotal structures being a torsor over $\Aut_\otimes(\Id_{\TLf{q}{\KK}})\cong \ZZ/2\ZZ$, \cite[Exercise 4.7.16]{EGNO}.
	
	It will be important for us to think of $q$ as either a formal parameter as defined above or as a specific element of $\KK$;
	more precisely, the evaluation homomorphism $\mathrm{ev}_a:\KK[q, q^{-1}] \to \KK$ given by $q\mapsto a$ for some $a\in \KK^\times$ defines an induction functor $\KK[q,q^{-1}]$--$\mathbf{Mod} \to \Vect$, which then defines a monoidal $2$-functor $\mathrm{Ev}_a:\mathbf{Cat}_{\KK[q,q^{-1}]} \to \mathbf{Cat}_\KK$. 
	Then we can speak of $\preTLf{a}{\KK} := \mathrm{Ev}_a(\preTLf{q}{\KK})$ for any $a\in\KK^\times$. 
	
	\begin{definition}
		The \emph{Temperley--Lieb Category} $\TLf{q}{\KK}$ is defined to be the Cauchy completion $\preTLf{a}{\KK}^c$.
	\end{definition}
	
	Let $\mathbf{Fund}(\uqsl)$ be the full monoidal subcategory of $\uqsl$--$\mathbf{Mod}$ generated by the fundamental $2$-dimensional representation.
	The following result is folklore and, at least in the non-quantum setting, dates back to \cite{RTW}. 
	A much more general version of it, in particular, treating the quantum case, is proven in detail in \cite[Theorem~2.58]{E}, for example. 
	\begin{theorem}\label{thm:sl2fund}
		For $q\in\KK^\times$, there's a monoidal equivalence $\preTLf{q}{\KK} \simeq \mathbf{Fund}(\uqsl)$.
	\end{theorem}
	
	\begin{corollary}
		For $q\in\KK^\times$ not a root of unity, every representation of $\uqsl$ of type $\mathbf I$ is a direct summand of some tensor power of the fundamental representation. Therefore, we have a monoidal equivalence $\mathbf{Rep}_{\mathbf I}(\uqsl)\simeq\TLf{q}{\KK}$, where $\mathbf{Rep}_{\mathbf I}(\uqsl)$ is the category of type $\mathbf I$ representations of $\uqsl$.
	\end{corollary}

Since the Temperley--Lieb Category is defined by generators and relations, it is characterized by the following universal property, for example, see \cite[Theorem 2.1]{EO1}.

\begin{proposition}\label{univ-prop}
	Let $\cat$ be a tensor category with associator $\alpha$, and $X$ is an object of $\cat$ equipped with two morphisms $f: \id\to X\otimes X$ and $g:X\otimes X\to \id$ satisfying
	$$(\Id_X\otimes g)\circ \alpha_{X,X,X} \circ  (f\otimes \Id_X) = \Id_X = (g\otimes \Id_X)\circ \alpha_{X,X,X} \circ (\Id_X\otimes f), \text{  and}$$
	$$g\circ f = -[2]_q \Id_\id.$$
	Then, there is a unique tensor functor $F:\TLf{q}{\KK}\to \cat$ with $F(\obj{1}) = X$, $F(\CUP)=f$, $F(\CAP) = g$. \hfill $\square$
	\end{proposition}

	\begin{remark}
	Note that $\Hom_{\TLf{q}{\KK}}(\obj{n},\obj{m}) = \Hom_{\preTLf{q}{\KK}}(\obj{n},\obj{m})$, so whenever we use the notation $\obj{n}, \obj{m}, \dots$ for objects, we may unambiguously drop the subscripts of $\Hom$ or  $\End$ to refer to morphism spaces in either category.
		
	\end{remark}


	\subsection*{Braidings} \label{TL_braidings}
	
	We now examine all the possible braidings on $\TLf{q}{\KK}$.
	
	Suppose $\sigma$ is a braiding on $\TLf{q}{\LL}$, and let us denote $\sigma_{\obj 1,\obj 1}: \obj{1}\otimes \obj{1}\to \obj{1}\otimes \obj{1}$ by a positive crossing and its inverse by a negative crossing:
	\begin{center}
\begin{tikzpicture}[scale=0.5]
	\begin{pgfonlayer}{nodelayer}
		\node [style=none] (0) at (-1, 1) {};
		\node [style=none] (1) at (1, 1) {};
		\node [style=none] (2) at (-1, -1) {};
		\node [style=none] (3) at (1, -1) {};
		\node [style=none] (4) at (-3, 0) {$\sigma_{\obj1,\obj1} = $};
		\node [style=none] (5) at (2, 0) {,};
		\node [style=none] (6) at (4, 0) {$\sigma_{\obj1,\obj1}^{-1} = $};
		\node [style=none] (7) at (6, 1) {};
		\node [style=none] (8) at (8, 1) {};
		\node [style=none] (9) at (6, -1) {};
		\node [style=none] (10) at (8, -1) {};
	\end{pgfonlayer}
	\begin{pgfonlayer}{edgelayer}
		\draw [style=thickstrand] (0.center) to (3.center);
		\draw [style=braid-over] (2.center) to (1.center);
		\draw [style=thickstrand] (9.center) to (8.center);
		\draw [style=braid-over] (7.center) to (10.center);
	\end{pgfonlayer}
\end{tikzpicture}

	\end{center}
	
	Now, by the two hexagon axioms, one easily sees that $\sigma_{\obj1,\obj1}$ determines all of $\sigma$; indeed,
	
	\begin{center}
\begin{tikzpicture}[scale=0.25]
	\begin{pgfonlayer}{nodelayer}
		\node [style=none] (0) at (-10, 6) {};
		\node [style=none] (1) at (-8, 6) {};
		\node [style=none] (2) at (-4, 6) {};
		\node [style=none] (3) at (-10, -4) {};
		\node [style=none] (4) at (-8, -4) {};
		\node [style=none] (5) at (-4, -4) {};
		\node [style=none] (6) at (4, 6) {};
		\node [style=none] (7) at (6, 6) {};
		\node [style=none] (8) at (10, 6) {};
		\node [style=none] (9) at (10, -4) {};
		\node [style=none] (10) at (6, -4) {};
		\node [style=none] (11) at (4, -4) {};
		\node [style=none] (12) at (-6, 6) {\dots};
		\node [style=none] (13) at (1, 1) {\dots};
		\node [style=none] (14) at (8, -4) {\dots};
		\node [style=none] (16) at (-7, -5.5) {$m$};
		\node [style=none] (17) at (7, -5.5) {$n$};
		\node [style=none] (18) at (-7, 7.5) {$n$};
		\node [style=none] (19) at (7, 7.5) {$m$};
		\node [style=none] (20) at (-6, -4) {\dots};
		\node [style=none] (21) at (8, 6) {\dots};
		\node [style=none] (22) at (-13, 1) {$\sigma_{\obj{m},\obj{n}} = $};
	\end{pgfonlayer}
	\begin{pgfonlayer}{edgelayer}
		\draw [style=thickstrand] (0.center) to (11.center);
		\draw [style=thickstrand] (10.center) to (1.center);
		\draw [style=thickstrand] (9.center) to (2.center);
		\draw [style=braid-over] (3.center) to (6.center);
		\draw [style=braid-over] (4.center) to (7.center);
		\draw [style=braid-over] (5.center) to (8.center);
	\end{pgfonlayer}
\end{tikzpicture}

	\end{center}
where each crossing is a $\sigma_{\obj1,\obj1}$.
Since $\{\Id_{\obj2}, \CUP\circ\CAP\}$ is a basis for $\End(\obj2)$, we may write 
\begin{equation}\tag{$\star$}\label{braiding}
\begin{tikzpicture}[scale=0.5]
	\begin{pgfonlayer}{nodelayer}
		\node [style=none] (0) at (-2.25, 1) {};
		\node [style=none] (1) at (-0.75, 1) {};
		\node [style=none] (2) at (-2.25, -1) {};
		\node [style=none] (3) at (-0.75, -1) {};
		\node [style=none] (4) at (0, 0) {$=$};
		\node [style=none] (6) at (2, 1) {};
		\node [style=none] (7) at (3, 1) {};
		\node [style=none] (8) at (2, -1) {};
		\node [style=none] (9) at (3, -1) {};
		\node [style=none] (10) at (4, 0) {$+$};
		\node [style=none] (11) at (5, 0) {$b$};
		\node [style=none] (12) at (6, 1) {};
		\node [style=none] (13) at (7, 1) {};
		\node [style=none] (14) at (6, -1) {};
		\node [style=none] (15) at (7, -1) {};
		\node [style=none] (16) at (1, 0) {$a$};
		\node [style=none] (17) at (7.5, 0) {,};
	\end{pgfonlayer}
	\begin{pgfonlayer}{edgelayer}
		\draw [style=thickstrand] (6.center) to (8.center);
		\draw [style=thickstrand] (9.center) to (7.center);
		\draw [style=thickstrand, bend right=90, looseness=2.25] (12.center) to (13.center);
		\draw [style=thickstrand, bend left=90, looseness=2.25] (14.center) to (15.center);
		\draw [style=thickstrand] (0.center) to (3.center);
		\draw [style=braid-over] (2.center) to (1.center);
	\end{pgfonlayer}
\end{tikzpicture}
\end{equation}
for some $a,b\in\KK$. But if $\sigma$ is to be natural, we must also have 
\begin{center}
	\begin{tikzpicture}[scale=0.5]
	\begin{pgfonlayer}{nodelayer}
		\node [style=none] (0) at (-6, -2) {};
		\node [style=none] (1) at (-5, -2) {};
		\node [style=none] (2) at (-4, -2) {};
		\node [style=none] (3) at (-4, 0) {};
		\node [style=none] (4) at (-3, -1.1) {$=$};
		\node [style=none] (5) at (-1, -2) {};
		\node [style=none] (6) at (0, -2) {};
		\node [style=none] (7) at (-1.75, -1.25) {};
		\node [style=none] (8) at (-1.25, -0.75) {};
		\node [style=none] (9) at (-2, -2) {};
		\node [style=none] (10) at (0, 0) {};
		\node [style=none] (11) at (1, -1.1) {$=$};
		\node [style=none] (12) at (2.25, -1) {$a^2$};
		\node [style=none] (13) at (6.25, -1) {$+$};
		\node [style=none] (14) at (7.25, -1) {$ab$};
		\node [style=none] (15) at (11.25, -1) {$+$};
		\node [style=none] (16) at (12.25, -1) {$ba$};
		\node [style=none] (17) at (16.25, -1) {$+$};
		\node [style=none] (18) at (17.25, -1) {$b^2$};
		\node [style=none] (19) at (3.25, -2) {};
		\node [style=none] (20) at (4.25, -2) {};
		\node [style=none] (21) at (5.25, -2) {};
		\node [style=none] (22) at (5.25, 0) {};
		\node [style=none] (23) at (8.25, -2) {};
		\node [style=none] (24) at (10.25, 0) {};
		\node [style=none] (25) at (9.25, -2) {};
		\node [style=none] (26) at (10.25, -2) {};
		\node [style=none] (27) at (13.25, -2) {};
		\node [style=none] (28) at (14.25, -2) {};
		\node [style=none] (29) at (15.25, -2) {};
		\node [style=none] (30) at (15.25, 0) {};
		\node [style=none] (31) at (13.25, -0.5) {};
		\node [style=none] (32) at (14.25, -0.5) {};
		\node [style=none] (33) at (18.25, -2) {};
		\node [style=none] (34) at (19.25, -2) {};
		\node [style=none] (35) at (20.25, -2) {};
		\node [style=none] (36) at (20.25, 0) {};
		\node [style=none] (37) at (1, -4) {$=$};
		\node [style=none] (38) at (4.5, -4) {$(a^2-[2]_qab+b^2)$};
		\node [style=none] (39) at (8, -5) {};
		\node [style=none] (40) at (9, -5) {};
		\node [style=none] (41) at (10, -5) {};
		\node [style=none] (42) at (10, -3) {};
		\node [style=none] (43) at (11, -4) {$+$};
		\node [style=none] (44) at (13, -5) {};
		\node [style=none] (45) at (14, -5) {};
		\node [style=none] (46) at (15, -5) {};
		\node [style=none] (47) at (15, -3) {};
		\node [style=none] (48) at (12, -4) {$ab$};
		\node [style=none] (49) at (16, -4) {,};
	\end{pgfonlayer}
	\begin{pgfonlayer}{edgelayer}
		\draw [style=thickstrand] (0.center) to (3.center);
		\draw [style=thickstrand, bend left=90, looseness=2.00] (1.center) to (2.center);
		\draw [style=thickstrand] (7.center) to (5.center);
		\draw [style=thickstrand] (6.center) to (8.center);
		\draw [style=braid-over] (9.center) to (10.center);
		\draw [style=thickstrand, in=135, out=135, looseness=2.00] (7.center) to (8.center);
		\draw [style=thickstrand, bend left=90, looseness=2.00] (19.center) to (20.center);
		\draw [style=thickstrand] (22.center) to (21.center);
		\draw [style=thickstrand, bend left=90, looseness=2.00] (25.center) to (26.center);
		\draw [style=thickstrand] (23.center) to (24.center);
		\draw [style=thickstrand, bend right=90, looseness=1.50] (31.center) to (32.center);
		\draw [style=thickstrand] (30.center) to (29.center);
		\draw [style=thickstrand, bend left=90, looseness=2.00] (27.center) to (28.center);
		\draw [style=thickstrand, bend left=90, looseness=2.00] (33.center) to (34.center);
		\draw [style=thickstrand] (35.center) to (36.center);
		\draw [style=thickstrand, bend left=90, looseness=1.50] (31.center) to (32.center);
		\draw [style=thickstrand, bend left=90, looseness=2.00] (39.center) to (40.center);
		\draw [style=thickstrand] (41.center) to (42.center);
		\draw [style=thickstrand, bend left=90, looseness=2.00] (45.center) to (46.center);
		\draw [style=thickstrand] (44.center) to (47.center);
	\end{pgfonlayer}
\end{tikzpicture}
\end{center}
whence $ab=1$ and $[2]_q = a^2+a^{-2}$, which has exactly four solutions for $q\neq \pm 1$, namely $a = \pm q^{\pm 1/2}$, where the signs are chosen independently.
A simple check of naturality verifies that each of these indeed yields a braiding.
The inverse braiding is given by replacing $a$ with $a^{-1}$, which is another one of these four solutions.
These two solutions coincide if $a=a^{-1}\iff a = \pm 1\iff q=1$, giving two symmetric braidings, i.e. satisfying $\sigma^2 = 1$.
When $q=-1$, we get two non-symmetric braidings corresponding to $ a= \pm \sqrt{-1}$, and these braidings satisfy $\sigma^4=1$.

	\subsection*{Indecomposables}

	We would like to identify the indecomposable objects of $\TLf{q}{\KK}$, as well as establish its semisimplicity for generic $q$. 

\begin{proposition}\label{TL_tractable}
$\TLf{q}{\KK}$ is a tractable tensor category.	
\end{proposition}

\begin{proof}
	Let $\Vect_\ZZ$ be the category of $\ZZ$-graded $\KK$-vector spaces.
	Pivotal structures on $\Vect_\ZZ$ are given by tensor automorphisms of the identity functor, $\tAut\left(\Id_{\Vect_\ZZ}\right)\cong \Hom_{\mathrm{Group}}(\ZZ, \KK^\times)\cong \KK^\times$ (see \cite{EGNO}, Exercise 4.7.16), and in particular $-q\in \KK^\times$ endows $\Vect_\ZZ$ with a spherical structure making the one-dimensional vector space in degree $1$ (resp. degree $-1$) have quantum dimension $-q$ (resp. $-q^{-1}$).
	
	Now, let $X$ be the two-dimensional vector space in $\Vect_\ZZ$, with basis $\{u,v\}$ where $\deg(u) = 1$ and $\deg(v) = -1$.
	Then, $\dim_q(X) = -[2]_q$.
	Let $f:\id \to X\otimes X$ be given by $f(1)=-q u\otimes v + v\otimes u$, and $g:X\otimes X\to \id$ be given by $g(u\otimes u)=g(v\otimes v) = 0$, $g(u\otimes v) = 1$, and $g(v\otimes u) = -q^{-1}$.
	It is easy to check to $X, f$, and $g$ satisfy the conditions in the universal property of \autoref{univ-prop}, so we have a tensor functor $F:\TLf{q}{\KK}\to \Vect_\ZZ$ sending $\obj1$ to $X$ and $\CUP$ and $\CAP$ to $f$ and $g$, respectively.
	Moreover, $F$ intertwines the pivotal structures, and so the claim follows by \autoref{functorTractable}. 
\end{proof}

	Let $A=\End(\obj n)$.
	The idempotents of $A$ are in one-to-one correspondence with summands of $A$ as (projective) left $A$-module. 
	Moreover, this correspondence is compatible with the notion of an isomorphism (where two idempotents are said to be isomorphic if they are so as objects of $\TLf{q}{\KK}$.
	Thus, we find that primitive idempotents in $A$, i.e. indecomposable summands of $\obj n$ in $\TLf{q}{\KK}$, correspond exactly to indecomposable projective modules over $A$.
	Since $A$ is finite-dimensional, these exactly correspond to simple $A$ modules. 
	To a simple module $L$ over $A$, let us denote by $e_L$ the corresponding idempotent in $A$. 
	Now, for any idempotent $e\in A$, we have $\Hom_A(Ae, L)\neq 0$ if and only if $e_L$ is a summand of $e$. 
	
Now, we can apply this analysis inductively to compute the decompositions of $\obj{n}$ for each $n$. 
First, $\obj{0}$ and $\obj{1}$ are indecomposable since their endomorphism algebras are $1$-dimensional and hence simple.

	\begin{lemma}
		For any $n\geq 1$, we have $\obj n$ as a direct summand of $\obj{n+2}$. Moreover, $\obj 0 $ is a summand of $\obj 2$ if and only if $q^2 \neq -1$.
	\end{lemma}

\begin{proof}
	We proceed by induction on $n$.
	Observe that $p=(\Id_{\obj 1} \otimes \CUP)$ and $\iota=(\CAP\otimes \Id_{\obj 1})$ act, respectively, as inclusion and projection of $\obj 1$ to and from $\obj 3$; indeed, $\iota\circ p$ is an idempotent and $p\circ \iota= \Id_{\obj 1}$. 
	So, $\obj{1}$ is a summand of $\obj{3}$.
	Therefore, $\obj{n+1} = \obj{1}\otimes \obj{n}$ as a summand of $\obj{3}\otimes\obj{n} = \obj{n+3}$.
	
	For the second claim, observe that $\End(\obj 2) = \langle 1, x\;|\; x^2 = -[2]_q x\rangle$, which is semisimple if and only if $[2]_q\neq 0$ (or equivalently $q^2\neq -1$); in the semisimple case, $-[2]_q^{-1}\CAP\circ \CUP$ is the idempotent corresponding to $\obj 0$ as a summand of $\obj 2$.
	\end{proof}
	
	\begin{lemma}\label{mod_e_end}
		Let  \begin{center}
\begin{tikzpicture}[scale=0.7]
	\begin{pgfonlayer}{nodelayer}
		\node [style=none] (0) at (-1, 1) {};
		\node [style=none] (1) at (-0.5, 1) {};
		\node [style=none] (2) at (0, 1) {};
		\node [style=none] (3) at (0.5, 1) {};
		\node [style=none] (4) at (1, 1) {};
		\node [style=none] (5) at (2, 1) {};
		\node [style=none] (6) at (-1, -1) {};
		\node [style=none] (7) at (-0.5, -1) {};
		\node [style=none] (8) at (0, -1) {};
		\node [style=none] (9) at (0.5, -1) {};
		\node [style=none] (10) at (1, -1) {};
		\node [style=none] (11) at (2, -1) {};
		\node [style=none] (12) at (1.5, 0) {$\dots$};
		\node [style=none] (13) at (-2, 0) {$e=$};
		\node [style=none] (14) at (3.5, 0) {$\in\End(\obj{n}).$};
	\end{pgfonlayer}
	\begin{pgfonlayer}{edgelayer}
		\draw [style=thickstrand] (0.center) to (8.center);
		\draw [style=thickstrand, bend left=90, looseness=2.25] (6.center) to (7.center);
		\draw [style=thickstrand, bend right=90, looseness=2.25] (1.center) to (2.center);
		\draw [style=thickstrand] (3.center) to (9.center);
		\draw [style=thickstrand] (10.center) to (4.center);
		\draw [style=thickstrand] (5.center) to (11.center);
	\end{pgfonlayer}
\end{tikzpicture}
\end{center}	
Then, the two-sided principal ideal $I = \langle e \rangle$ is of codimension 1 in $\End(\obj n)$, so that $\End(\obj n)/\langle e \rangle\cong \KK\cdot \Id_{\obj n}$. 
		\end{lemma}

\begin{proof}
	We can show that all diagrams except for $\Id_{\obj n}$ are in $I$.
	Note that any Temperley--Lieb diagram $\End(\obj{n})$ other than the identity must have at least one cup and one cap. 
	One can add more cups and caps to $e$ by composing with the appropriate morphisms. 
	Then, one can move the cups and caps by pre-composing and post-composing with an appropriate braiding to get any other diagram with at least one cup and cap. 
\end{proof}

\begin{corollary}\label{newSummands}
	For all $n\geq 2$ , there is a exactly one indecomposable summand $T_n$ of $\obj n$ which is not a summand of $\obj {n-2}$.
\end{corollary}

\begin{proof}
	A summand of $\obj n$ corresponds to an idempotent in $A=\End(\obj n)$, which in turn corresponds to a summand of $A$ as a projective $A$-module.
	This correspondence respects the notions of sub-objects and isomorphisms, so indecomposable summands of $\obj n$ are in bijection with projective indecomposable summands of the regular $A$-module. 
	Since $A$ is finite dimensional, there is a canonical bijection between projective indecomposable $A$-modules and simple $A$-modules.
	Under this bijection, an indecomposable summand of $\obj{n}$ which is not also summands of $(\obj n, e)\cong \obj{n-2}$ corresponds to a simple $\End(\obj{n})$-module $L$ such that $eL= 0$, or equivalently, a simple $\End(\obj{n})/\langle e\rangle$-module.
	But by \autoref{mod_e_end}, there's only one such simple module.
\end{proof}
	
\begin{lemma}\label{nextIndecomposbleLemma}
	$T_{n+1}$ is a summand of $T_n\otimes \obj{1}$.
\end{lemma}

\begin{proof}
	By \autoref{newSummands}, $\obj{n} = T_{n} \oplus X$, where $X$ consists of summands appearing in $\obj{k}$ for $k\leq n-2$. So, upon tensoring with $\obj{1}$, all the summands of $X\otimes \obj{1}$ must be summands of $\obj{n-1}$, which $T_{n+1}$ is not, by definition.
\end{proof}

\begin{corollary}\label{fusionRule}
	Assume $\dim_q T_i\neq 0$ for $i=0,1,\dots, n$. Then, $T_n\otimes \obj{1} \cong T_{n+1}\oplus T_{n-1}$.
\end{corollary}

\begin{proof}
	Applying \autoref{multiplicitiesLemma} to \autoref{nextIndecomposbleLemma}, we see that $T_{n-1}$ is a summand (of multiplicity $1$) in $T_n\otimes \obj{1}$. 
	If $T_{n-2}\subset \obj{n-2}$ were a summand of $T_n\otimes \obj 1 \subset \obj{n+1}$, we would have a nonzero morphism $\obj{n-2}\to\obj{n+1}$, which is impossible due to mismatched parities.
	Further, if $T_{m}$ is a summand of $T_n\otimes\obj{1}$ for $m\leq n-3$, we also get a contradiction since that implies, again by \autoref{multiplicitiesLemma}, that $T_n$ is a summand of $T_{m}\otimes\obj{1}$ and hence of $\obj{m+1}$, contradicting the definition of $T_n$.  
\end{proof}

Let us denote $T_0:= \obj{0} = \id_{\TLf{q}{\KK}}$ and $T_1 = \obj{1}$. 
We have the following celebrated theorem of Jones \cite{J} (or see \cite{KL}, \cite{T}, or \cite{Wang} for an exposition):

\begin{theorem}\label{thm:TLindecomposables} For any algebraically closed field $\KK$ and $q\in\KK^\times$ satisfying $q^4\neq 1\neq q^6$, we have 

	\begin{enumerate}
		\item $\{T_n\}_{n\in\NN}$ is a complete list of indecomposable objects in $\TLf{q}{\KK}$;
		\item $\dim_q(T_n) = [n+1]_{\pm q}$, where the sign depends on the choice of spherical structure.
\vspace{0.1in}

\noindent Moreover, when $q$ is not a root of unity in $\KK$, we have		
		
		\item $T_m\otimes T_n \cong T_{|m-n|} \oplus T_{|m-n|+2} \oplus \dots \oplus T_{m+n}$;
		\item The category $\TLf{q}{\KK}$ is semisimple. 
	\end{enumerate} 
\end{theorem}

\begin{proof}
	\begin{enumerate}
		\item This follows inductively from \autoref{newSummands}.
		\item \autoref{fusionRule} implies that $\dim_q(T_{n+1}) = [2]_{\pm q}\dim_q(T_n)-\dim_q(T_{n-1})$, from which a simple induction proves the claim. 
		\item When $q$ is not a root of unity, we have $\dim_q(T_n) = [n+1]_{\pm q} \neq 0$ for all $n\in\NN$.
		When $n=1$, the claim is a restatement of \autoref{fusionRule}, and for other $m$ proceed inductively by tensoring both sides with $T_1=\obj{1}$. 
		\item Now, we can make use of the fact that $\Rep{SL_2}$ satisfies the same Clebsch-Gordan fusion rule (3), i.e. both categories have isomorphic Grothendieck rings.  Suppose $\obj{n}= \bigoplus_i T_i^{\oplus m_i}$. Then, 
		$$\dim(\End(\obj{n})) = \sum_{i,j} m_i m_j \dim(\Hom(T_i, T_j)) \geq \sum_i m_i^2 \dim(\End(T_i))\geq \sum_i m_i^2,$$
		with the equality occurring if and only if $ \End(T_i) \cong \KK$ and $\Hom(T_i,T_j) = 0$ when $i\neq j$. 
		But the equality does occur when replacing $\obj{n}$ with $V^{\otimes n}$ where $V\in\Rep{SL_2}$ is the fundamental $2$-dimensional representation, which we know has the same direct sum decomposition as $\obj{n}$ in $\TLf{q}{\KK}$. 
		Hence both sides of the inequality must coincide with the $n^{\text{th}}$ Catalan number. It follows that $T_n$ are simple and $\TLf{q}{\KK}$ is semisimple.
		\end{enumerate}	
	\end{proof}

\begin{theorem}[\cite{W}]
When $q$ is not a root of unity, the object $T_n$ is the image of the Jones-Wenzl idempotent $JW_n\in\End(\obj{n})$, defined recursively by $JW_1 = \Id_{\obj 1}$ and 
\begin{center}
\begin{tikzpicture}[scale=0.5]
	\begin{pgfonlayer}{nodelayer}
		\node [style=none] (0) at (1, 0) {};
		\node [style=none] (1) at (-0.75, 0) {};
		\node [style=jonesrectangle] (2) at (0, 2.5) {$JW_{n+1}$};
		\node [style=none] (3) at (-0.75, 2) {};
		\node [style=none] (4) at (1, 2) {};
		\node [style=none] (5) at (0.15, 1) {$\cdots$};
		\node [style=none] (6) at (-0.75, 3) {};
		\node [style=none] (7) at (1, 3) {};
		\node [style=none] (8) at (1, 5) {};
		\node [style=none] (9) at (-0.75, 5) {};
		\node [style=none] (10) at (2, 2.5) {$=$};
		\node [style=jonesrectangle] (11) at (3.5, 2.5) {$JW_n$};
		\node [style=jonesrectangle] (12) at (10.75, 3.75) {$JW_n$};
		\node [style=none] (13) at (0.15, 4) {$\cdots$};
		\node [style=none] (14) at (2.75, 3) {};
		\node [style=none] (15) at (4.25, 3) {};
		\node [style=none] (16) at (4.25, 2) {};
		\node [style=none] (17) at (2.75, 2) {};
		\node [style=none] (18) at (2.75, 0) {};
		\node [style=none] (19) at (4.25, 0) {};
		\node [style=none] (20) at (4.25, 5) {};
		\node [style=none] (21) at (2.75, 5) {};
		\node [style=none] (22) at (11, 4.25) {};
		\node [style=none] (23) at (10, 4.25) {};
		\node [style=none] (24) at (10, 3.25) {};
		\node [style=none] (26) at (11.5, 0) {};
		\node [style=none] (27) at (10, 0) {};
		\node [style=none] (28) at (10, 5) {};
		\node [style=none] (29) at (11, 5) {};
		\node [style=none] (30) at (11.5, 4.25) {};
		\node [style=none] (31) at (11.5, 5) {};
		\node [style=none] (32) at (5, 5) {};
		\node [style=none] (33) at (5, 0) {};
		\node [style=jonesrectangle] (34) at (10.75, 1.25) {$JW_n$};
		\node [style=none] (35) at (11, 1.75) {};
		\node [style=none] (36) at (10, 1.75) {};
		\node [style=none] (37) at (10, 0.75) {};
		\node [style=none] (38) at (11.5, 0.75) {};
		\node [style=none] (39) at (11.5, 1.75) {};
		\node [style=none] (43) at (12, 1.75) {};
		\node [style=none] (44) at (11.75, 2.25) {};
		\node [style=none] (45) at (12, 0) {};
		\node [style=none] (46) at (11.5, 3.25) {};
		\node [style=none] (47) at (12, 3.25) {};
		\node [style=none] (48) at (11.75, 2.75) {};
		\node [style=none] (49) at (12, 5) {};
		\node [style=none] (50) at (11, 3.25) {};
		\node [style=none] (51) at (10.6, 2.5) {$\cdots$};
		\node [style=none] (52) at (10.6, 4.7) {$\cdots$};
		\node [style=none] (53) at (11, 0.75) {};
		\node [style=none] (54) at (11, 0) {};
		\node [style=none] (55) at (10.6, 0.275) {$\cdots$};
		\node [style=none] (56) at (3.5, 4) {$\cdots$};
		\node [style=none] (57) at (3.5, 1) {$\cdots$};
		\node [style=none] (58) at (6, 2.5) {$-$};
		\node [style=none] (59) at (8, 2.5) {$\displaystyle\frac{[n]_q}{[n+1]_q}$};
	\end{pgfonlayer}
	\begin{pgfonlayer}{edgelayer}
		\draw [style=thickstrand] (4.center) to (0.center);
		\draw [style=thickstrand] (3.center) to (1.center);
		\draw [style=thickstrand] (9.center) to (6.center);
		\draw [style=thickstrand] (8.center) to (7.center);
		\draw [style=thickstrand] (29.center) to (22.center);
		\draw [style=thickstrand] (31.center) to (30.center);
		\draw [style=thickstrand] (17.center) to (18.center);
		\draw [style=thickstrand] (16.center) to (19.center);
		\draw [style=thickstrand] (21.center) to (14.center);
		\draw [style=thickstrand] (20.center) to (15.center);
		\draw [style=thickstrand] (32.center) to (33.center);
		\draw [style=thickstrand, in=0, out=90] (43.center) to (44.center);
		\draw [style=thickstrand, in=180, out=90] (39.center) to (44.center);
		\draw [style=thickstrand, in=0, out=-90] (47.center) to (48.center);
		\draw [style=thickstrand, in=-180, out=-90] (46.center) to (48.center);
		\draw [style=thickstrand] (43.center) to (45.center);
		\draw [style=thickstrand] (49.center) to (47.center);
		\draw [style=thickstrand] (23.center) to (28.center);
		\draw [style=thickstrand] (37.center) to (27.center);
		\draw [style=thickstrand] (38.center) to (26.center);
		\draw [style=thickstrand] (50.center) to (35.center);
		\draw [style=thickstrand] (24.center) to (36.center);
		\draw [style=thickstrand] (53.center) to (54.center);
	\end{pgfonlayer}
\end{tikzpicture}
\end{center}

\end{theorem}

The following example illustrates that, in general, $\TLf{q}{\KK}$ is not semisimple.
	
	\begin{example}
	The idempotent 
\[\begin{tikzpicture}[scale=0.7]
	\begin{pgfonlayer}{nodelayer}
		\node [style=none] (0) at (-3, 0) {$J=$};
		\node [style=none] (2) at (-1, 0.75) {};
		\node [style=none] (3) at (-0.75, 0.75) {};
		\node [style=none] (4) at (-0.5, 0.75) {};
		\node [style=none] (5) at (-1, -0.75) {};
		\node [style=none] (6) at (-0.75, -0.75) {};
		\node [style=none] (7) at (-0.5, -0.75) {};
		\node [style=none] (8) at (0.25, 0) {$+$};
		\node [style=none] (9) at (1, 0.75) {};
		\node [style=none] (10) at (1.25, 0.75) {};
		\node [style=none] (11) at (1.5, 0.75) {};
		\node [style=none] (12) at (1, -0.75) {};
		\node [style=none] (13) at (1.25, -0.75) {};
		\node [style=none] (14) at (1.5, -0.75) {};
		\node [style=none] (15) at (2.25, 0) {$+$};
		\node [style=none] (16) at (3, 0.75) {};
		\node [style=none] (17) at (3.25, 0.75) {};
		\node [style=none] (18) at (3.5, 0.75) {};
		\node [style=none] (19) at (3, -0.75) {};
		\node [style=none] (20) at (3.25, -0.75) {};
		\node [style=none] (21) at (3.5, -0.75) {};
		\node [style=none] (22) at (4.75, 0) {$-\sqrt{2}$};
		\node [style=none] (23) at (6, 0.75) {};
		\node [style=none] (24) at (6.25, 0.75) {};
		\node [style=none] (25) at (6.5, 0.75) {};
		\node [style=none] (26) at (6, -0.75) {};
		\node [style=none] (27) at (6.25, -0.75) {};
		\node [style=none] (28) at (6.5, -0.75) {};
		\node [style=none] (31) at (7.75, 0) {$-\sqrt{2}$};
		\node [style=none] (32) at (9, 0.75) {};
		\node [style=none] (33) at (9.25, 0.75) {};
		\node [style=none] (34) at (9.5, 0.75) {};
		\node [style=none] (35) at (9, -0.75) {};
		\node [style=none] (36) at (9.25, -0.75) {};
		\node [style=none] (37) at (9.5, -0.75) {};
	\end{pgfonlayer}
	\begin{pgfonlayer}{edgelayer}
		\draw [style=thickstrand] (2.center) to (5.center);
		\draw [style=thickstrand] (3.center) to (6.center);
		\draw [style=thickstrand] (7.center) to (4.center);
		\draw [style=thickstrand] (9.center) to (14.center);
		\draw [style=thickstrand, bend left=90, looseness=2.50] (12.center) to (13.center);
		\draw [style=thickstrand, bend right=90, looseness=2.50] (10.center) to (11.center);
		\draw [style=thickstrand, bend right=90, looseness=2.50] (16.center) to (17.center);
		\draw [style=thickstrand] (18.center) to (19.center);
		\draw [style=thickstrand, bend left=90, looseness=2.50] (20.center) to (21.center);
		\draw [style=thickstrand] (25.center) to (28.center);
		\draw [style=thickstrand, bend right=90, looseness=2.50] (23.center) to (24.center);
		\draw [style=thickstrand, bend left=90, looseness=2.50] (26.center) to (27.center);
		\draw [style=thickstrand, bend right=90, looseness=2.50] (33.center) to (34.center);
		\draw [style=thickstrand, bend left=90, looseness=2.50] (36.center) to (37.center);
		\draw [style=thickstrand] (32.center) to (35.center);
	\end{pgfonlayer}
\end{tikzpicture}
\]	is negligible in $\TLf{e^{\pi i/4}}{\CC}$, so $\dim_q (\obj{3},J) = 0$.
\end{example}
		
	\subsection*{Fusion Quotients}

	We shall now review the semisimplification of $\TLf{q}{\KK}$ when $q$ is a root of unity. 
	For the remainder of this section, let us fix an integer $\varkappa\geq 3$ and a primitive root of unity $q$ of order $2\varkappa$ in an algebraically closed field $\KK$ with $\char \KK\neq 2$.

	\begin{definition}
		Let $\VerF{\KK}$ be the the semisimplification $\overline{\TLf{q}{\KK}}$ of $\TLf{q}{\KK}$. 
	\end{definition}

\begin{proposition}\label{fusionRules}
	$\VerF{\KK}$ is a fusion category (i.e. semisimple tensor category with finitely many simples) with simple objects $L_n:=\overline{T_n}$ with $n=0, 1, 2, \dots, \varkappa-2$ and $\dim L_n = [n+1]_{\pm q}$. 
	They satisfy the truncated Clebsch-Gordon fusion rule given by 
	$$L_m\otimes L_n \cong \begin{cases}
		L_{|m-n|}\oplus L_{|m-n|+2} \oplus \dots \oplus L_{m+n} \quad \quad &m+n \leq \varkappa-2\\
		 L_{|m-n|}\oplus L_{|m-n|+2} \oplus \dots \oplus L_{(2\varkappa -4) - (m+n)} \quad \quad & m+n > \varkappa-2
	\end{cases}.$$
\end{proposition}

\begin{proof}
	The first statement is a direct consequence of \autoref{tractable-semisimplification}, \autoref{TL_tractable}, and \autoref{thm:TLindecomposables}, and so is the fusion rule in the case $m+n\leq \varkappa-2$. To see the fusion rule when $m+n>\varkappa-2$, first observe that 
	$$L_1\otimes L_{\varkappa-2}\cong \overline{T_{\varkappa-3}}\oplus \overline{T_{\varkappa-1}} \cong L_{\varkappa-3},$$
	since $T_{\varkappa-1}$ is negligible and so $\overline{T_{\varkappa-1}} = 0$. A simple induction argument then finishes the proof.
\end{proof}

Note that the fusion rules can be summarized via the $A_{\varkappa-1}$ Dynkin diagram, where vertices correspond to simples and neighbors are summands upon tensoring with $L_1$, hence the notation.
$\VerF{\KK}$ is known to be equivalent (as monoidal categories) to $\overline{\Tilt{\uqsl}{\KK}}$, the semisimplification of the category of tilting modules of $\uqsl$ (for example see \cite[Section 3.3]{BK} or \cite{Schop}). 
The categories $\overline{\Tilt{\uqg}{\KK}}$ were studied by Andersen and Paradowski in \cite{A} and \cite{AP}.
	
The modular data of $\VerF{\KK}$ is given by the following formulae (for example, see \cite[Section 5.2]{KL}, \cite[Section XII.5]{T}, or \cite[Theorem 3.3.20]{BK}).
	
\begin{proposition} For each simple objects $L_m$ and $L_n$ of $\VerF{\KK}$, we have
$$s_{m,n}= s_{L_m,L_n} =(-1)^{m+n} [(m+1)(n+1)]_{q} \quad\text{and}\quad \theta_n= t_{L_n,L_n} =  (-a)^{n(n+2)} = \pm q^{\pm n(n+2)/2},$$
where $a = \pm q ^{\pm 1/2}$ is the braiding constant from \autoref{braiding}.
\end{proposition} 

The following proposition is easily seen by examining the fusion rules.

\begin{proposition}\label{fusionSubCats}
	$\VerF{\KK}$ contains exactly four monoidal subcategories: the trivial category ($\simeq \Vect$) containing the single simple object $L_0 = \id$; the category $\langle L_{\varkappa -2}\rangle$ containing exactly two simples $L_0$ and $L_{\varkappa-2}$; the adjoint category $\Ad(\VerF{\KK}):= \langle L_2\rangle$, containing the simples $L_{2n}$; and $\VerF{\KK}$ itself.
\end{proposition}

\begin{proposition}\label{VerFmodularity}
	The M\"uger center of $\VerF{\KK}$ is trivial, hence $\VerF{\KK}$ is modular.
\end{proposition}
	
\begin{proof}
	The M\"uger center is a fusion subcategory of $\VerF{\KK}$, so by \autoref{fusionSubCats}, it suffices to rule out the three nontrivial possibilities. 
	If the M\"uger center is $\Ad(\VerF{\KK})$ or $\VerF{\KK}$, we would have the $L_2$ centralizes all objects of $\VerF{\KK}$, in particular, $\sigma_{L_2, L_1}\circ \sigma_{L_1\circ L_2} = \Id_{L_1\otimes L_2}$.
	Recall that the twists satisfy the relation
	$$\theta_{X\otimes Y} = \sigma_{Y,X}\sigma_{X,Y}(\theta_X\otimes \theta_Y),$$
	for example, see \cite[Section 8.10]{EGNO} or \cite[Section 2.2]{BK}.
	Therefore, we would have $\theta_{L_1\otimes L_2} = \theta_{L_1}\otimes \theta_{L_2}$. Since the left-hand side and the right-hand sides act on $L_3$ by the scalar $\theta_{3}$ and $\theta_{1}\theta_2$, we must have 
	$$1 = \frac{\theta_3}{\theta_1\theta_2} = \frac{\pm q^{\pm 15/2}}{q^{\pm 3/2}q^{\pm 4}} = \pm q^{\pm 2} \implies q^4 =1,$$
	contradicting that the order of $q$ is $2\varkappa \geq 6$.	
	
	To rule out the case that the M\"uger center is $\langle L_{\varkappa-2}\rangle$, we again use the $T$-matrix: if $L_{\varkappa-2}$ centralizes all objects of $\VerF{\KK}$, we would have $\sigma_{L_n, L_{\varkappa-2}}\circ \sigma_{L_{\varkappa-2}\circ L_n} = \Id_{L_{\varkappa-2}\otimes L_n}$ for every $n=0,1,\dots, \varkappa-2$, implying that, for some fixed choice of signs $e_1, e_2\in\{\pm1\}$ depending on the braiding,
	$$1 = \frac{\theta_{\varkappa-2-n}}{\theta_{\varkappa-2}\theta_n} = e_1 \frac{q^{e_2(\varkappa-n)(\varkappa-n-2)/2}}{q^{e_2\varkappa(\varkappa-2)/2}q^{e_2n(n+2)/2}} = e_1 q^{e_2\varkappa n} = e_1(-1)^n,$$
	(since $q^\varkappa = -1$ whenever $q$ is a \emph{primitive} $2\varkappa^{\text{th}}$ root of unity) but this is impossible to satisfy for all $n$ when $\varkappa\geq 3$ and $\char \KK \neq 2$. 
	
	Thus, $\Mug{\VerF{\KK}} \simeq \Vect$. Modularity follows immediately by \autoref{modularityCriteria}.
	\end{proof}
	
	\begin{corollary}\label{factorizableFusions}
		The canonical functor $D_\varkappa:\VerF{\KK} \boxtimes \VerF{\KK}^\rev\to \Dr{\VerF{\KK}}$ is an equivalence of braided tensor categories.
	\end{corollary}
	
	\begin{proof}
		This is a direct consequence of \autoref{modularCenter} and \autoref{VerFmodularity}.
	\end{proof}
	
\begin{definition}
	Define $\mathbf{Ver}_p$ as the semisimplification of $\TLf{1}{\overline\FF_p}\simeq \mathbf{Tilt}_{\overline\FF_p}(SL_2)$. 
\end{definition}

The  categories $\Verp$ are also equivalent to the semisimplification of the category $\mathbf{Rep}_{\overline\FF_p}(\ZZ/p\ZZ)$, and so its simple objects can be described via the theory of Jordan normal forms, \cite{O}. 
These categories are of particular importance as they act as the simplest targets of fiber functors in positive characteristic, \cite{O}, \cite{CEO}.
We record some basic facts about these categories in the following proposition.

\begin{proposition}\cite[Section 3.2]{O} The category $\Verp$ has the same Grothendieck ring as $\mathcal A_{p-1}(\overline\FF_p)$; in particular, it has $p-1$ simple objects $L_0, L_1, \dots, L_{p-2}$ obeying the truncated Clebsch-Gordon fusion rule of \autoref{fusionRules}.
\end{proposition}


\section{$\TLf{q}{\CC}$ in an Ultraproduct of Fusion Categories}\label{sec:TLultra}

		\begin{proposition}
		Given $q\in\CC$ which is transcendental over $\QQ$, there exists a sequence $q_k$ of primitive $(2\varkappa_k)^\text{th}$ roots of unity, with $\varkappa_k\to \infty$, and a field isomorphism $\uprod \CC \cong \CC$ mapping $\ulimit{k} q_k\mapsto q$.
	\end{proposition}
	
	\begin{proof}
		By \autoref{uprodFields}, we know that a field isomorphism $\uprod \CC \cong \CC$ exists.
		Let $q_k$ be any such sequence (e.g. $q_k = e^{\frac{\pi i}{\varkappa_k}}$). 
		Since $\varkappa_k\to \infty$, any polynomial of degree less than $n\in\NN$ is satisfied by at most finitely many $q_k$, and hence not satisfied by $\ulim{k}  q_k$ by \Los. 
		It follows that $\ulim{k}  q_k$ is transcendental and hence we can arrange the desired field isomorphism, for example, by extending the obvious isomorphism $\overline{\QQ}\left(\ulimit{k}  q_j\right)\cong \overline{\QQ}(q)$ to all of $\CC$.
	\end{proof}
			
	We wish to also express algebraic numbers $q$ as ultralimits of roots of unity $q_k$ of increasing even orders. To achieve this, we need to vary the characteristic $p$ as well as the order, First, let us illustrate how this happens when $q$ is an integer, as the general case will be a direct generalization. 
	
	\begin{proposition}
		Let $q\in\ZZ\setminus\{-1,0,1\}$. There exists infinitely many primes $p_k$ such that $q \mod p_k$ has a multiplicative order $2^k$.
	\end{proposition}
	\begin{proof}
		Consider the prime divisors of $A_k:= q^{2^k}+1$ for $k\in\NN$. 
		For any such prime $p$, we have $q^{2^k} \equiv -1 \mod p$ so that $q$ has order $2^{k+1}$ in $\FF_{p}$ as its order divides $2^{k+1}$ but not $2^k$.
		
		Observe that, for each $k\in\NN$, we have a prime divisors of $A_k$ that does not divide $A_j$ for all $j<k$. 
		Indeed, a simple induction shows that
		$$A_k-2 = A_{k-1} A_{k-2} \dots A_1 A_0 (n-1),$$
		hence, a prime dividing both $A_j$ for some $j<k$ and $A_k$ must also divide $2$;
		but each $A_k$ has a prime factor different from $2$ (if $q$ is even, this is obvious as $A_k$ is odd; and if $q$ is odd, then $q^2 \equiv 1 \mod 8$ so $A_k \equiv 2 \mod 8$ so it's not a power of $2$ as long as $q^2\neq 1$). 
		It follows that the set of prime divisors of $A_k$ as $k\to\infty$ is infinite.\footnote{As a corollary, this furnishes a proof of the infinitude of primes, which is a variation on a similar proof due to Goldbach and Hurwitz, and popularized by P\'olya and Szeg\H o. See \cite{H} for a detailed history.}
	\end{proof}

To generalize this to algebraic integers, we will need to use the following generalization of Zsigmondy's Theorem to algebraic number fields, due to Postnikova and Schinzel which we now state.

\begin{theorem}\cite[Theorem 1]{PS}\label{AlgPrimitiveDiv}
	Let $\alpha$ and $\beta$ be in the ring of integers $\mathcal O_{\KK}$ of an algebraic number field $\KK$. 
	There exists an integer $N\in\NN$ such that for all $n>N$, there is a prime ideal $\mathfrak{p}_n$ of $\mathcal O_\KK$ with $\mathfrak{p}_n\mid (\alpha^n-\beta^n)$ and $\mathfrak{p}_n\nmid (\alpha^\ell -\beta^\ell)$ for all $\ell<n$. 
\end{theorem}

	\begin{proposition} \label{algebraicNumFromRootsOfUnity}
		Let $q$ be any nonzero algebraic complex number that is not a root of unity, and let $f(x)\in\ZZ[x]$ be its minimal polynomial. There exists an infinite set $\mathcal P_q$ of primes $\{p_k\}_{k\in\NN}$ such that $f(x) \mod p_k$ has a root $q_k\in\overline{\FF}_{p_k}$ of multiplicative order $2^k$.
	\end{proposition}

\begin{proof}
		Let $q = \alpha/\beta$ for algebraic integers $\alpha, \beta$ with monic minimal polynomials $g(x), h(x)\in\ZZ[x]$, respectively.
		We work in a splitting field $\KK$ of $gh$ over $\QQ$. 
		By \autoref{AlgPrimitiveDiv}, for sufficiently large $k$, we have a prime ideal of  $\mathfrak p_{k}$ containing $\alpha^{2^{k+1}}-\beta^{2^{k+1}} = (\alpha^{2^{k}}-\beta^{2^{k}})(\alpha^{2^{k}}+\beta^{2^{k}})$ but not containing $\alpha^{2^{\ell+1}}-\beta^{2^{\ell+1}}$ for all $\ell<k$, and hence it contains $\alpha^{2^{k}}+\beta^{2^{k}}$ but not $\alpha^{2^{\ell}}+\beta^{2^{\ell}}$ for all $\ell<k$. 
		Now, $\mathfrak p_k \cap \ZZ$ is a prime ideal in $\ZZ$, principally generated by some rational prime $p_k\in\ZZ$. 
		Since $\mathfrak p_k$ contains $\alpha^{2^{k}}+\beta^{2^{k}}$, it also contains its norm $$N_k := \prod_{\sigma\in\gal(\KK/\QQ)} \left(\sigma(\alpha)^{2^{k}}+\sigma(\beta)^{2^{k}}\right).$$
		$N_k$ is an integer, and hence $N_k \in \mathfrak p_k\cap \ZZ = (p_k)$.
		
		It follows that $N_k \equiv 0 \mod p_k$, and hence $\overline{N_k} = \prod_{i=1}^d \left ({\alpha_i}^{2^k} + \beta_i^{2^{k}}\right) = 0 \in \overline{\FF}_{p_k}$, where $\alpha_i$ and $\beta_i$ are the roots of $g(x)$ and $h(x) \mod p_k$ in $\overline \FF_{p_k}$, respectively.
		Thus, for some $j\in\{1,\dots, d\}$, we have $\alpha_j^{2^k} = -\beta_j^{2^k}$, in $\overline \FF_{p_k}$.
		Setting $q_k:=\alpha_j/\beta_j$, we have $q^{2^k} = -1$, whence $q_k$ is of order $2^{k+1}$ in $\overline\FF_{p_k}$. 
		Note that if $\mathfrak p_k$ and $\mathfrak p_{k'}$ lie over the same prime ideal $(p_k)\in\ZZ$ for $k\neq k'$, we would have that some $q_k$ has orders $2^{k+1}$ and $2^{k'+1}$ in \emph{the same} field $\overline{\FF}_{p_k}$, which is impossible. 
		Thus, we get an infinite set $\mathcal P_q$ of rational primes $\{p_k\}_{k\in\NN}$ in this manner. 
	\end{proof}
	
	\begin{remark}
		We will only use a weaker form of \autoref{algebraicNumFromRootsOfUnity}, namely that $q_k$ can be chosen to have even multiplicative orders $2\varkappa_k$ with $\varkappa_k\to \infty$, rather than having power-of-two orders. 
	\end{remark}
	
	\begin{corollary}
		Let $q$ be any nonzero algebraic complex number that is not a root of unity and $\mathcal P_q = \{p_k\}_{k\in\NN}$ be the set of primes guaranteed by \autoref{algebraicNumFromRootsOfUnity}, and let $q_k\in\overline{\FF}_{p_k}$ be the corresponding even-order roots of unity. 
		There exists an isomorphism of fields $\uprod \overline{\FF}_{p_k} \cong \CC$ sending $\ulimit{k}q_k$ to $q$.
	\end{corollary}
	
	\begin{proof}
		By \autoref{uprodFields}, we know that an isomorphism $\uprod \overline{\FF}_{p_k} \cong \CC$ exists.
		Since, by construction, each $q_k$ satisfies the minimal polynomial for $q$, by \Los, so does $\uprod q_k$.
		Therefore, any such isomorphism maps $\uprod q_k$ to some Galois conjugate of $q$; upon composition with an appropriate Galois automorphism, we get the desired map.
	\end{proof}
	
	Fix a non-principal ultrafilter $\mathscr U$ on $\NN$. 
		Given transcendental (resp. algebraic non-root-of-unity) $q\in\CC$, fix an isomorphism $\uprod \CC \cong \CC$ (resp. $\uprod \overline{\FF}_{p_k} \cong \CC$) identifying $\ulim{k} q_k$ with $q$ for some sequence $\{q_k\}_{k\in\NN}$ of roots of unity of order $2\varkappa_k\to\infty$ in $\CC$ (resp. $\overline{\FF}_{p_k}$). 
		By \autoref{tensorCatsUltraprod}, $\uprod \VerF{\CC}$ (resp. $\uprod \VerF{\FFpk}$) is a $\CC$-linear tensor category via that isomorphism $\uprod \CC \to \CC$ (resp. $\uprod  \FFpk \cong \CC$).
		Let $\mathcal A_\infty(q)$ be its full subcategory generated by the object $X_1:=\ulim{k} L_1$ under tensor products, direct sums, summands.
		Even though $\uprod \VerF{\KK}$ depends on the ultrafilter, the choice of field isomorphisms, and the sequence $\varkappa_k$, the subcategory $\mathcal A_\infty(q)$ only depends on $q :=\ulim{k} q_k$, justifying the notation. The next proposition shows why.
		
	\begin{corollary}\label{ultraTL}
		There exists a monoidal equivalence $E:\TLf{q}{\CC}\overset{\sim}{\to} \mathcal A_\infty(q)$.
	\end{corollary}	
	
	\begin{proof}
		Let $f = \ulim{k} \overline{\CUP}$ and $g = \ulim{k}\overline{\CAP}$.
		By \Los, $g\circ f = -[2]_q \Id_{\id}$. 
		So, by the universal property in \autoref{univ-prop}, we get a tensor functor $E:\TLf{q}{\CC}\to \mathcal A_\infty$, satisfying $E(\obj{1}) = X_1$, $E(\CUP) = f$ and $E(\CAP) = g$. 
		It is clearly essentially surjective since $\mathcal A_\infty(q)$ is, by definition, generated by $X$.
		
		For any simple object $T_n$ of $\TLf{q}{\CC}$, and for any $k\in\NN$ large enough so that $\varkappa_k -2 \geq n$, we have simples $L_n$ in $\VerF{\KK}$ giving rise to an object $X_n := \ulim{k} L_n \subset X_1^{\otimes n}$ in $\mathcal A_\infty(q)$.
		Define $E(T_n)= X_n$.
		Now, by \Los, any two such objects $X_m$ and $X_n$ satisfy the same fusion rules as $T_m$ and $T_n$ since it is the same rule for $L_m$ and $L_n$ in $\Ver$ whenever $\varkappa - 2 \geq m+n$.
		Similarly, by \Los, $\Hom_{\mathcal A_\infty(q)}(X_n, X_m) \cong \CC \delta_{mn} $.
		It follows that $\{X_n\}_{n\in\NN}$ is a complete list of all simple objects of $\mathcal A_\infty(q)$, and that $E$ induces a bijection between the morphism spaces of $\mathcal A_{\infty}(q)$ and those of $\TLf{q}{\CC}$; thus, $E$ is an equivalence.	
	\end{proof}	
	
	\begin{notation}
		In what follows, we will often abuse notation and write $\uprod \Ver$, omitting the underlying fields when they are understandable from context or when the argument does not depend on the field. 
		For example, when we write $\uprod \Ver$ we will mean $\uprod \VerF{\CC}$ when we wish to study $\TLf{q}{\CC}$ when $q$ is transcendental and $\uprod \VerF{\FFpk}$ when we wish to study $\TLf{q}{\CC}$ when $q$ is algebraic non-root-of-unity.
		We will also omit from our notation the dependence of $\varkappa_k$ on $k$ in the context of ultraproducts, simply writing $\varkappa$.
	\end{notation}


\section{The Center of the Generic Temperley--Lieb Category}\label{sec:TLcenter}

\subsection*{Half-Braidings}	
	
	First, let us describe some examples of objects in the Drinfeld center of $\TLf{q}{\KK}$.
	As discussed in \autoref{sec:prelim}, due to the fact that $\TLf{q}{\CC}$ is braided, each object $X$ in $\TLf{q}{\KK}$ gives us two obvious objects in the Drinfeld center: $(X, \sigma_{X,-})$ and $(X, \sigma_{-,X}^{-1})$, where $\sigma$ is a braiding on $\TLf{q}{\KK}$. These give us the fully faithful functor 
	\begin{align*}
		D:\TLf{q}{\KK}\boxtimes\TLf{q}{\KK}^\rev &\to \Dr{\TLf{q}{\KK}}\\
		A\boxtimes B\hspace{.66in} &\mapsto (A\otimes B,  (\sigma_{A,-}\otimes \Id_B)\circ (\Id_A\otimes \sigma_{-,B}^{-1})).
	\end{align*}
	However, this functor is not essentially surjective; indeed, it is easy to exploit the $\ZZ/2\ZZ$--grading on $\TLf{q}{\KK}$ to construct an object of the center, that is not in the image of $D$.
	\begin{lemma}\label{TLgrading}
		$\TLf{q}{\KK}$ admits a $\ZZ/2\ZZ$--grading, which induces a functor $G:\Rep{\ZZ/2\ZZ} \to \Dr{\TLf{q}{\KK}}$.
	\end{lemma}
	
	\begin{proof}
		The grading is given by taking the full subcategory $\TLf{q}{\KK}_0$ with simples $T_{2n}$ (or equivalently the Karoubian subcategory generated by $\obj{2n}$) to be in degree $0$ and taking the full subcategory $\TLf{q}{\KK}_1$ with simples $T_{2n+1}$ to be in degree $1$. The rest follows by \autoref{grading-duality}.
	\end{proof}
	
	More explicitly, define $\xi_{T_n}: \obj{0}\otimes {T_n} \to {T_n} \otimes \obj{0}$ by $(-1)^n\cdot r^{-1}\circ l$, and extend linearly to define $\xi_X$ for all objects $X$.
	It is then easy to check that $(\obj{0}, \xi)$ generates a copy of $\Rep{\ZZ/2\ZZ}$ inside the center.
	
Thus, we can combine the functors $D$ and $G$ via the universal property of the Deligne tensor product.	
	
\begin{proposition}\label{functorZ}
	There exists a tensor functor
		$$Z:\TLf{q}{\KK}\boxtimes\TLf{q}{\KK}^\rev\boxtimes \Rep{\ZZ/2\ZZ} \to \Dr{\TLf{q}{\KK}}$$
		extending the functors $D:\TLf{q}{\KK}\boxtimes\TLf{q}{\KK}^\rev\to \Dr{\TLf{q}{\KK}}$ and $G:\Rep{\ZZ/2\ZZ}\to \Dr{\TLf{q}{\KK}}$. \hfill $\square$
\end{proposition}

	Our goal in this section is to prove that $Z$ is an equivalence of monoidal (but not braided!) categories when $\KK = \CC$ and $q$ is not a root of unity. 
	
	\begin{definition}\label{MWdef}
		In $\Dr{\TLf{q}{\CC}}$, we define $$M_{i,j} := Z(T_i\boxtimes T_j\boxtimes\id) = (T_i\otimes T_j,  (\sigma_{T_i,-}\otimes \Id_{T_j})\circ (\Id_{T_i}\otimes \sigma_{-,T_j}^{-1})),$$ $$W_{i,j} := Z(T_i\boxtimes T_j\boxtimes S) \cong M_{i,j}\otimes (\obj 0, \xi),$$
		where $S$ is the nontrivial simple objects of $\Rep{\ZZ/2\ZZ}$.
	\end{definition}
		
	\begin{remark}
		Considering other braidings of $\TLf{q}{\KK}$ does not give us any new half braidings on objects. Indeed, we showed that there are four braidings of $\TLf{q}{\KK}$ for $q\neq \pm 1$, one of which and its inverse give us the the objects $(X, \sigma_{X,-})$ and $(X, \sigma_{-,X}^{-1})$ in the center; replacing $\sigma$ with its inverse merely interchanges these two objects in the center, where as replacing $\sigma$ and $\sigma^{-1}$ with the other two braidings corresponds to the involutive autoequivalence on the center given by $-\otimes (\obj{0}, \xi)$. 
	\end{remark} 

We begin by a simple but very useful criterion for characterizing half-braidings in $\TLf{q}{\KK}$:
	
	\begin{lemma}\label{half-braid-criterion-TL}
		An isomorphism $\varphi_{\obj 1}: X\otimes \obj 1 \to \obj 1\otimes X$ in $\TLf{q}{\KK}$	uniquely determines a half-braiding for $X$ if and only if the following two diagrams commutes.
\[\begin{tikzcd}
	{X\otimes \obj 1\otimes \obj 1} & {\obj 1\otimes X\otimes \obj 1} & {\obj 1\otimes \obj 1\otimes X} \\
	{X\otimes \obj{0}} && {\obj{0}\otimes X}
	\arrow["{\varphi_{\obj 1}\otimes \Id_{\obj 1}}", from=1-1, to=1-2]
	\arrow["{\Id_X\otimes \CAP}"', from=1-1, to=2-1]
	\arrow["{ \Id_{\obj 1}\otimes \varphi_{\obj 1}}", from=1-2, to=1-3]
	\arrow["{\CAP\otimes \Id_X}", from=1-3, to=2-3]
	\arrow["{\varphi_{\obj{0}}}"', from=2-1, to=2-3]
\end{tikzcd}\quad
\begin{tikzcd}
	{X\otimes \obj 1\otimes \obj 1} & {\obj 1\otimes X\otimes \obj 1} & {\obj 1\otimes \obj 1\otimes X} \\
	{X\otimes \obj{0}} && {\obj{0}\otimes X}
	\arrow["{\varphi_{\obj 1}\otimes \Id_{\obj 1}}", from=1-1, to=1-2]
	\arrow["{ \Id_{\obj 1}\otimes \varphi_{\obj 1}}", from=1-2, to=1-3]
	\arrow["{\Id_X\otimes \CUP}", from=2-1, to=1-1]
	\arrow["{\varphi_{\obj{0}}}"', from=2-1, to=2-3]
	\arrow["{\CUP\otimes \Id_X}"', from=2-3, to=1-3]
\end{tikzcd}\]
	\end{lemma}
	
	\begin{proof}
		One direction is obvious: if $\varphi$ is a half-braiding, then, by the hexagon axiom, $\varphi_{\obj1\otimes\obj1} = (\Id_{\obj1}\otimes \varphi_{\obj 1})\circ(\varphi_{\obj1}\otimes \Id_{\obj1})$, so the diagrams commute by naturality of $\varphi$.
		For the converse, we first need to construct the braiding $\varphi$ from a single morphism $\varphi_{\obj 1}$.
		By \autoref{Kar-Dr}, it suffices to build a half-braiding on $\preTLf{q}{\KK}$ and extend it to the Cauchy completion in the canonical way prescribed therein. 
		This is possible because $\obj 1$ generates $\preTLf{q}{\KK}$ under tensor product and taking direct sums.
		More precisely, we define $\varphi_{\obj n}$ inductively in the only way possible to guarantee that the hexagon axiom holds, namely 
		$$\sigma_{\obj n} := (\Id_{\obj{n-1}}\otimes \varphi_{\obj 1})\circ (\varphi_{\obj{n-1}}\otimes \Id_{\obj 1}) = (\Id_{\obj1}\otimes \varphi_{\obj{n-1}})\circ (\varphi_{\obj 1}\otimes \Id_{\obj{n-1}}).$$
		Then, one defines $\varphi$ on direct sums in the obvious way: $\varphi_{\obj n \oplus \obj m}:= \varphi_{\obj n}\oplus \varphi_{\obj m}$.
		
		Now, it remains to check that $\varphi$ is a natural isomorphism.
		But observe that the diagrams above are precisely the naturality diagrams for the generating morphisms $\CUP$ and $\CAP$, and naturality in all other morphisms in $\preTLf{q}{\KK}$ follows by tensoring and composing these two diagrams.
		To see that $\varphi$ is an isomorphism, we build an a natural transformation from $\varphi_{\obj 1}^{-1}$ in the same fashion, which is easily checked to yield an inverse to $\varphi$.
		\end{proof}
		
		The same criterion (with the same proof) also hold for the fusion quotients $\Ver(\KK)$.
		
	\begin{lemma}\label{half-braid-criterion-fusion}
		An isomorphism $\varphi_{L_1}: X\otimes L_1  \to L_1\otimes X$ in $\VerF{\KK}$	uniquely determines a half-braiding for $X$ if and only if the following two diagrams commutes.
\[\begin{tikzcd}
	{X\otimes L_1\otimes L_1} & {L_1\otimes X\otimes L_1} & {L_1\otimes L_1\otimes X} \\
	{X\otimes {\id}} && {{\id}\otimes X}
	\arrow["{\varphi_{L_1}\otimes \Id_{L_1}}", from=1-1, to=1-2]
	\arrow["{\Id_X\otimes \overline{\CAP}}"', from=1-1, to=2-1]
	\arrow["{ \Id_{L_1}\otimes \varphi_{L_1}}", from=1-2, to=1-3]
	\arrow["{\overline{\CAP}\otimes \Id_X}", from=1-3, to=2-3]
	\arrow["{\varphi_{{\id}}}"', from=2-1, to=2-3]
\end{tikzcd}
\begin{tikzcd}
	{X\otimes L_1\otimes L_1} & {L_1\otimes X\otimes L_1} & {L_1\otimes L_1\otimes X} \\
	{X\otimes {\id}} && {{\id}\otimes X}
	\arrow["{\varphi_{L_1}\otimes \Id_{L_1}}", from=1-1, to=1-2]
	\arrow["{ \Id_{L_1}\otimes \varphi_{L_1}}", from=1-2, to=1-3]
	\arrow["{\Id_X\otimes \overline{\CUP}}", from=2-1, to=1-1]
	\arrow["{\varphi_{{\id}}}"', from=2-1, to=2-3]
	\arrow["{L_1\otimes \Id_X}"', from=2-3, to=1-3]
\end{tikzcd}\]
	\end{lemma}

	\subsection*{Filtration and Ultraproducts}
	The category $\TLf{q}{\KK}$ has a natural filtration due to the fact that it is generated by a single object $\obj{1}$. 
	Namely, for each $r, s\in\NN$, the filtered components $\TLf{q}{\KK}^{\leq r, s}$ is the full subcategory whose objects of the form $Z_1 \oplus Z_2 \oplus \dots \oplus Z_s$, where each $Z_i$ is a summand of some object $\obj{n}$ with $n\leq r$.
	When $q^2$ is of order $\varkappa$ in $\KK$, upon semisimplification, we get an induced filtration on the categories $\VerF{\KK}$.
	Said differently, the filtered components $\VerF{\KK}^{\leq r}$ are the full subcategories whose objects are of the form $Z_1 \oplus Z_2 \oplus \dots \oplus Z_s$, where each $Z_i$ is among $L_0, L_1, \dots, L_r$.
	These bifiltrations give rise to coarser filtrations $\cat^{\leq r} := \bigcup_{s\in\NN} \cat^{\leq r,s}$ where $\cat$ is either $\VerF{\KK}$ or $\TLf{q}{\CC}$. 
	
	Note that each of the filtered pieces $\Ver^{\leq r,s}$ or $\TLf{q}{\CC}^{\leq r,s}$ is not closed under taking direct sums or tensor products; they only have a partially defined additive and monoidal structures. With the coarser filtrations, the filtered components $\Ver^{\leq r}$ or $\TLf{q}{\CC}^{\leq r}$ are closed under direct sums, but still not under tensor products.
	We say that a functor $F:\cat^{\leq r}\to\mathcal D^{\leq r}$ \emph{intertwines the partially defined monoidal structures of} $\cat^{\leq r}$ and $\mathcal D^{\leq r}$ if, whenever $X\otimes Y$ exists in $\cat^{\leq r}$, we have that $F(X)\otimes F(Y)$ exists in $\mathcal D^{\leq r}$ and we have some \emph{unspecified} isomorphism $F(X\otimes Y)\cong F(X)\otimes F(Y)$ in $\mathcal D^{\leq r}$. 
	
	Given a filtered tensor category $\cat$ with filtered components $\cat^{\leq r}$, we get an induced filtration on its Drinfeld center: $\Dr{\cat}^{\leq r}$ is the full subcategory of $\Dr{\cat}$ consisting of those objects $(X,\varphi)$ where the underlying object $X$ lies in the corresponding filtered component $\cat^{\leq r}$. 
	Thus, we get a filtration on $\Dr{\TLf{q}{\KK}}$ and on $\Dr{\VerF{\KK}}$ with filtered components $\Dr{\TLf{q}{\KK}}^{\leq r,s}$ and $\Dr{\VerF{\KK}}^{\leq r,s}$, respectively.
		
	Having a filtration where each filtered component has finitely many objects is very useful when working with ultraproducts. 
	Indeed, the category $\uprod \Ver$ is way too large and we only care about a small part of it, namely $\mathcal A_\infty (q) \simeq \TLf{q}{\CC}$.
	But taking ultraproducts of only a given filtered component gives us a much more manageable category, eliminating the need to further restrict to smaller categories, as the next proposition shows.
	
	\begin{proposition}\label{filteredUltraprodTL}
	The functor $E$ of \autoref{ultraTL} restricted to $\TLf{q}{\CC}^{\leq r,s}$ is an equivalence
		$$E^{\leq r,s}:\TLf{q}{\CC}^{\leq r, s}\simeq \mathcal A_{\infty}^{\leq r,s} =\uprod \Ver^{\leq r, s}  $$ of $\CC$-linear categories intertwining their partially defined monoidal structures.
	\end{proposition} 
	
	\begin{proof}
		Suppose $Y = X_{j_1} \oplus \dots \oplus X_{j_t}$ with $t\leq s$ and $0\leq j_i\leq r$, where $X_m = E(T_m)$ in $\mathcal A_{\infty}^{\leq r,s}$.
		Then, $Y = \ulim{k} \left(T_{j_1} \oplus \dots \oplus T_{j_t}\right)$. Thus, $\uprod \Ver^{\leq r,s} \supseteq \mathcal A_\infty^{\leq r,s}$.
		
		Conversely, let $\{Y_k\}_{k\in\NN}$ be a sequence of objects $Y_k$ in $\Ver^{\leq r, s}$. 
		We wish to show that $\ulim{k} Y_k$ is in $\mathcal A_\infty^{\leq r, s}$.
		Observe that each $Y_k$ has only finitely many options: $s$ options for the number of summands, and each summand is chosen among the $r+1$ options $L_0, \dots, L_r$.
		Let $Q_1, Q_2, \dots, Q_n$ be an exhaustive list of all the options for each $Y_k$ (they are the same options for $k$ large enough).
		Let $S_i = \{k\in\NN \mid Y_k = Q_i\}$ for $i = 1, \dots, n$. 
		By \autoref{ultraPartitionLemma}, $S_i\in\mathscr U$ for exactly one $i$, say $i^*$.
		It follows that $\ulim{k} Y_k = \ulim{k} Q_{i^*} =: Q$ which is in $\mathcal A_\infty^{\leq r,s}$; indeed, if $Q_{i^*}  = L_{j_1}\oplus \dots \oplus L_{j_t}$, then $Q = X_{j_1} \oplus \dots \oplus X{j_t}$ and each $X_m$ is a summand of $X_1^{\otimes m}$. 
		Thus, $\uprod \Ver^{\leq r,s} \subseteq \mathcal A_\infty^{\leq r,s}$.
	\end{proof}
	
	\subsection*{The Center as a Limit}
	Taking Drinfeld center unfortunately does not commute with taking ultraproducts, but the filtration saves us again.

\begin{proposition}\label{filteredUltraCenter}
	For $r,s\geq 2$, we have an equivalences 
	\begin{align*}
		F^{\leq r,s}: \uprod \Dr{\Ver}^{\leq r,s} &\overset{\sim}{\to} \Dr{\TLf{q}{\CC}}^{\leq r,s}\\
		\ulimit{k}(Y_k, \varphi_k)&\mapsto\left(E^{-1}\left(\ulimit{k} Y_k\right),  E^{-1}\left(\ulimit{k}\varphi_k\right)\right)
	\end{align*} 
	of $\CC$-linear categories intertwining their partially defined monoidal structures. \end{proposition} 

\begin{proof}
	
	Suppose that $(Y_k, \varphi_k)_{k\in\NN}$ is a sequence whose terms are objects in the factors $\Dr{\Ver}^{\leq r,s}$ of the ultraproduct.
	Taking the ultraproduct $Y:=\ulim{k} Y_k$ of the underlying objects $Y_k$ as an object of $\uprod \Ver^{\leq r,s}$ gives an object of $\TLf{q}{\CC}^{\leq r,s}$ by \autoref{filteredUltraprodTL}.
	Similarly, $\varphi:= \ulim{k} (\varphi_k)_X$ gives us a family of morphisms $\varphi_X:X\otimes Y \to Y\otimes X$ in $\TLf{q}{\CC}^{\leq r,s}$, which glue to give a natural transformation since the naturality diagrams commute by \Los. 
	Since each of the $\varphi_k$ satisfies the hexagon axiom of a half-braiding in the whole center $\Dr{\Ver}$, {\Los} for the ultraproduct $\uprod \Ver \supset \mathcal A_\infty (q) \simeq \TLf{q}{\CC}$ implies that $\varphi$ is again a half-braiding on all of $\TLf{q}{\CC}$.
	Finally, again by \Los, we get a well-defined functor $[(Y_k,\varphi_k)_{k\in\NN}] \mapsto (Y,\varphi)$, which is intertwines the partially defined additive and monoidal structures of the two categories. 
	
	Conversely, suppose that $(Y,\varphi)$ is some object in $\Dr{\TLf{q}{\CC}}^{\leq r,s}$.
	By \autoref{filteredUltraprodTL}, we get a sequence (or rather an equivalence class of sequences) $(Y_k, \varphi_k)_{k\in\NN}$ of elements $Y_k$ and half-braidings $\varphi_k$ for them in $\Ver^{\leq r,s}$.
	\emph{A priori} those $\varphi_k$ are only guaranteed by {\Los} to be natural half-braidings on the filtered components, but may not braid $Y_k$ globally over all of $\Ver$ when $\varkappa \gg r$.
	However, \autoref{half-braid-criterion-fusion} assures us that this is not the case: as soon as the naturality and half-braiding condition is satisfied on $\Ver^{\leq 2,2}$, it will automatically hold for all of $\Ver$. 
	Thus, we get an assignment $(Y,\varphi)\mapsto [(Y_k, \varphi_k)_{k\in\NN}]$, which is functorial and compatible with the additive and monoidal structures by \Los.
\end{proof}

Composing the ultraproduct of the restrictions of the canonical equivalences $D_\varkappa:\Ver\boxtimes \Ver^{\rev} \overset{\sim}{\to} \Dr{\Ver}$ of \autoref{factorizableFusions} with $F^{\leq r,s}$, we get an equivalence 
\[\begin{tikzcd}
	{G^{\leq r,s}:} & {\uprod (\Ver\boxtimes \Ver^{\rev})^{\leq r,s}} && {\uprod \Dr{\Ver}^{\leq r,s}} && {\Dr{\TLf{q}{\CC}}^{\leq r,s}}
	\arrow["{\uprod D_\varkappa^{\leq r,s}}", from=1-2, to=1-4]
	\arrow["{F^{\leq r,s}}", from=1-4, to=1-6]
\end{tikzcd}\]

Since $G^{\leq r,s}$ just extends $G^{\leq r',s'}$ whenever $r'<r$ and $s'<s$, taking their direct limit, we get an equivalence of $\CC$-linear categories $$G = \lim\limits_{\longrightarrow} G^{\leq r,s}: \bigcup_{r,s\in\NN} \uprod (\Ver\boxtimes \Ver^{\rev})^{\leq r,s} \overset{\sim}{\to} \bigcup_{r,s\in\NN}\Dr{\TLf{q}{\CC}}^{\leq r,s} = \Dr{\TLf{q}{\CC}},$$
intertwining their monoidal structures (up to some unspecified isomorphisms). 

\begin{remark}\label{filtrationOnTensorProd}
	The filtration on $\Dr{\Ver}$ is somewhat obscured upon identifying it with $\Ver\boxtimes\Ver^\rev$ via the equivalence $D_\varkappa$ of \autoref{factorizableFusions}.
	By definition, the filtered component $\Dr{\Ver}^{\leq r,s}$ consists of the full subcategory of the center whose underlying objects lie in $\Ver^{\leq r,s}$.
	Therefore, the objects $(A\boxtimes B)$ in the filtered components $(\Ver\boxtimes\Ver^{\rev})^{\leq r,s}$ are exactly those such that $A\otimes B$ lies in the filtered component $\Ver^{\leq r,s}$. 
\end{remark}	

\begin{lemma}
	When $2s \geq \varkappa \geq r+2$, the simple objects of $\Ver\boxtimes \Ver^\rev$ in the filtered piece $(\Ver\boxtimes \Ver^\rev)^{\leq r,s}$ are precisely those of the form $L_i\boxtimes L_j$ or $L_{\varkappa-2-i}\boxtimes L_{\varkappa -2 -j}$ where $i+j\leq r$.
\end{lemma}

\begin{proof}
	The simple objects in $\Ver\boxtimes\Ver^\rev$ are those of the form $L_i\boxtimes L_j$.
	By \autoref{filtrationOnTensorProd}, $L_i\boxtimes L_j$ will be in the filtered component $(\Ver\boxtimes \Ver^\rev)^{\leq r,s}$ if and only if the decomposition of $L_i\otimes L_j$ from \autoref{fusionRules} lies in the filtered component $\Ver^{\leq r, s}$.
	This happens if and only if either $i+j\leq \varkappa-2$ and $i+j\leq r$ or else $i+j>\varkappa -2$ and $2\varkappa-4-(i+j)\leq r$. 
	Since $2s\leq \varkappa$, we never get too many summands in the decomposition than our filtered component allows.
	Moreover, when $\varkappa \geq r+2$, our two conditions reduce to the following: either $i+j\leq r$ or $(\varkappa-2-i)+(\varkappa-2-j)\leq r \iff i'+j' \leq r$, where $i' = \varkappa-2-i$ and $j' = \varkappa-2-j$, and the claim follows.
	\end{proof}

	\subsection*{Semisimplicity}
	
	\begin{lemma}\label{xiLemma}
		In $\Dr{\Ver}$, there exists an isomorphism 
		$$(L_0, \overline\xi) \cong (L_{\varkappa-2}\otimes L_{\varkappa-2}, (\overline\sigma_{L_{\varkappa-2},-}\otimes \Id_{L_{\varkappa-2}})\circ (\Id_{L_{\varkappa-2}}\otimes \overline\sigma_{-,L_{\varkappa-2}}^{-1}) ),$$
		where $\xi$ is the non-trivial half-braiding on $\obj 0$ via the grading of $\TLf{q}{\KK}$.
	\end{lemma}
	
	\begin{proof}
		First, note that the universal grading group $G$ of $\Ver$ is $\ZZ/2\ZZ$; indeed, it must be cyclic since $\Ver$ is $\otimes$-generated by $L_1$ and every element $g\in G$ must be of order $2$ since $L_n\otimes L_n$ contains $L_0\cong \id$ as a summand. 
		Moreover, we have a faithful $\ZZ/2\ZZ$--grading given by taking $L_n$ to be in degree $1$ for odd $n$ and $0$ for even $n$,
		Thus, by \autoref{FusionGradingVsHalfBraidings}, there are exactly two non-isomorphic objects in $\Dr{\Ver}$ with $\id$ as their underlying object in $\Ver$, namely $(\id, r^{-1}\circ l)$ (which is the monoidal unit of $\Dr{\Ver}$) and $(\id, \overline\xi)$.
		
		Now, the object 
		$$(L_{\varkappa-2}\otimes L_{\varkappa-2}, (\overline\sigma_{L_{\varkappa-2},-}\otimes \Id_{L_{\varkappa-2}})\circ (\Id_{L_{\varkappa-2}}\otimes \overline\sigma_{-,L_{\varkappa-2}}^{-1}) ) \cong D_\varkappa (L_{\varkappa-2}\boxtimes L_{\varkappa -2} )$$
		does have $L_{\varkappa-2}\otimes L_{\varkappa-2} \cong L_0 \cong  \id$ as its underlying object, and it cannot be the monoidal unit since $D_{\varkappa}$ is an equivalence by \autoref{factorizableFusions}. Hence, it must be isomorphic to $(\id, \overline\xi)$ and the lemma follows.
	\end{proof}
	
	\begin{lemma}
				In $\Dr{\TLf{q}{\CC}}$, there exists isomorphisms
				$$M_{i,j}\cong G\left(\ulimit{k} (L_i\boxtimes L_j)\right)\quad \text{and}\quad W_{i,j} \cong G\left(\ulimit{k}(L_{\varkappa-2-i}\boxtimes L_{\varkappa-2-j})\right).$$
	\end{lemma}
	
	\begin{proof}
		For $r\geq i+j$ and $s> \min(i,j)$, we have 
		\begin{align*}
			G\left(\ulimit{k} L_i\boxtimes L_j\right) &\cong (G^{\leq r,s})\left(\ulimit{k} L_i\boxtimes L_j\right) = (F^{\leq r,s})\circ \left(\uprod D_{\varkappa}^{\leq r,s}\right) \left(\ulimit{k} L_i\boxtimes L_j\right)\\
			&\cong (F^{\leq r,s})\left(\ulimit{k} (L_i\otimes L_j, (\overline\sigma_{L_i,-}\otimes \Id_{L_j})\circ (\Id_{L_i}\otimes \overline\sigma_{-,L_j}^{-1}) )\right)\\
			&= \left(E^{-1}\left(\ulimit{k} (L_i\otimes L_j)\right),  E^{-1}\left(\ulimit{k}(\overline\sigma_{L_i,-}\otimes \Id_{L_j})\circ (\Id_{L_i}\otimes \overline\sigma_{-,L_j}^{-1}) )\right)\right)\\
			&\cong (T_i\otimes T_j, (\sigma_{T_i,-}\otimes \Id_{T_j})\circ (\Id_{T_i}\otimes \sigma_{-,T_j}^{-1}) )\\
			&\cong Z(T_i\boxtimes T_j\boxtimes \id) = M_{i,j}.
		\end{align*}
		Next, observe that for all $r,s\in\NN$, we have 
		\begin{align*}
			G\left(\ulimit{k} L_{\varkappa-2}\boxtimes L_{\varkappa-2}\right) &\cong (G^{\leq r,s})\left(\ulimit{k} L_{\varkappa-2}\boxtimes L_{\varkappa-2}\right)\\
			&\cong (F^{\leq r,s})\left(\ulimit{k} (L_{\varkappa-2}\otimes L_{\varkappa-2}, (\overline\sigma_{L_{\varkappa-2},-}\otimes \Id_{L_{\varkappa-2}})\circ (\Id_{L_{\varkappa-2}}\otimes \overline\sigma_{-,L_{\varkappa-2}}^{-1}) )\right)\\
			&\cong \left(E^{-1}\left(\ulimit{k} L_0 \right),  E^{-1}\left(\ulimit{k} \overline \xi \right)\right)\cong (\obj 0, \xi),
			\end{align*}
			where the third isomorphism is due to \autoref{xiLemma}.
			Finally, since $G$ intertwines the monoidal strucures, we have
			\begin{align*}
				W_{i,j} &\cong M_{i,j} \otimes (\obj 0, \xi) \cong G\left(\ulimit{k}(L_i\boxtimes L_j)\right)\otimes G\left(\ulimit{k}(L_{\varkappa-2}\boxtimes L_{\varkappa-2})\right)\\
				&\cong G\left(\ulimit{k}(L_i\boxtimes L_j)\otimes \ulimit{k}(L_{\varkappa-2}\boxtimes L_{\varkappa-2})\right)\\
				&\cong G\left(\ulimit{k}((L_i\otimes L_{\varkappa - 2})\boxtimes (L_j \otimes L_{\varkappa -2}))\right) \\
				&\cong G\left(\ulimit{k}(L_{\varkappa -2 -i}\boxtimes L_{\varkappa - 2 -j}) \right).
			\end{align*}
	\end{proof}
	
	\begin{theorem}\label{centerSemisimplicity}
		 For $q\in\CC$ not a root of unity, $\Dr{\TLf{q}{\CC}}$ is semisimple. $\{M_{i,j}, W_{i,j} \mid i,j\in\NN\}$ is a complete and irredundant  list of its isomorphism classes of simple objects. 
	\end{theorem}
	
	\begin{proof} 
	Semisimplicity follows directly by \autoref{ultraSemisimplicity} since the functor $G$ identifies $\Dr{\TLf{q}{\CC}}$ with a (Karoubian, Krull-Schmidt) subcategory of $\uprod \Ver \boxtimes \Ver^\rev$.
	 
		Suppose $V$ is some simple object of the center.
		$V\cong G\left(\ulim{k} (L_{i_k} \boxtimes L_{j_k})\right)$ for sequences $\{L_{i_k}\}_{k\in\NN}$ and $\{L_{j_k}\}_{k\in\NN}$ of objects in $\Ver$ and $\Ver^\rev$, which must be simple by \autoref{ultraSemisimplicity}. 
		Since the underlying object of $V$ in $\TLf{q}{\CC}$ must correspond (under $E$) to $\ulim{k} (L_{i_k}\otimes L_{j_k}) \cong \ulim{k} (L_{|i_k-j_k|}\oplus \dots)$, we see that $|i_k - j_k|$ must be $\mathscr U$-almost always a constant as $k\to\infty$.
		Let $r\in\NN$ be the unique positive integer such that $V$ lies in $\Dr{\TLf{q}{\CC}}^{\leq r}$ but not in $\Dr{\TLf{q}{\CC}}^{\leq r-1}$. 
		Since $G$ is compatible with the filtrations, it must be the case that $L_{i_k}\boxtimes L_{j_k}$ lies in $(\Ver \boxtimes \Ver^\rev)^{\leq r}$ but not in $(\Ver \boxtimes \Ver^\rev)^{\leq r-1}$ for $\mathscr U$-almost all $k$.
		From \autoref{fusionRules}, we see that either $i_k+j_k = r$ or $(\varkappa - 2 - i_k) + (\varkappa - 2 -  j_k) = r$. 
		It follows that either $i_k$ and $j_k$ are both $\mathscr U$-almost always constants or else $(\varkappa - 2 - i_k)$ and $(\varkappa - 2 -  j_k)$ are both $\mathscr U$-almost always constants as $k\to\infty$.
		In the former case, we get $V\cong M_{i,j}$ and in the latter we get $V\cong W_{i,j}$.
		Finally, the existence of an isomorphism between any two of these simple objects will imply (by \Los) the existence of isomorphisms between non-isomorphic simples in $\Ver\boxtimes\Ver^\rev$, which is impossible, so our list is irredundant.
	\end{proof}

		\begin{theorem}\label{EQUIVALENCE}
		For $q\in\CC$ not a root of unity, the tensor functor 
		$$Z:\TLf{q}{\CC}\boxtimes\TLf{q}{\CC}^\rev\boxtimes \Rep{\ZZ/2\ZZ} \to \Dr{\TLf{q}{\CC}}$$
		of \autoref{functorZ} is an equivalence.
	\end{theorem}
	
	\begin{proof}
		Both the domain and codomain of $Z$ are semisimple and linear over an algebraically closed field (so Schur's Lemma holds).
		Thus, to see that $Z$ is an equivalence, it is enough to show that $Z$ is a bijection on simples which is clear from \autoref{MWdef} and \autoref{centerSemisimplicity}.
	\end{proof}
	
	\begin{corollary}\label{centerFusion}
		For $q\in\CC$ not a root of unity, we have the following fusion rules for $\Dr{\TLf{q}{\ZZ}}$:
		$$M_{i,j}\otimes M_{i',j'} \cong W_{i,j}\otimes W_{i',j'}\cong \bigoplus_{n = 0}^{\min(i',j')}\bigoplus_{m = 0}^{\min(i,j)} M_{i+j-2m,\, i'+j'-2n}$$
		$$\text{and}\quad  W_{i,j}\otimes M_{i',j'}\cong \bigoplus_{n = 0}^{\min(i',j')}\bigoplus_{m = 0}^{\min(i,j)} W_{i+j-2m,\, i'+j'-2n}.$$
		\hfill $\square$
	\end{corollary}

	Since each of the categories $\TLf{q}{\CC}$, $\TLf{q}{\CC}^\rev$, and $\Rep{\ZZ/2\ZZ}$ has a faithful $\ZZ/2\ZZ$--grading, we get a grading on the center.
	
	\begin{lemma}
		The category $\Dr{\TLf{q}{\CC}}\simeq \TLf{q}{\CC}\boxtimes \TLf{q}{\CC}\boxtimes \Rep{\ZZ/2\ZZ}$ has a faithful $(\ZZ/2\ZZ)^{\oplus 3}$--grading given by $M_{i,j}$ and $W_{i,j}$ being in the degrees $(\overline i, \overline j, 0)$ and $(\overline i,\overline  j, 1)$, respectively, where $\overline n:=n\mod 2. \hfill \square$
	\end{lemma}
	
	Consider the bicharacter $\gamma:(\ZZ/2\ZZ)^{\oplus 3}\times (\ZZ/2\ZZ)^{\oplus 3} \to \CC^\times$ given by $\gamma\left((i,j,k), (i', j', k')\right) = (-1)^{k(i'+j')}$. 
	We can twist the braiding on $\TLf{q}{\CC}\boxtimes \TLf{q}{\CC}\boxtimes \Rep{\ZZ/2\ZZ}$ by $\gamma$ to get a new braided tensor category, which turns out to be equivalent to the center as \emph{braided} tensor categories.
	
	\begin{theorem}\label{mainThm}
		For $q\in\CC$ not a root of unity, the canonical tensor functor 
		$$Z:\bigg(\TLf{q}{\CC}\boxtimes \TLf{q}{\CC}^\rev\boxtimes \Rep{\ZZ/2\ZZ}\bigg)^\gamma \to \Dr{\TLf{q}{\CC}}$$
		is an equivalence of braided tensor categories.
	\end{theorem}
	
	\begin{proof}
		We need to show that the equivalence $Z$ of \autoref{EQUIVALENCE} intertwines the braidings of the two categories on their monoidal generators, namely $T_i\boxtimes T_j\boxtimes \id$ and $\id\boxtimes \id \boxtimes S$ versus $M_{i,j}$ and $W_{0,0}$.
		Let $\sigma$ be the braiding of $\TLf{q}{\CC}$ as before, and let $\sigma^{\mathcal Z}$ and $\sigma^\gamma$ be the braiding of $\Dr{\TLf{q}{\CC}}$ and $\TLf{q}{\CC}\boxtimes \TLf{q}{\CC}^\rev\boxtimes \Rep{\ZZ/2\ZZ}$, respectively.
		Since the canonical functor $D:\TLf{q}{\CC}\boxtimes \TLf{q}{\CC}^\rev \to \Dr{\TLf{q}{\CC}}$ is automatically braided, and $Z$ extends $D$ by definition, it is clear that $Z$ intertwines the braidings $\sigma^\gamma_{T_i\boxtimes T_j\boxtimes \id, T_{i'}\boxtimes T_{j'}\boxtimes \id}$ (which are unaffected by the $\gamma$-twist) with $\sigma^\mathcal Z_{M_{i,j}, M_{i',j'}}$.
		Similarly, as $Z$ extends the braided functor $\Rep{\ZZ/2\ZZ}\to \Dr{\TLf{q}{\CC}}$, it is clear that it intertwines $\sigma^\gamma_{\id\boxtimes\id\boxtimes S, \id\boxtimes\id\boxtimes S}$ (which is again unaffected by the $\gamma$-twist) with $\sigma^\mathcal Z_{W_{0,0}, W_{0,0}}$.
		So, it is enough to check that it also intertwines the braidings $\sigma^\gamma_{\id\boxtimes \id\boxtimes S , T_i, T_j, \id}$ and $\sigma^\gamma_{T_i, T_j, \id,\id\boxtimes \id\boxtimes S}$ with $\sigma^\mathcal Z_{W_{0,0}, M_{i,j}}$ and $\sigma^\mathcal Z_{M_{i,j}, W_{0,0}}$, respectively. 
		
		First,
		$$\sigma^\mathcal Z_{M_{i,j}, W_{0,0}} := \varphi_{\id}:(T_i\otimes T_j, \varphi)\otimes (\id, \xi)\to (\id,\xi)\otimes (T_i\otimes T_j, \varphi),$$
		where $\varphi$ is the half-braiding $(\sigma_{T_i,-}\otimes \Id_{T_j})\circ (\Id_{T_i}\otimes \sigma_{-,T_j}^{-1})$. 
		But $\varphi_\id$ is trivial (i.e. a composition of left and right unitors) by the basic properties of braidings.
		On the other hand, 
		$$\sigma^\gamma_{T_i, T_j, \id,\id\boxtimes \id\boxtimes S} := \gamma\left((\overline i,\overline j,0), (0,0,1)\right) \,\sigma_{T_i, \id}\boxtimes \sigma^{-1}_{\id, T_j}\boxtimes \sigma^{\Rep{\ZZ/2\ZZ}}_{\id, S} = \sigma_{T_i, \id}\boxtimes \sigma^{-1}_{\id, T_j}\boxtimes \sigma^{\Rep{\ZZ/2\ZZ}}_{\id, S},$$
		which again consists only of left and right unitors. Applying $Z$ to the latter we get two morphisms in the center with the same source and target, which consist entirely of unitors and hence coincide by Mac Lane's Coherence Theorem. 
		
		Next, we have 
		$$\sigma^\mathcal Z_{W_{0,0},M_{i,j}}:= \xi_{T_i\otimes T_j}:(\id,\xi)\otimes (T_i\otimes T_j, \varphi) \to (T_i\otimes T_j,\varphi)\otimes (\id, \xi),$$
		which is given by $(-1)^{i+j}$ (again multiplied by a map consisting of unitors, which we can ignore due to the same Coherence Theorem argument). On the other hand,
		$$\sigma^\gamma_{\id\boxtimes \id\boxtimes S, T_i, T_j, \id} := \gamma\left((0,0,1), (\overline i, \overline j, 0)\right) \,\sigma_{\id, T_i}\boxtimes \sigma^{-1}_{T_j, \id}\boxtimes \sigma^{\Rep{\ZZ/2\ZZ}}_{S, \id} = (-1)^{i+j}.$$
	\end{proof}


\section{Some Exceptional Cases and Concluding Remarks}\label{sec:exceptionalQ}

	\subsection*{The Classical Limit, $q = 1$}
	
	The ultraproduct techniques we used so far still tells us something interesting about the center at $q=1$, but it does not help us fully understand it. 	
	
	\begin{proposition}\label{VerpApprox}
		The category $\TLf{1}{\CC}$ is equivalent to the tensor subcategory of $\uprod \Verp$ generated by the object $\ulim{p} L_1$.
	\end{proposition}
	
	\begin{proof}
		The same proof as that of \autoref{ultraTL} applies.
	\end{proof}
	
	The main obstacle, however, is that the Drinfeld center of $\mathbf{Ver}_{p_k}$ is complicated; indeed, $\Dr{\Verp}$ is non-semisimple, and it is not well-understood to the author's knowledge. 
	
	Instead, an algebro-geometric approach is more useful here.
	In the $q=1$ case, it is well-known that the Temperley--Lieb category is equivalent to the category $\Rep{SL_2}$ of complex representations of the algebraic group $SL_2$.
	The following results on the center of $\Rep{G}$ for an algebraic group $G$ are folklore.

	\begin{proposition}\label{sheafyCenter}
		For a connected algebraic group $G$, we have a canonical equivalence of tensor categories between the center $\Dr{\Rep{G}}$ and the category $\mathbf{Coh}_G^\mathrm{fin.supp.}(G)$ of $G$-equivariant coherent sheaves on $G$ supported on finitely many points.  	\end{proposition}
	
	\begin{proof}[Proof Sketch]
		 The category $\Rep{G}$ is equivalent as a tensor category to the category $\mathcal O(G)\mathbf{-comod}$ of finite dimensional comodules over the ring of regular functions $\mathcal O(G)$. 
	\cite[Proposition 7.15.3]{EGNO} says that the center of the category $K\mathbf{-comod}$ for a Hopf algebra $K$ is monoidally equivalent to the category $\mathcal{YD}_{K}^K$ of finite-dimensional Yetter--Drinfeld modules over $K$. In particular, $\Dr{\Rep{G}}\simeq \YD{G}$. 
	The category $\YD{G}$ is naturally equivalent as a tensor category to the category $\mathbf{Coh}_G^{f.s.}(G)$ of $G$-equivariant coherent sheaves over $G$ supported on finitely many points, where $G$ acts by conjugation.
	Since we are considering coherent sheaves that are equivariant with respect to the conjugation action of $G$, the support must be a union of conjugacy classes of $G$, which is finite if and only if the support lies entirely in the center of $G$ assuming $G$ is connected. 
	\end{proof}
	
	\begin{corollary}
		We have a $Z(G)$-grading on the center $\Dr{\Rep{G}} \simeq \bigoplus_{g\in Z(G)} \mathcal Z_{g},$
		with mutually equivalent (as abelian categories) graded components.
	\end{corollary}
	
	\begin{proof}
		$\mathbf{Coh}_G^\mathrm{fin.supp.}(G)$ is a direct sum of categories $\mathcal Z_g$ of $G$-equivariant coherent sheaves supported within $\{g\}$ for $g\in Z(G)$.
		Translation by $g$ induces an equivalence between $\mathcal Z_e$ and $\mathcal Z_g$.  
	\end{proof}
	
	\begin{corollary}\label{sheafy-center-sl2}
		Taking $G=SL_2$, we have a monoidal equivalence between $\Dr{\TLf{1}{\CC}}$ and $\mathbf{Coh}_{SL_2}^{\mathrm{supp}\subset\{\pm1\}}(SL_2)$ of $SL_2$--equivariant (with respect to the conjugation action) coherent sheaves on $SL_2$ supported $\{\pm 1\}$.
	\end{corollary}
	
	\begin{proposition}\label{sheavesVsFields}
		Objects in $\mathcal Z_e$ are in one-to-one correspondence with rational $G$-representations $V$ equipped with a $G$-equivariant linear map $\mathfrak g^*\to \End_\CC(V)$ whose image consists of mutually commuting nilpotent endomorphisms, where $\mathfrak g^*$ has the coadjoint $G$-action.
	\end{proposition}
	
	\begin{proof}
	Let $\mathfrak m$ be the maximal ideal of $\mathcal O(G)$ corresponding to the group unit. 
	Then, a coherent sheaf supported on $\{e\}$ is the same data as a finitely generated module $V$ over $\mathcal O(G)_{\mathfrak m}$ which is annihilated by $\mathfrak m^N$ for some $N>0$. 
	Moreover, there is an isomorphism between the $\mathfrak m$-adic completion of $O(G)$ and the completion of ${\mathcal O}(\mathfrak m/\mathfrak m^2)\cong \Sym(\mathfrak g^*)$. 
	Thus, $V$ acquires a natural action of the completion of $\Sym(\mathfrak g^*)$. 
	Because $V$ is annihilated by $\mathfrak m^N$, this action must factor through $\Sym(\mathfrak g^*)/\mathfrak m_0^N$, where $\mathfrak m_0$ is the irrelevant ideal generated by $\mathfrak g^*$.
	Consequently, the module structure on $V$ determines (by restricting to degree-one elements) a linear map $\theta:\mathfrak g^*\to \End_\CC(V)$ such that $\theta(x)$ is nilpotent for all $x\in\mathfrak g^*$. 
	Conversely, a map $\theta:\mathfrak g^*\to \End_\CC(V)$ whose image consists of mutually commuting nilpotent endomorphisms lifts to an action of $\Sym(\mathfrak g^*)$ (and hence of $\mathcal O(G)/\mathfrak m^N$ for some $N>0$). 
	Requiring that our sheaves be $G$-equivariant is the same as requiring that $V$ is a rational $G$-representation and $\theta$ is $G$-equivariant with respect to the coadjoint action on $\mathfrak g^*$. 
	\end{proof}
	
	\begin{remark}
	As an example, some obvious objects in the $\mathcal Z_e$ are given by $\mathcal O(G)/\mathfrak m^n$ for $n>0$, which are all indecomposable.
	However, this is far from being a complete classification.
	Indeed, by \autoref{sheavesVsFields}, classifying coherent sheaves on $G$ supported on the group unit is equivalent to classifying $\dim_\CC(\mathfrak g^*)$-tuples of commuting nilpotent matrices, which is a wild problem in the formal sense of \cite{Dro} whenever $\dim_\CC(\mathfrak g^*)\geq 2$.
	Even, under extra restriction of equivariance, the classification problem for indecomposables in $\mathcal Z_e$ seems likely to be a wild problem as well. 
	\end{remark}

	\subsection*{The case of $q=-1$}
	
	For this case, our plan is to use a cocycle twist of the $q=1$ case to compute its center.
	For what follows, we fix an abelian 3-cocycle $(\omega,\gamma)\in Z_{\mathrm{ab}}^3(\ZZ/2\ZZ, \KK^\times)$ given by $\omega(g_1,g_2,g_3) = (-1)^{g_1g_2g_3}$ and $\gamma(g_1,g_2) = (\sqrt{-1})^{g_1g_2}$.
		
	\begin{proposition}\label{qTwistedTL}
		 There exists a braided equivalence of tensor categories $\TLf{q}{\KK}^{(\omega,\gamma)}\simeq \TLf{-q}{\KK}$.
	\end{proposition}
	
	\begin{proof}
		The object $\obj 1$ in $\TLf{q}{\KK}^{(\omega,\gamma)}$ together with the maps $f = \CUP$ and $g = -\CAP$ satisfy the condition of \autoref{univ-prop} of the universal property of $\TLf{-q}{\KK}$.
		Thus, we get a canonical tensor functor $\TLf{-q}{\KK} \to \TLf{q}{\KK}^{(\omega,\gamma)}$ given by $\obj 1\mapsto \obj 1$, $\CUP\mapsto \CUP$ and $\CAP\mapsto -\CAP$.
		It is clearly an equivalence (in fact, an isomorphism) and it intertwines every braiding on $\TLf{-q}{\KK}$ with some braiding on $\TLf{q}{\KK}$ (by choosing the appropriate square root of $q$). 
	\end{proof}
	
	\begin{corollary}
		$\Dr{\TLf{-q}{\KK}}\simeq \Dr{\TLf{q}{\KK}}^{(\omega,\gamma)}$. Thus, $\Dr{\TLf{-1}{\KK}}\simeq \mathbf{Coh}_{SL_2}^{\mathrm{supp}\subset\{\pm1\}}(SL_2)^{(\omega,\gamma)}$. 
	\end{corollary}
	
	\begin{proof}
		This is a direct consequence of \autoref{TwistedCenter}, \autoref{qTwistedTL}, and \autoref{sheafy-center-sl2}.
	\end{proof}
	
	\begin{remark}
		One can also approximate $\Dr{\TLf{-1}{\KK}}$ as an ultraproduct of $\Verp^{(\omega,\gamma)}\simeq \Verp^+\boxtimes \mathbf{Sem}$ as in \autoref{VerpApprox}, where $\Verp^+$ is the tensor subcategory of $\Verp$ generated by $L_2$ and $\mathbf{Sem}\simeq\sVect^{(\omega,\gamma)}$ is the modular tensor category of \emph{semions}. This is because the operation of  twisting associators and braidings commutes with taking ultraproducts by \Los.
	\end{remark}

	\subsection*{The Root of Unity Case}
	
	Suppose now that $q$ is a primitive root of unity.
	\autoref{ultraTL} still holds in this case, but the approximating categories $\TLf{q_k}{\FFpk}$ are much less well-behaved. 
	In particular, the sequence $q_k$ must also be roots of unity in $\FFpk$ for almost all $k$; so, the approximating categories are almost never semisimple. 
	Thus, the problem of understanding the center in this case is more complicated with some interesting variants. 
	
	\begin{problem}
		Compute the Drinfeld center of the Temperley-Lieb at roots of unity and the canonical functor from its semisimplification to the centers of the semisimplified categories $\Ver$.
	\end{problem}

	Another interesting variant of this problem which may lead to more non-trivial half-braidings was suggested by Ostrik (unpublished).
	
	\begin{problem}
		Compute the center of the derived category of the Temperley--Lieb category at a root of unity. 
	\end{problem}

	\subsection*{Other Lie Types}
	
	Fix a semisimple lie algebra $\mathfrak g$ with weight and root lattices $P$ and $Q$ respectively. 
	The category $U_q(\mathfrak g)\mathbf{-mod}$ of finite dimensional type-I representations of $U_q(\mathfrak g)$ admits a faithful $P/Q$ grading. 
	Thus, we have a canonical functor $G:\Rep{P/Q}\to \Dr{U_q(\mathfrak g)\mathbf{-mod}}$ by \autoref{grading-duality}. 
	Moreover, the braiding on $U_q(\mathfrak g)\mathbf{-mod}$ gives us the canonical functor $D: U_q(\mathfrak g)\mathbf{-mod}\boxtimes U_q(\mathfrak g)\mathbf{-mod}^\rev \to \Dr{U_q(\mathfrak g)\mathbf{-mod}}$. 
	
	\begin{conjecture}
	There exists a $(P/Q)^{\oplus 3}$ bicharacter $\gamma$ such that the canonical functor $$Z:\bigg(U_q(\mathfrak g)\mathbf{-mod}\boxtimes U_q(\mathfrak g)\mathbf{-mod}^\rev \boxtimes \Rep{P/Q} \bigg)^\gamma \to \Dr{U_q(\mathfrak g)\mathbf{-mod}}$$
	is equivalences of braided tensor categories whenever $q\in\CC$ is not a root of unity.
	\end{conjecture}
	
	The same ultraproduct techniques should be applicable in this setting. 
	One gets modular tensor categories by setting $q$ to some appropriate roots of unity and semisimplifying. 
	The ultraproduct of the resulting categories contains $U_q(\mathfrak g)\mathbf{-mod}$, and, considering appropriate filtrations, one gets equivalences on the filtered pieces.
	The main difficulty is that the filtrations on the centers of these fusion categories is more complicated as the affine Weyl groups become more complex. 
	We leave this investigation for later work.

	\subsection*{The Crystal Limit, $q\to 0$}
	
	In a previous joint work with Stroi\'nski \cite{AS}, we studied a renormalization of $\TLf{q}{\KK}$ which allows setting $q=0$. The resulting category can be defined exactly as in \autoref{TL-def} except that the relations are modified into
	
				\[
			\begin{tikzpicture}
	\begin{pgfonlayer}{nodelayer}
		\node [style=none] (0) at (-1.5, 0.75) {};
		\node [style=none] (1) at (-1.25, 1) {};
		\node [style=none] (2) at (-1, 0.75) {};
		\node [style=none] (3) at (-2, 0.5) {};
		\node [style=none] (4) at (-1.75, 0.25) {};
		\node [style=none] (5) at (-1.5, 0.5) {};
		\node [style=none] (6) at (-1, 0) {};
		\node [style=none] (7) at (-2, 1.25) {};
		\node [style=none] (8) at (-0.5, 0.75) {$=$};
		\node [style=none] (9) at (0, 0.75) {0};
		\node [style=none] (11) at (1.5, 0.75) {};
		\node [style=none] (12) at (1.25, 1) {};
		\node [style=none] (13) at (1, 0.75) {};
		\node [style=none] (14) at (2, 0.5) {};
		\node [style=none] (15) at (1.75, 0.25) {};
		\node [style=none] (16) at (1.5, 0.5) {};
		\node [style=none] (17) at (1, 0) {};
		\node [style=none] (18) at (2, 1.25) {};
		\node [style=none] (19) at (0.5, 0.75) {$=$};
	\end{pgfonlayer}
	\begin{pgfonlayer}{edgelayer}
		\draw [style=thickstrand, bend left=45] (0.center) to (1.center);
		\draw [style=thickstrand, bend right=45] (2.center) to (1.center);
		\draw [style=thickstrand, bend right=45] (3.center) to (4.center);
		\draw [style=thickstrand, bend left=45] (5.center) to (4.center);
		\draw [style=thickstrand] (5.center) to (0.center);
		\draw [style=thickstrand] (6.center) to (2.center);
		\draw [style=thickstrand] (3.center) to (7.center);
		\draw [style=thickstrand, bend right=45] (11.center) to (12.center);
		\draw [style=thickstrand, bend left=45] (13.center) to (12.center);
		\draw [style=thickstrand, bend left=45] (14.center) to (15.center);
		\draw [style=thickstrand, bend right=45] (16.center) to (15.center);
		\draw [style=thickstrand] (16.center) to (11.center);
		\draw [style=thickstrand] (17.center) to (13.center);
		\draw [style=thickstrand] (14.center) to (18.center);
	\end{pgfonlayer}
\end{tikzpicture}
			\]
			and
			\[
\begin{tikzpicture}
	\begin{pgfonlayer}{nodelayer}
		\node [style=none] (1) at (0, 1) {};
		\node [style=none] (3) at (-0.5, 0.5) {};
		\node [style=none] (4) at (0, 0) {};
		\node [style=none] (5) at (0.5, 0.5) {};
		\node [style=none] (6) at (1.65, 0.5) {$=\Id_{{\obj{0}}}.$};
	\end{pgfonlayer}
	\begin{pgfonlayer}{edgelayer}
		\draw [style=thickstrand, bend right=45] (3.center) to (4.center);
		\draw [style=thickstrand, bend left=45] (5.center) to (4.center);
		\draw [style=thickstrand, bend right=45] (5.center) to (1.center);
		\draw [style=thickstrand, bend left=45] (3.center) to (1.center);
	\end{pgfonlayer}
\end{tikzpicture}
			\]
			
			The Cauchy completion of the resulting category, which we denote $\mathbf{CrysTL}$, is monoidally equivalent to the linearized category of $\mathfrak{sl}_2$--crystals.
			Let us use the same notation for the objects of $\mathbf{CrysTL}$ as we did for $\TLf{q}{\KK}$; in particular, the object corresponding to $n$ dots is denoted $\obj n$ and the simple objects are denoted $T_0, T_1, T_2, \dots$.
			Note that, \autoref{half-braid-criterion-TL} still holds in $\mathbf{CrysTL}$ with the exact same proof. 
			
			The category $\mathbf{CrysTL}$ is not braided or rigid, but it is still semisimple with the same Grothendieck ring as the generic Temperley--Lieb category. In particular, we still have a $\ZZ/2\ZZ$--grading, so we get a braided tensor functor $G:\Rep{\ZZ/2\ZZ}\to \Dr{\mathbf{CrysTL}}$ which gives us two distinct half-braidings on $\id$.
			
			\begin{conjecture}\label{crystal-conj}
				The canonical functor $G:\Rep{\ZZ/2\ZZ}\to \Dr{\mathbf{CrysTL}}$ is a monoidal equivalence.
			\end{conjecture}
			
			We give some evidence for this claim presently.

			\begin{enumerate}
				\item Let $\varphi$ be a half-braiding on $\obj 0 \cong \id$. 
				Then, $\varphi_{\obj 1}:\obj 0\otimes \obj 1\to \obj 1 \otimes \obj 0$ is of the form $\lambda \cdot r^{-1}l$ for some scalar $\lambda\in\KK^\times$. 
				Naturality of $\varphi$ with respect to $\CAP:\obj 1\otimes \obj 1 \to \obj 0$ along with the hexagon axiom imply that $\lambda = \pm 1$ each of which giving a distinct half-braiding by \autoref{half-braid-criterion-TL}; they generate the image of the functor $G$.
				\item Next, observe that the data of a half-braiding on $\obj 0^{\oplus n}$ is equivalent to the data of an $n$-dimensional representation of the grading group $\ZZ/2\ZZ$, which is a direct sum of one-dimensional representations.
			It follows that the only half-braidings $\obj 0^{\oplus n}$ admits are direct sums of half-braidings on $\obj 0$. 
			\item $\obj{m}$ admits no half-braidings in $\mathbf{CrysTL}$ for all positive integers $m$. To see this, suppose that $\varphi$ is a half-braiding on $\obj m$ for a positive integer $m$. 
				So, we have an isomorphism $$\varphi_{\obj 1}: \obj{m}\otimes \obj {1}\to \obj{1}\otimes \obj{m},$$
				which by naturality with respect to $\CAP:\obj{1}\otimes \obj{1}\to \obj{0}$ and the hexagon axiom satisfies
				
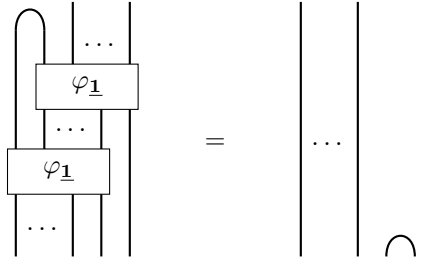
\[
\begin{tikzpicture}[scale=.75]
	\begin{pgfonlayer}{nodelayer}
		\node [style=none] (0) at (-3, -4.25) {};
		\node [style=none] (1) at (-2.5, -4.25) {};
		\node [style=none] (2) at (-2, -4.25) {};
		\node [style=none] (3) at (-1.5, -4.25) {};
		\node [style=none] (4) at (-1, -4.25) {};
		\node [style=jonesrectangle] (5) at (-2.25, -2.75) {$\quad\varphi_{\obj{1}}\quad$};
		\node [style=none] (7) at (-2.5, -2.5) {};
		\node [style=none] (8) at (-2, -2.5) {};
		\node [style=none] (11) at (-3, -0.25) {};
		\node [style=none] (12) at (-2.5, -0.25) {};
		\node [style=none] (14) at (-1.5, 0.25) {};
		\node [style=none] (15) at (-1, 0.25) {};
		\node [style=none] (16) at (0.5, -2.25) {$=$};
		\node [style=none] (17) at (2, -4.25) {};
		\node [style=none] (18) at (2.5, -4.25) {};
		\node [style=none] (19) at (3, -4.25) {};
		\node [style=none] (20) at (3.5, -4.25) {};
		\node [style=none] (21) at (4, -4.25) {};
		\node [style=none] (22) at (2, 0.25) {};
		\node [style=none] (23) at (2.5, 0.25) {};
		\node [style=none] (24) at (3, 0.25) {};
		\node [style=jonesrectangle] (25) at (-1.75, -1.25) {$\quad\varphi_{\obj{1}}\quad$};
		\node [style=none] (26) at (-2.5, -3.75) {$\dots$};
		\node [style=none] (27) at (-2, -2) {$\dots$};
		\node [style=none] (28) at (-2, -1) {};
		\node [style=none] (29) at (-2, 0.25) {};
		\node [style=none] (30) at (-1.5, -1) {};
		\node [style=none] (31) at (2.5, -2.25) {$\dots$};
		\node [style=none] (32) at (-1.5, -0.5) {$\dots$};
		\node [style=none] (33) at (4.5, -2.25) {,};
	\end{pgfonlayer}
	\begin{pgfonlayer}{edgelayer}
		\draw [style=thickstrand] (4.center) to (15.center);
		\draw [style=thickstrand] (2.center) to (8.center);
		\draw [style=thickstrand] (7.center) to (12.center);
		\draw [style=thickstrand] (17.center) to (22.center);
		\draw [style=thickstrand] (19.center) to (24.center);
		\draw [style=thickstrand, bend left=90, looseness=2.50] (20.center) to (21.center);
		\draw [style=thickstrand] (0.center) to (11.center);
		\draw [style=thickstrand, bend left=90, looseness=2.50] (11.center) to (12.center);
		\draw [style=thickstrand] (29.center) to (28.center);
		\draw [style=thickstrand] (30.center) to (3.center);
	\end{pgfonlayer}
\end{tikzpicture}
				\]
				where the ``\dots'' represent $\obj{m-2}$ vertical lines. Suppose $\varphi_{\obj{1}}$ is some Temperley--Lieb diagram. 
				Due to the ziz-zag-equal-zero relation, the cap on the left-hand side cannot be straightened.
				It cannot be removed via the circle relation either since that would require having a cup connecting the first strand to some other strand in the bottom copy of $\varphi_{\obj{1}}$, which causes a zig-zag with the top copy of $\varphi_{\obj{1}}$.
				Thus, the only way to bring the cap to the right as in the righ-hand side of the equation is if $\varphi_{\obj 1}$ is of the form
				$$
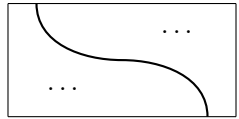
\begin{tikzpicture}[scale=.75]
	\begin{pgfonlayer}{nodelayer}
		\node [style=none] (0) at (-2, 1) {};
		\node [style=none] (1) at (2, 1) {};
		\node [style=none] (2) at (2, -1) {};
		\node [style=none] (3) at (-2, -1) {};
		\node [style=none] (4) at (0, 0) {};
		\node [style=none] (5) at (1, 0.5) {$\dots$};
		\node [style=none] (6) at (-1, -0.5) {$\dots$};
		\node [style=none] (7) at (-1.5, 1) {};
		\node [style=none] (8) at (1.5, -1) {};
	\end{pgfonlayer}
	\begin{pgfonlayer}{edgelayer}
		\draw (0.center) to (1.center);
		\draw (1.center) to (2.center);
		\draw (2.center) to (3.center);
		\draw (3.center) to (0.center);
		\draw [style=thickstrand, in=-180, out=-90] (7.center) to (4.center);
		\draw [style=thickstrand, in=90, out=0] (4.center) to (8.center);
	\end{pgfonlayer}
\end{tikzpicture}$$
with the subdiagrams ``$\dots$'' necessarily having no through strands. 
But this implies the whole left-hand side morphism has no through strands, contradicting the right-hand side of the equation.
Finally, if $\varphi_{\obj{1}}$ was a linear combination of multiple diagrams, upon inserting each diagram in the left-hand side of the equation, we do not get a contribution to the term on the right from any of the terms of $\varphi_{\obj{1}}$ due to the analysis above.
Thus, $\obj{m}$ admits no half-braidings. 
\item $\obj m \oplus \obj {m+1}$ admits no half-braiding in $\mathbf{CrysTL}$ for any positive $m$. Indeed, suppose now that $\varphi$ is a half-braiding on $\obj{m}\oplus\obj{m+1}$ for some positive integer $m$. In particular, we get an isomorphism $$(\obj{m-1}\oplus\obj {m+1}) \oplus (\obj{m}\oplus\obj{m+2})\cong (\obj m \oplus \obj{m+1})\otimes \obj 1 \overset{\varphi_{\obj{1}}}{\longrightarrow} \obj 1\otimes (\obj m \oplus \obj{m+1})\cong (\obj{m-1}\oplus\obj {m+1}) \oplus (\obj{m}\oplus\obj{m+2}) .$$
Since there are no nonzero morphisms from $\obj{m}$ to $\obj{n}$ for $m$ and $n$ of different parity, we must have $\varphi_{\obj{1}} = \psi_{\obj{1}}\oplus \vartheta_{\obj{1}}$ for some isomorphisms $\psi_{\obj 1}:\obj m\otimes \obj 1\to \obj 1 \otimes \obj m$ and $\vartheta_{\obj 1}:\obj{m+1}\otimes \obj 1\to \obj 1 \otimes \obj{m+1}$ both satisfying the naturality condition of \autoref{half-braid-criterion-TL}. 
Thus, we get half-braidings on $\obj m$ and on $\obj{m+1}$, which was shown to be impossible. 
\item Every object in $\mathbf{CrysTL}$ embeds in $\obj m \oplus \obj{m+1}$ for sufficiently large $m$. This is easily seen by the fact that the multiplicity of $T_n$ strictly increases in $\obj m$ as $m\to \infty$ whenever $m$ and $n$ have the same parity.
			\end{enumerate}
			
Thus, to prove \autoref{crystal-conj} it suffices to show that a summand $X$ of $\obj m$ admits a half-braiding only if $X\cong \id^{\oplus n}$.
Suppose $\varphi$ is a half-braiding on a summand $X$ of $\obj{m}$.
Let the idempotent $e\in\End({\obj m})$ be the projection onto $X$. 
Then, the morphisms $\psi_Z:= (\Id_Z\otimes e)\circ \varphi_Z \circ  (e\otimes \Id_Z)$ define a natural transformation $\psi:\obj m \otimes - \implies -\otimes \obj{m}$ obeying the hexagon axiom (but is not an isomorphism). 
By naturality with respect to $\CAP$, we have 
\[\begin{tikzcd}
	{\obj m\otimes \obj 1\otimes \obj 1} & {\obj 1\otimes \obj m\otimes \obj 1} & {\obj 1\otimes \obj 1\otimes \obj m} \\
	{\obj m\otimes \obj 0} && {\obj 0\otimes \obj m}
	\arrow["{\psi_{\obj 1}\otimes \Id_{\obj 1}}", from=1-1, to=1-2]
	\arrow["{\Id_X\otimes \CAP}"', from=1-1, to=2-1]
	\arrow["{ \Id_{\obj 1}\otimes \psi_{\obj 1}}", from=1-2, to=1-3]
	\arrow["{\CAP\otimes X}", from=1-3, to=2-3]
	\arrow["{\varphi_{\obj 0} = e\,\circ \, r^{-1}l\,\circ\, e\,}"', from=2-1, to=2-3]
\end{tikzcd}\] 
To prove that $X$ is a direct sum of $\obj 0$, it suffices to show that the diagrams appearing in $e$ contain no through strands, so that $e$ factors through $\obj 0^{\oplus n}$. Thus, we get the following combinatorial version of \autoref{crystal-conj}.

\textbf{Conjecture 2 (diagrammatic version).} \textit{Let $e\in\End(\obj m)$ be an idempotent, $\varphi\in\End((\obj{m},e)\otimes \obj 1)$ an isomorphism, and $\psi = (\Id_{\obj 1}\otimes e) \circ \varphi \circ (e\otimes \Id_{\obj 1})$. 
If the equation 
\[\begin{tikzpicture}[scale=.75]
	\begin{pgfonlayer}{nodelayer}
		\node [style=none] (0) at (-3, -4.25) {};
		\node [style=none] (1) at (-2.5, -4.25) {};
		\node [style=none] (2) at (-2, -4.25) {};
		\node [style=none] (3) at (-1.5, -4.25) {};
		\node [style=none] (4) at (-1, -4.25) {};
		\node [style=jonesrectangle] (5) at (-2.25, -2.75) {$\quad\;\psi\;\quad$};
		\node [style=none] (7) at (-2.5, -2.5) {};
		\node [style=none] (8) at (-2, -2.5) {};
		\node [style=none] (11) at (-3, -0.25) {};
		\node [style=none] (12) at (-2.5, -0.25) {};
		\node [style=none] (14) at (-1.5, 0.25) {};
		\node [style=none] (15) at (-1, 0.25) {};
		\node [style=none] (16) at (0.5, -2.25) {$=$};
		\node [style=none] (17) at (2, -4.25) {};
		\node [style=none] (18) at (2.5, -4.25) {};
		\node [style=none] (19) at (3, -4.25) {};
		\node [style=none] (20) at (3.5, -4.25) {};
		\node [style=none] (21) at (4, -4.25) {};
		\node [style=none] (22) at (2, 0.25) {};
		\node [style=none] (23) at (2.5, 0.25) {};
		\node [style=none] (24) at (3, 0.25) {};
		\node [style=jonesrectangle] (25) at (-1.75, -1.25) {$\quad\;\psi\;\quad$};
		\node [style=none] (26) at (-2.5, -3.75) {$\dots$};
		\node [style=none] (27) at (-2, -2) {$\dots$};
		\node [style=none] (28) at (-2, -1) {};
		\node [style=none] (29) at (-2, 0.25) {};
		\node [style=none] (30) at (-1.5, -1) {};
		\node [style=none] (31) at (2.5, -3.25) {$\dots$};
		\node [style=none] (32) at (-1.5, -0.5) {$\dots$};
		\node [style=none] (33) at (4.5, -2.25) {.};
		\node [style=none] (34) at (2.5, -0.75) {$\dots$};
		\node [style=jonesrectangle] (35) at (2.5, -2) {};
		\node [style=jonesrectangle] (36) at (2.5, -2) {$\quad e\quad$};
	\end{pgfonlayer}
	\begin{pgfonlayer}{edgelayer}
		\draw [style=thickstrand] (4.center) to (15.center);
		\draw [style=thickstrand] (2.center) to (8.center);
		\draw [style=thickstrand] (7.center) to (12.center);
		\draw [style=thickstrand] (17.center) to (22.center);
		\draw [style=thickstrand] (19.center) to (24.center);
		\draw [style=thickstrand, bend left=90, looseness=2.50] (20.center) to (21.center);
		\draw [style=thickstrand] (0.center) to (11.center);
		\draw [style=thickstrand, bend left=90, looseness=2.50] (11.center) to (12.center);
		\draw [style=thickstrand] (29.center) to (28.center);
		\draw [style=thickstrand] (30.center) to (3.center);
	\end{pgfonlayer}
\end{tikzpicture}
\] holds, $e$ must be a linear combination of diagrams containing no through-strands.}

It is conceivable that this conjecture can be settled by a simple combinatorial argument as in (3) above, but at this point, such a proof seems evasive.



\appendix

\section{Ultraproducts for the Working $\otimes$-Categorician}\label{sec:ultraprod}

In this appendix, we shall give a quick and elementary review of the basic notions of model theory, building up to ultraproducts and \Los, as well as their applicability to tensor categories and tensor functors. 
For a comprehensive introduction to ultrafilters and ultraproducts and their applications, see \cite{G}. For a treatment of ultraproducts focused on their usage in algebra, refer to \cite{Schout}.

	\subsection*{Languages, Theories, and Models}
	
	\begin{definition}
		A \emph{first-order language} $\mathcal L$ is a set of symbols of three types: constant symbols $c_i$, function symbols $f_j$ and relation symbols $R_k$. 
		Associated to each function symbol and relation symbol is a positive integer called its \emph{arity}.
	\end{definition}
	
	\begin{definition}
		Let $\mathcal L$ be a first-order language. An $\mathcal L$-structure $\mathcal M$ is the following data:
		\begin{itemize}
			\item a nonempty underlying set $M$;
			\item for each constant symbol $c\in\mathcal L$, a distinguished element $c^\mathcal M \in M$, called the \emph{interpretation} of $c$ in $\mathcal M$. 
			\item  for each $n$-ary function symbol $f\in\mathcal L$, a function $f^\mathcal M: M^n\to M$, called the \emph{interpretation} of $f$ in $\mathcal M$. 
			\item for each $n$-ary relation symbol $R\in\mathcal L$, a relation $R^\mathcal M \subseteq M^n$, called the \emph{interpretation} of $R$ in $\mathcal M$. 
		\end{itemize}
	\end{definition}
	
	\begin{example} [The language of rings] Let $\mathcal L = \{0, 1, +, -, \cdot\}$ be a language, where $0$ and $1$ are constant symbols and $+$, $-$, and $\cdot$ are binary function symbols. An example of an $\mathcal L$-structure is the set $\ZZ$ with the usual interpretations of the symbols.
	\end{example}
		
		The \emph{alphabet} for writing statements in a language $\mathcal L$ consists of the following symbols: the logical connectives $\neg, \land, \lor, \to, \leftrightarrow$; the parentheses `$($' and `$)$'  and comma `,'; the quantifiers $\forall$ and $\exists$; countably many variable symbols $x_0, x_1, x_2, \dots$; along with the symbols of $\mathcal L$. 
	
	\begin{definition}
		The set $\mathcal T$ of $\mathcal L$-\emph{terms} is defined inductively as follows:
		\begin{itemize}
			\item Let $\mathcal T_0 = \{c\in \mathcal L\mid c \text{ is a constant symbol}\}\cup \{x_0, x_1, x_2, \dots\}$.
			\item For $n\in\NN$, let $\mathcal T_{k+1} = \mathcal T_k \,\cup \{f(t_1, t_2, \dots, t_n) \mid f\in\mathcal L \text{ is an }n\text{-ary function symbol and }t_1, t_2, \dots t_n\in \mathcal T_k\}$.
			\item Let $\mathcal T = \bigcup_{k\in\NN} \mathcal T_k$.
		\end{itemize}
	\end{definition}
	
	\begin{definition}
		Let $\mathcal L$ be a first-order language. 
		We define the set $\mathcal F$ of first-order formulae in $\mathcal L$ inductively, as follows:
		\begin{itemize}
			\item Let $\mathcal F_0$ be the set of expression (so-called \emph{atomic formulae}) of the form $R(t_1, t_2, \dots, t_n)$ where $R\in\mathcal L$ is relation symbol and $t_1, t_2, \dots, t_n$ are $\mathcal L$-terms.
			\item For every $n\in\NN$, let 
			\begin{align*}
				\mathcal F_{n+1} = \mathcal F_n &\cup \{\neg \varphi \mid \varphi \in \mathcal F_n \} \cup \{\varphi \,\land \psi \mid \varphi, \psi\in\mathcal F_n \} \cup \{\varphi \,\lor \psi \mid \varphi, \psi\in\mathcal F_n \} \cup \{\varphi \,\to \psi \mid \varphi, \psi\in\mathcal F_n \}\\
				&\cup \{\varphi \,\leftrightarrow \psi \mid \varphi, \psi\in\mathcal F_n \} \cup \{\forall x_i \,\varphi \mid \varphi\in\mathcal F_n, \,i\in\NN \} \cup \{\exists x_i \,\varphi \mid \varphi\in\mathcal F_n, \,i\in\NN \}.
			\end{align*}
			\item Let $\mathcal F = \bigcup_{n\in\NN} \mathcal F_n$.
		\end{itemize}
		
		We add parenthesis to clarify the scopes of connectives and quantifiers as usual. 
		An occurrence of a variable in a first-order formula is called \emph{free} if it does not fall under the scope of a quantifier; a variable which has at least one free occurrence in a given formula is called a \emph{free variable}. 
		A first-order formula which has no free variables is called a \emph{sentence}.
	\end{definition}
	
	Informally, a first-order sentence can be thought of as a ``grammatically correct'' and ``finitary'' proposition.
	More precisely, in such a sentence, \emph{we are allowed to use only finitely many variables, finitely many applications of function/relation in our language, finitely many logical connectives, and finitely many quantifiers}. 
	Moreover, each variable must occur under the action of some quantifier for the sentence to make sense. 
	It is also crucial to realize what we are allowed or not allowed to quantify over. 
	The next two definition will clarify that the variables $x_0, x_1, x_2, \dots$, are to be interpreted  as placeholders for elements of the underlying set $M$ of our $\mathcal L$-structure $\mathcal M$. 
	This has the effect of \emph{only allowing us to quantify over elements of the structure} but not, for example, over subsets thereof, nor over families of statements.
	
	\begin{definition}
		Let $t = t(x_1, x_2, \dots, x_n)$ be a term of $\mathcal{L}$. Let $\mathcal{M}$ be an $\mathcal{L}$-structure with underlying set $M$ and let $a_1, a_2, \dots, a_n \in M$. The \textit{interpretation of the term $t$ in the structure $\mathcal{M}$ when the variables $x_1, x_2, \dots, x_n$ are interpreted respectively by the elements $a_1, a_2, \dots, a_n$} is an element of $M$, denoted by $t^\mathcal{M}(a_1, a_2, \dots, a_n)$
and is defined inductively as follows:
\begin{itemize}
    \item if $t = c$ where $c \in \mathcal{L}$ is a constant symbol, then
       $t^\mathcal{M}(a_1, a_2, \dots, a_n) = c^\mathcal{M}$;

    \item if $t = x_i$ where $1 \le i \le n$, then
       $t^\mathcal{M}(a_1, a_2, \dots, a_n) = a_i$;

    \item if $t = f(t_1, t_2, \dots, t_k)$ where $f \in \mathcal{L}$ is a $k$-ary function symbol and $t_1, t_2, \dots, t_k$ are terms of $\mathcal{L}$, then $t^\mathcal{M}(a_1, a_2, \dots, a_n) = f^\mathcal{M}\bigl(t_1^\mathcal{M}(a_1, \dots, a_n), \dots, t_k^\mathcal{M}(a_1, \dots, a_n)\bigr)$.
\end{itemize}
	\end{definition}
		
	\begin{definition}
	Suppose that $\varphi(x_1, x_2, \dots, x_n)$ is a first-order formula of some language $\mathcal{L}$. Let $\mathcal{M}$ be an $\mathcal{L}$-structure with underlying set $M$ and $a_1, a_2, \dots, a_n\in M$. We write
\[
    \mathcal{M} \models \varphi(a_1, a_2, \dots, a_n)
\]
when \textit{the formula $\varphi$ is satisfied in the structure $\mathcal{M}$ when the variables $x_1, x_2, \dots, x_n$ are interpreted respectively by the elements $a_1, a_2, \dots, a_n$}. 
We write $\mathcal{M} \not\models \psi$ when $\mathcal{M}$ does not satisfy the formula $\psi$. 
More precisely, the satisfaction of a formula in a structure is defined inductively as follows:

\begin{itemize}
    \item Suppose the formula $\varphi(x_1, \dots, x_n)$ is an atomic formula, say $\varphi = R(t_1, \dots, t_k)$ where $R \in \mathcal{L}$ is a $k$-ary relation symbol and $t_1, \dots, t_k$ are $\mathcal{L}$-terms, where the variables in each term $t_i$ are among $x_1, \dots, x_n$. Then we define $\mathcal{M} \models \varphi(a_1, \dots, a_n)$ if and only if
    \[
        (t_1^\mathcal{M}(a_1, \dots, a_n), \dots, t_k^\mathcal{M}(a_1, \dots, a_n)) \in R^\mathcal{M}.
    \]

    \item Suppose $\varphi = \neg\psi(x_1, \dots, x_n)$. Then $\mathcal{M} \models \varphi(a_1, \dots, a_n)$ if and only if $\mathcal{M} \not\models \psi(a_1, \dots, a_n)$.

    \item Suppose $\varphi = (\psi \wedge \theta)$. Then $\mathcal{M} \models \varphi(a_1, \dots, a_n)$ if and only if $\mathcal{M} \models \psi(a_1, \dots, a_n)$ and $\mathcal{M} \models \theta(a_1, \dots, a_n)$.

    \item Suppose $\varphi = (\psi \vee \theta)$. Then $\mathcal{M} \models \varphi(a_1, \dots, a_n)$ if and only if $\mathcal{M} \models \psi(a_1, \dots, a_n)$ or $\mathcal{M} \models \theta(a_1, \dots, a_n)$.

    \item Suppose $\varphi = (\psi \rightarrow \theta)$. Then $\mathcal{M} \models \varphi(a_1, \dots, a_n)$ if and only if $\mathcal{M} \not\models \psi(a_1, \dots, a_n)$ or $\mathcal{M} \models \theta(a_1, \dots, a_n)$.

    \item Suppose $\varphi = (\psi \leftrightarrow \theta)$. Then $\mathcal{M} \models \varphi(a_1, \dots, a_n)$ if and only if either both $\mathcal{M} \models \psi(a_1, \dots, a_n)$ and $\mathcal{M} \models \theta(a_1, \dots, a_n)$, or else $\mathcal{M} \not\models \psi(a_1, \dots, a_n)$ and $\mathcal{M} \not\models \theta(a_1, \dots, a_n)$.

    \item Suppose $\varphi = \forall y \, \psi(x_1, x_2, \dots, x_n, y)$ where $y \notin \{x_1, \dots, x_n\}$. 
    Then $\mathcal{M} \models \varphi(a_1, a_2, \dots, a_n)$ if and only if for every element $b$ in $\mathcal{M}$ we have that $\mathcal{M} \models \psi(a_1, a_2, \dots, a_n, b)$.

    \item Suppose $\varphi = \exists y \, \psi(x_1, x_2, \dots, x_n, y)$ where $y \notin \{x_1, \dots, x_n\}$. Then $\mathcal{M} \models \varphi(a_1, a_2, \dots, a_n)$ if and only if there exists at least one element $b$ in $\mathcal{M}$ such that $\mathcal{M} \models \psi(a_1, a_2, \dots, a_n, b)$.

    \item Suppose $\varphi = \forall x_i \, \psi(x_1, x_2, \dots, x_n)$ where $1 \le i \le n$. Then $\mathcal{M} \models \varphi(a_1, \dots, a_i, \dots, a_n)$ if and only if for every element $b$ in ${M}$ we have that $\mathcal{M} \models \psi(a_1, \dots, a_{i-1}, b, a_{i+1}, \dots, a_n)$.

    \item Suppose $\varphi = \exists x_i \, \psi(x_1, x_2, \dots, x_n)$ where $1 \le i \le n$. Then $\mathcal{M} \models \varphi(a_1, \dots, a_i, \dots, a_n)$ if and only if there is at least one element $b$ in ${M}$ such that $\mathcal{M} \models \psi(a_1, \dots, a_{i-1}, b, a_{i+1}, \dots, a_n)$.
\end{itemize}
	\end{definition}

	\begin{definition}
		Given a language $\mathcal L$, and a set $T$ of first order $\mathcal L$-sentences, a \emph{model} for $T$ is an $\mathcal L$-structure $\mathcal M$ satisfying every sentence in $T$; symbolically, $\forall \varphi\in T, \,\mathcal M\models \varphi$. 
		When $\mathcal M$ is a model of $T$, we write $\mathcal M\models T$.
	\end{definition}
	
	\begin{definition}
		Given a set of $\mathcal L$-sentences $T$, an $\mathcal L$-sentence $\varphi$ is called a \emph{logical consequence} of $T$ if every model of $T$ satisfies $\varphi$. 
		A \emph{first-order theory} or an $\mathcal L$\emph{-theory} is a set of $\mathcal L$-sentences together with all their logical consequences.
		A set of $\mathcal L$-sentences generating a theory $\mathcal T$ under logical consequence are called \emph{axioms} for $\mathcal T$.
	\end{definition}
	
	\begin{example}
		The theory of groups can be expressed in the language $\mathcal L = \{e, *\}$, where $e$ is a constant symbol and $*$ is a binary function symbol, with the following axioms:
		\begin{itemize}
			\item $\forall x\,\forall y\,\forall z\;(x*(y*z)) = ((x*y)*z)$,
			\item $\forall x\,(x*e = x) \land (e*x = x)$,
			\item $\forall x\,\exists y\; (x*y = e) \land (y*x=e)$.
		\end{itemize}
		Note that `$=$' is always assumed to be a binary relation symbol in every language and has the usual interpretation. A model for the theory of groups is a group. 
	\end{example}	
		
	\subsection*{Ultrafilters and Ultraproducts}
	
	\begin{definition}
		Given a set $I$, a family $\mathscr F\subset 2^I$ of subsets of $I$ is called a filter on $I$ if it satisfies the following:
		\begin{itemize}
			\item $I\in \mathscr F$ and $\emptyset\notin \mathscr F$;
			\item if $A\in \mathscr F$ and $B\supseteq A$, then $B\in \mathscr F$;
			\item if $A, B\in \mathscr F$, then $A\cap B\in \mathscr F$.
		\end{itemize}
	\end{definition}
	
	One should think of filter as a declaration of some sets which we consider to be ``large''. 
	Consequently, following example illustrates a type of filter which we wish to avoid.
	
	\begin{example}[Principal filter]
		Fix some $x\in I$, and define $\mathscr F = \{A\subseteq I\mid x\in A\}$. Then, $\mathscr F$ is a filter on $I$ called the \emph{principal filter generated by $x$}. Filters which are not of this type are called \emph{non-principal}.
	\end{example}
	
	Everything that follows works fine, but is utterly useless, for principal filters, so we shall often restrict ourselves to the case of a non-principal filter. A good example of a non principal filter is the following.
	
	\begin{example}[The Frech\'et Filter]
		Let $I$ be an infinite set, and let $\mathscr F$ be the set of cofinite subsets of $I$. Then, $\mathscr F$ is a non-principal filter on $I$, called the \emph{Frech\'et Filter}.
	\end{example}
	
	\begin{definition}
		An \emph{ultrafilter} on $I$ is a filter $\mathscr U$ such that, for every subset $A\subseteq I$, either $A\in \mathscr U$ or $A\notin \mathscr U$.
	\end{definition}
	
	The Frech\'et filter is a great example to keep in mind due to the following well-known fact: every non-principal ultrafilter must contain the Frech\'et filter. This justifies the intuition of ``sets in the ultrafilter are `large' ''. Indeed, from now on, we will refer to sets in a some ultrafilter $\mathscr U$ as \emph{$\mathscr U$-large}, and we will say that a set of proposition $\{\varphi_i\}_{i\in I}$ is true \emph{$\mathscr U$-almost everywhere} if $\{i\in I\mid \varphi_i \text{ is true}\}\in \mathscr U$. We will also omit the mention of $\mathscr U$ and simply say \emph{``large''} or \emph{``almost everywhere''} whenever the ultrafilter is clear from context. 
	
	The following result is a standard application of Zorn's Lemma.
	
	\begin{proposition}
		Every filter can be extended to an ultrafilter. In particular, there exists an ultrafilter on every infinite set.
	\end{proposition}
	
	\begin{remark}
		Henceforth, let us only consider the case when $I$ is countable, or simply $I=\NN$. 
		Ultraproducts over uncountable sets still make sense but to simplify notations and terminology (e.g. talking about sequences, rather than choice functions), let us focus on this special case which suffices for our purposes.
	\end{remark}
	
	\begin{lemma}\label{ultraPartitionLemma}
		Suppose $S_1, S_2, \dots, S_n$ is a partition of $\NN$, and $\mathscr U$ is an ultrafilter on $\NN$. Then, $S_i\in\mathscr U$ for exactly one $i\in\{1,2,\dots, n\}$.
	\end{lemma}
	
	\begin{proof}
		If $n=1$, this is trivial as $S_1 = \NN\in\mathscr U$. Otherwise, the definition of an ultrafilter guarantees that exactly one of $S_1\cup \dots \cup S_{n-1}$ or $S_n$ lies in $\mathscr U$, so the claim follows by induction on $n$.
	\end{proof}
	
	\begin{definition}
		Given a sequence of sets $\{A_i\}_{i\in\NN}$, their ultraproduct with respect to some ultrafilter $\mathscr U$ is the quotient set $\uprod A_i :=\left(\prod_{i\in\NN} A_i\right)/\sim_{\mathscr U}$, where the equivalence relation $\sim_{\mathscr U}$ is defined as follows: $(a_i)_{i\in\NN}\sim (a_i')_{i\in\NN}$ if $\{i\in \NN\mid a_i = a_i'\}\in\mathscr U$; in other words, the set of indices on which the two sequences agree is $\mathscr U$-large, or equivalently, if the two sequences are equal $\mathscr U$-almost everywhere.
		We denote the equivalence class of a sequence $a = (a_i)_{i\in\NN}$ in the ultraproduct $\uprod A_i$ by $[a]$ or $\ulimit{i} a_i$.
		\end{definition}
		
		If the factors in the ultraproduct have more interesting structures than merely being sets, the ultraproduct inherits that structure as shown below.
		
		\begin{definition}
			Let $\mathscr U$ be an ultrafilter on $\NN$ and $\{\mathcal M_i\}_{i\in\NN}$ be a sequence of $\mathcal L$-structures with underlying sets $\{M_i\}_{i\in\NN}$.
			We define their ultraproduct $\mathcal M^* = \uprod \mathcal M_i$ to be the $\mathcal L$-structure with underlying set $\uprod M_i$, where the interpretations of the language symbols are defined as follows:
			\begin{itemize}
				\item for a constant symbol $c\in\mathcal L$, let $c^{\mathcal M^*} = \ulimit{i} c^{\mathcal M_i}$;
				\item for a function symbol $f\in\mathcal L$ of arity $k$ and sequence $a^1, a^2, \dots, a^k\in \prod_{i\in\NN} M_i$, let $$f^{\mathcal M^*}\left([a^1], [a^2], \dots, [a^k]\right) = \ulimit{i}  f^{\mathcal M_i}\left(a^1_i, a^2_i, \dots, a^k_i\right);$$
				\item for a relation symbol $R\in\mathcal L$ of arity $k$ and sequence $a^1, a^2, \dots, a^k\in \prod_{i\in\NN} M_i$, let 
				$$\left([a^1], [a^2], \dots, [a^k]\right)\in R^{\mathcal M^*} \iff \left\{i\in\NN\mid \left(a^1_i, a^2_i, \dots, a^k_i\right)\in R^{\mathcal M_i}\right\}\in \mathscr U.$$
			\end{itemize}
		\end{definition}
		
	Now, we can finally state the \emph{Fundamental Theorem of Ultraproducts}: 
	
	\begin{theorem}[\Los]
		Let $\mathscr U$ be an ultrafilter on $\NN$ and $\{\mathcal M_i\}_{i\in\NN}$ be a sequence of $\mathcal L$-structures with underlying sets $\{M_i\}_{i\in\NN}$.
		Then, for any first-order formula $\varphi(x_1, x_2, \dots, x_n)$ of $\mathcal L$ and any sequences $a^1, a^2, \dots, a^n\in\prod_{i\in\NN} M_i$, we have that
		$$\uprod \mathcal M_i \models \varphi\left([a^1], [a^2], \dots, [a^n]\right) \iff \left\{i\in\NN \mid \mathcal M_i \models \varphi\left(a^1_i, a^2_i, \dots, a^k_i\right)\right\}\in\mathscr U.$$
		In particular, since $\mathscr U$ must contain the Frech\'et filter, if $\{i\in\NN\mid \mathcal M_i\not\models \varphi\}$ is finite, then $\uprod \mathcal M_i \models \varphi$.
	\end{theorem}
		
	\begin{corollary}
		Let $\mathscr U$ be an ultrafilter on $\NN$, $\{\mathcal M_i\}_{i\in\NN}$ be a sequence of $\mathcal L$-structures, and $\mathcal T$ be a first-order theory. Then, $\{i\in\NN\mid \mathcal M_i\models \mathcal T\}\in \mathscr U$ if and only if $\uprod \mathcal M_i \models \mathcal T$.
	\end{corollary}

	\subsection*{Ultraproducts of Tensor Categories}
	
	In \cite[Chapter 4]{C}, a first-order theory of \emph{small} (tannakian) tensor categories is constructed.
	It is expressed in a language with relation symbols that specify the following data: whether an element is an object or morphism, whether a morphism has two objects as its source and target, whether a morphism/object is the composition/sum/tensor product of two others, whether a morphism is the associator/commutor of two objects. 
	The axioms for this theory guarantee that these relations behave as expected and satisfy the axioms for rigid abelian symmetric/braided monoidal $\KK$-linear category, with $\End(\id) \cong \KK$. In fact, the way $\KK$-linearity is encoded in the axioms is by asserting that $\End(\id)$ is a field, as one can think of scaling a morphism $f:X\to Y$ by $c\in\KK\cong \End(\id)$ to be the morphism $l_Y\circ (c\otimes f)\circ l_X^{-1}$, where $l_{-}:\id\otimes - \to -$ is the left unitor.
	We will not repeat all the axioms here for brevity, but since the categories we are dealing with have slightly different structure, we must be careful in using those results. 
	
	There are three main differences between our \autoref{tensorCat} and the categories in \emph{loc. cit.}: 1) our categories are not assumed to be braided from the start; 2) we work with idempotent-complete instead of abelian categories; 3) our categories are assumed to have finite-dimensional morphism spaces.
	The first issue is easily remedied. 
	One may simply drop the relation assigning a braiding to every two objects along with all axioms relating to it.
	Secondly, one keeps the additive category structure but drops the axioms of being an abelian category, i.e. the existence of kernels and cokernels and the normality of monos and epis.
	Instead, we assert that idempotents split, which is expressible by a first-order sentence in this language as follows:
	$$\forall e\;(((e:A\to A) \land (e\circ e = e)) \implies \exists B \,\exists f\,\exists g\; ((f:A\to B) \land (g:B\to A)\land (g\circ f = e) \land (f\circ g = \Id_B))).$$
	This ensures our categories are idempotent-complete. 
	The third issue is impossible to remedy since the property of being finite-dimensional is usually \emph{not} expressible in first-order logic.
	Indeed, taking an ultraproduct of vector spaces of unbounded finite dimensions yields an infinite-dimensional vector space.
	The way to deal with this is by considering only a smaller more manageable subcategory of the ultraproduct which has finite dimensional morphism spaces.

	The discussion above together with the results of \cite[Chapter 4]{C} can be summarized in the following Theorem.
	
	\begin{theorem}\label{tensorCatsUltraprod}
		Let $\mathscr U$ be an ultrafilter, and let $\{\mathcal C_i\}_{i\in\NN}$ be a sequence of small categories. 
		The ultraproduct $\uprod \mathcal C_i$ is a well-defined small category whose objects and morphisms are, respectively, $\sim_{\mathcal U}$ equivalence classes of sequences of objects and morphisms in $\mathcal C_i$. 
		Moreover, if the categories $\mathcal C_i$ are monoidal, rigid, additive, idempotent-complete, abelian, braided, or symmetric $\mathscr U$-almost everywhere, then so is $\uprod \mathcal C_i$. 
		If, for $\mathscr U$-many $i\in\NN$, $\mathcal C_i$ is $\KK_i$-linear, then $\uprod \mathcal C_i$ is $\left(\uprod \KK_i\right)$-linear.
	\end{theorem}
	
	Due to the following well-known result, one can pass information between categories in characteristic $p$ to their ultraproduct in characteristic $0$. 
	It is a consequence of Steinitz' Theorem, and a full-proof is given, for example, in \cite[Section 2.4.1]{Schout}. 
	
	\begin{proposition}\label{uprodFields}
	Let $\mathscr U$ be a non-principal ultra filter, and let $\{\KK_i\}_{i\in\NN}$ be a sequence of algebraically closed fields, each of which is of cardinality at most $2^{\aleph_0}$. 
	Suppose for every prime number $p$, only finitely many $\KK_i$ have characteristic $p$. 
	Then, we have an isomorphism of fields $\uprod \KK_i \cong \CC$.
	In particular,
	\begin{itemize}
		\item $\uprod \CC \cong \CC$.
		\item Let $\{p_i\}_{i\in\NN}$ be a sequence of prime numbers with $\liminf\limits_{i\to\infty} p_i = \infty$. Then, $\uprod \overline{\FF_{p_i}} \cong \CC$.
	\end{itemize}
	\end{proposition}

	Now, we would like to examine the semisimplicity of (subcategories of) an ultraproduct of semisimple categories.
	In general, the ultraproduct of semisimple categories will not be semisimple, but sufficiently well-behaved subcategories thereof will still be semisimple.
	To exploit \Los, we need to express semisimplicity by a first-order language sentence. 
	The next two general folklore lemmas do the trick.
	
	Recall that a morphism is called \emph{split} if it can be written as a composition $\iota\circ \pi$ where $\iota$ is a split monomorphism and $\pi$ is a split epimorphism.
	
\begin{lemma}\label{split-semisimple}
	A tensor category $\cat$ over an algebraically closed field is semisimple if and only if every morphism of $\cat$ is split.
\end{lemma}

\begin{proof}
	One direction is obvious: if $\cat$ is semisimple, and $f:X\to Y$ is a morphism, then $f$ factors through its image, so the projection $\pi:X\to\im(f)$ and inclusion $\iota:\im(f)\to Y$ give a splitting of $f$. 
	
	Conversely, suppose every morphism is split. 
	Since $\cat$ is Krull-Schmidt (see \autoref{Krull-Schmidt}), to prove that it is semisimple, it is enough to show that, for any indecomposables $X\ncong Y$, we have $\Hom(X,Y) = 0$ and $\End(X) \cong \KK$; indeed, such a category is equivalent to finite-dimensional vector spaces graded by isomorphism classes of indecomposables.  
	Suppose $f:X\to Y$ is any morphism between two indecomposables, and let $\pi:X\to Z$ and $\iota:Z\to Y$ be split epimorphism and split monomorphism such that $\iota\circ\pi = f$. 
	Thus, $Z$ is a summand of $X$ and of $Y$; but since they are both indecomposables, either $Z=0$ so that $f=0$ or $\pi$ and $\iota$ are isomorphisms and, hence, so is $f$. 
\end{proof}

\begin{lemma}\label{splitCriterion}
	In a Karoubian category, a morphism $f: A\to B$ is a split if and only if there exists a morphism $g:B\to A$ satisfying $f\circ g\circ f = f$.
\end{lemma}	

\begin{proof}
	Suppose $f = \iota\circ \pi$ where $\pi:A\to Z$ is a split epimorphism with a section $s:Z\to A$ (i.e. $\pi\circ s = \Id_Z$) and $\iota: Z\to B$ is a split monomorphism with retraction $r:B\to Z$ (i.e. $r\circ \iota = \Id_Z$). 
	Then, a quick calculation shows that $g := s\circ r$ satisfies $f\circ g\circ f = f$. 
	
	Conversely, suppose $g:B\to A$ satisfies $f\circ g\circ f = f$.
	Then, let $e:= g\circ f$ is an idempotent.
	Since $\cat$ is Karoubian, we have a splitting $e = u\circ p$ for some $p:A\to Z$ and $u:Z\to A$ with $p\circ u = \Id_Z$ (so $p$ is a split epimorphism). 
	Now, take $\pi := p$ and $\iota := f\circ u$.
	Note that $\iota$ is a split monomorphism with left-inverse $p\circ g$. 
	Moreover, a quick calculation shows that $f = \iota \circ \pi$, as desired.
\end{proof}

	\begin{proposition}\label{ultraSemisimplicity}
		Let $\mathscr U$ be an ultrafilter, $\KK = \uprod \KK_i$ an algebraically closed field, and $\{\mathcal C_i\}_{i\in\NN}$ a sequence of small semisimple $\KK_i$-linear categories. 
		Suppose $\mathcal D$ is a Karoubian, Krull-Schmidt $\KK$-linear subcategory of $\uprod \cat_i$.
		Then, $\mathcal D$ is semisimple.
		Moreover, $\ulim{i} X_i$ is a simple object in $ \uprod \cat_i$ if and only if $X_i$ is simple for $\mathscr U$-almost every $i$.
	\end{proposition}
	
	\begin{proof}
	By \autoref{split-semisimple}, it is enough show that every morphism in $\mathcal D$ is split.
		Any morphism $f:A\to B$ in $\mathcal D$ can be identified with $\ulim{i} f_i$ in $\uprod \cat_i $ for a sequence of morphisms $f_i:A_i\to B_i$ in $\cat_i$.
		But, by \autoref{split-semisimple}, each $f_k$  splits since $\cat_i$ is semisimple.
		So, by \autoref{splitCriterion}, we have a sequence $g_i:B_i\to A_i$ satisfying $f_i\circ g_i\circ f_i = f_i$. 
		It follows by {\Los} that $g:=\ulim{i} g_i$ satisfies $f\circ g\circ f=f$, so $f$ splits again by \autoref{splitCriterion}.
		
		Now, observe that $X:=\ulim{i} X_i$ is not simple if and only if there exists two linearly independent morphisms $f = \ulim{i} f_i$ and  $g = \ulim{i} g_i$ in $\End(X)$. 
		But, by \Los, this happens if and only if there are two sequences $f_i$ and $g_i$ of morphisms in $\End(X_i)$ such that $\forall \lambda \in\KK_i, \, f_i \neq \lambda g_i $ for $\mathscr U$-almost all $i$, i.e. $X_i$ is not simple $\mathscr U$-almost always.
	\end{proof}

	\subsection*{Ultraproducts of Tensor Functors}
	
	Similar to the language of tensor categories, one can define a language of tensor functors $(F, J):\cat\to \mathcal D$. In this language, we have the following relation symbols:
	$$\underline{\hspace{0.1in}}\in\text{Ob}(\cat), \quad \underline{\hspace{0.1in}}\in\text{Mor}(\cat), \quad \underline{\hspace{0.1in}}\in\text{Ob}(\mathcal D), \quad \underline{\hspace{0.1in}}\in\text{Mor}(\mathcal D), \quad \underline{\hspace{0.1in}}:\underline{\hspace{0.1in}}\to \underline{\hspace{0.1in}}, \quad \underline{\hspace{0.1in}}\circ \underline{\hspace{0.1in}} = \underline{\hspace{0.1in}}, \quad \underline{\hspace{0.1in}} + \underline{\hspace{0.1in}} = \underline{\hspace{0.1in}},$$
	$$\underline{\hspace{0.1in}} \otimes  \underline{\hspace{0.1in}} = \underline{\hspace{0.1in}}, \quad 
	\text{Assoc}\underline{\hspace{0.1in}}, \underline{\hspace{0.1in}},\underline{\hspace{0.1in}}  = \underline{\hspace{0.1in}}, \quad \text{Comm}\underline{\hspace{0.1in}} ,\underline{\hspace{0.1in}}  = \underline{\hspace{0.1in}}, \quad \text{Unit}\underline{\hspace{0.1in}} = \underline{\hspace{0.1in}}, \quad  F(\underline{\hspace{0.1in}} ) = \underline{\hspace{0.1in}}, \quad J\underline{\hspace{0.1in}} , \underline{\hspace{0.1in}} = \underline{\hspace{0.1in}}.$$
	They have their usual interpretations as \underline{\hspace{0.1in}} is an object or a morphism of $\cat$ or $\mathcal D$, a morphism \underline{\hspace{0.1in}} is from some object to another, a morphism is the composition or sum or tensor product of two others, a morphism is the associator or commutator (if the categories are braided) or unitor of some objects, an object or a morphism is the image of $F$ of another, and finally that a morphism in $\mathcal D$ is the monoidal structure isomorphism of $F$ for two given objects.
	
	The theory $\mathcal T_{\otimes\text{-Fun}}$ will have the same axioms as before, with some extra axioms guaranteeing that morphisms in $\cat$ have sources and targets in $\cat$ and compose, add, or tensor (whenever meaningful) to give other morphisms in $\cat$, the associators, commutators, unitors, between objects in $\cat$ are morphisms in $\cat$, and similarly for $\mathcal D$. 
	Finally, we should have axioms for the symbols $F$ and $J$ guaranteeing that $F$ is functorial and additive, and $J$ is natural and satisfies the hexagon axiom of structure isomorphisms for monoidal functors. 
	We can also require $F$ to be braided if $\cat$ and $\mathcal D$ have braidings. All of these conditions are given by equations and therefore can easily be expressed by first-order sentences. 
	
	It is easy to see that a model $\mathcal M \models \mathcal T_{\otimes\text{-Fun}}$ corresponds to to a pair of \emph{small} (braided) tensor categories and a (braided) tensor functor between them. As a consequence, {\Los} immediately implies the following:
	\begin{proposition}
		Let $\mathscr U$ be an ultrafilter and $\{F_i: \mathcal C_i \to \mathcal D_i\}_{i\in\NN}$ be a sequence of functors.
		The map $F$ given on objects by $F([X_i]) = [F_i(X_i)]$ and on morphisms by $F([f_i]) = [F_i(f_i)]$ is a well-defined functor $$\uprod F_i:\uprod \cat \to \uprod \mathcal D.$$
		If $\mathscr U$-almost every $F_i$ is monoidal with tensor structure $J_i: F_i(-\otimes -) \overset{\sim}{\to} F_i(-)\otimes F_i(-)$, then $J:= \ulim{i} J_i$ is a tensor structure for $F$. 
		If $F_i$ are braided $\mathscr U$-almost always, then so is $F$.	
	\end{proposition}
	
	\subsection*{Ultraproducts as Colimits} Let $\cat$ be a category that has all filtered colimits and arbitrary products. 
	One can give a definition of ultraproducts of objects of $\cat$ as a purely categorical construction as follows. Suppose $\mathscr U$ is an ultrafilter on an index set $I$ and $\{A_i\}_{i\in I}$ is a family of objects in $\cat$. 
	Consider the diagram category $\mathcal D$ whose objects are subsets of $I$ and where there is a unique morphism $J\to K$ whenever $J\supseteq K$. 
	Then, we have a functor $A_\bullet:\mathcal D\to \mathcal C$ sending $J$ to $A_J:= \prod_{i\in J} A_i$ and sending the morphism $J\supseteq K$ to the canonical projection $A_J\to A_K$.
	It is a well-known fact (e.g. see \cite[Section 6.10]{G}) that $\underset{\longrightarrow \mathcal D}{\lim} A_\bullet$ is isomorphic to the ultraproduct $\uprod A_i$ in 1-categories whose objects are models of some first-order theory, e.g. sets, groups, fields, etc. 
	To the author's knowledge, a similar construction in terms of filtered 2-colimits is not known in the 2-categorical setting, e.g. in the 2-category of tensor categories.
	Such a definition could potentially be useful to overcome the restriction to small categories as well as make some arguments more natural.	
	\begin{problem}
		Can we define ultraproducts of ``tensor categories'' as a 2-colimit in an appropriate 2-category so that its universal property is strong enough to replace \Los?
	\end{problem}




\end{document}